\documentclass{elsarticle}
\usepackage{amsmath}

\usepackage{geometry}
\usepackage{subfigure}
\usepackage{setspace}
\usepackage{amssymb}
\usepackage{amsthm}
\usepackage{mathtools}
\usepackage{mathrsfs}
\usepackage{stmaryrd}
\usepackage{amsfonts}
\usepackage{booktabs}
\usepackage{longtable}
\usepackage{multirow}
\usepackage{mathtools}
\usepackage{float}
\usepackage{subfigure}
\usepackage{indentfirst}
\usepackage{epstopdf}
\usepackage{epsfig,amssymb,amsmath,color,graphicx,epsf,pifont,amsmath,CJK}

\usepackage{graphicx} 

\usepackage{listings}
\usepackage{algorithm}
\usepackage{algorithmic}
\algsetup{
  indent=4em,
  linenosize=\small,
  linenodelimiter=.
}
\usepackage{tabularx}
\newcolumntype{L}[1]{>{\raggedright\arraybackslash}p{#1}}
\newcolumntype{C}[1]{>{\centering\arraybackslash}p{#1}}
\newcolumntype{R}[1]{>{\raggedleft\arraybackslash}p{#1}}

\graphicspath{{pic/}}
\allowdisplaybreaks

\usepackage{graphics}
\usepackage[colorlinks,
            linkcolor=red,
            anchorcolor=red,
            citecolor=blue
            ]{hyperref}

\usepackage{tikz}
\usepackage{cleveref}

\newtheorem{theorem}{Theorem}[section] 
\newtheorem{lemma}[theorem]{Lemma}

\newtheorem{remark}[theorem]{Remark}
\newtheorem{proposition}[theorem]{Proposition}
\newtheorem{corollary}[theorem]{Corollary}

\numberwithin{equation}{section}

\begin{document}
\begin{frontmatter}
\title{{\bf High-order accurate positivity-preserving and well-balanced discontinuous Galerkin schemes for ten-moment Gaussian closure equations with source terms}}


\author[address1]{Jiangfu Wang}\ead{2101110034@stu.pku.edu.cn}
\author[address2,address1]{Huazhong Tang}\ead{hztang@math.pku.edu.cn}
\author[address3,address4]{Kailiang Wu\corref{cor1}}\ead{wukl@sustech.edu.cn}
\cortext[cor1]{Corresponding author.}
\address[address1]{Center for Applied Physics and Technology, HEDPS and LMAM, School of Mathematical Sciences, Peking University, Beijing, 100871, PR China.}
\address[address2]{Nanchang Hangkong University, Nanchang, 330000, Jiangxi Province, PR China.}
\address[address3]{Department of Mathematics, SUSTech International Center for Mathematics, and Guangdong Provincial Key Laboratory of Computational Science and Material Design, Southern University of Science and Technology, Shenzhen, Guangdong 518055, PR China}
\address[address4]{National Center for Applied Mathematics Shenzhen (NCAMS), Shenzhen, Guangdong 518055, PR China}


\begin{abstract}

This paper proposes novel high-order accurate discontinuous Galerkin (DG) schemes for the one- and two-dimensional ten-moment Gaussian closure equations with source terms defined by a known potential function. Our DG schemes exhibit the desirable capability of being well-balanced (WB) for a known hydrostatic equilibrium state while simultaneously preserving positive density and positive-definite anisotropic pressure tensor. The well-balancedness is built on carefully modifying the solution states in the Harten--Lax--van Leer--contact (HLLC) flux, and appropriate reformulation and discretization of the source terms. Our novel modification technique overcomes the difficulties posed by the anisotropic effects,  maintains the high-order accuracy, and ensures that the modified solution state remains within the physically admissible state set. Positivity-preserving analyses of our WB DG schemes are conducted, by using several key properties of the admissible state set, the HLLC flux and the HLLC solver, as well as the geometric quasilinearization (GQL) approach in [Wu and Shu, {\em SIAM Review}, 65(4): 1031--1073, 2023], which was originally applied to analyze the admissible state set and the physical-constraints-preserving schemes for the relativistic magnetohydrodynamic equations in [Wu and Tang, {\em M3AS}, 27(10): 1871--1928, 2017], to address the difficulties arising from the nonlinear constraints on the pressure tensor. Moreover, the proposed WB DG schemes satisfy the weak positivity for the cell averages, implying the use of a simple scaling limiter to enforce the physical admissibility of the DG solution polynomials at certain points of interest. Extensive numerical experiments are conducted to validate the preservation of equilibrium states, accuracy in capturing small perturbations to such states, robustness in solving problems involving low density or low pressure, and high resolution for both smooth and discontinuous solutions.
\vspace{2mm}
\end{abstract}

\begin{keyword}
discontinuous Galerkin schemes; well-balanced; positivity-preserving; ten-moment Gaussian closure equations; anisotropic pressure tensor; geometric quasilinearization (GQL).
\end{keyword}

\end{frontmatter}
\section{Introduction}

The Boltzmann equation characterizes the spatio-temporal evolution of the probability density of particles. However, its practical applicability is often limited due to its high-dimensional nature. To simplify the description of a system, one can consider the velocity moments of the probability density function, leading to a reduced number of independent variables governed by a new set of (macroscopic) equations. The compressible Euler equations of gas dynamics are one of such macroscopic systems which can be derived from the Boltzmann equation \cite{levermore1996moment}. This derivation assumes local thermodynamic equilibrium, which results in a scalar pressure. However, in many problems, such as collisionless plasma  \cite{dubroca2004magnetic,morreeuw2006electron,johnson2012ten,wang2018electron,dong2019global} and the non-equilibrium gas dynamics \cite{brown1995numerical}, the local thermodynamic equilibrium assumption does not hold, and anisotropic effects are often present, rendering the Euler equations less suitable. As an alternative, the ten-moment Gaussian closure equations \cite{levermore1998gaussian} provide an effective paradigm for such applications, where the pressure is described by an anisotropic and symmetric tensor.


In the two-dimensional (2D) case, the ten-moment Gaussian closure equations with source terms can be written into the form of balance laws as
\begin{equation}\label{ten-moment2}
\frac{\partial\mathbf{U}}{\partial t}+\frac{\partial\mathbf{F(U)}}{\partial x}+\frac{\partial\mathbf{G(U)}}{\partial y}=\mathbf{S}^x(\mathbf{U})+\mathbf{S}^y(\mathbf{U}),
\end{equation}
where the solution vector $\mathbf{U}=(\rho,m_1,m_2,E_{11},E_{12},E_{22})^\top$ with $\rho$ denoting the density and $\mathbf{m}=(m_1,m_2)^\top$ representing the momentum with $m_i=\rho u_i$ ($i=1,2$). Additionally, $E_{11}$, $E_{12}$, and $E_{22}$ denote components of the symmetric energy tensor $\mathbf{E}$. The system \eqref{ten-moment2} is closed by the equation of state
\begin{equation}\label{EOS}
\mathbf{E}=\frac{1}{2}(\rho\mathbf{u}\otimes\mathbf{u}+\mathbf{p}),
\end{equation}
where $\mathbf{u}=(u_1,u_2)^\top = {\bf m}/\rho$ denotes the velocity,
the symbol $\otimes$ signifies the tensor product, and $\mathbf{p}=(p_{ij})_{1\le i,j \le 2}$ is the symmetric pressure tensor.
The fluxes in \eqref{ten-moment2} are given by
\begin{equation}\label{FU}
\mathbf{F(U)}=\left(\rho u_1,\rho u_1^2+p_{11},\rho u_1u_2+p_{12},(E_{11}+p_{11})u_1,E_{12}u_1+\frac{1}{2}(p_{11}u_2+p_{12}u_1),E_{22}u_1+p_{12}u_2\right)^\top,
\end{equation}
\begin{equation}\label{GU}
\mathbf{G(U)}=\left(\rho u_2,\rho u_1u_2+p_{12},\rho u_2^2+p_{22},E_{11}u_2+p_{12}u_1,E_{12}u_2+\frac{1}{2}(p_{12}u_2+p_{22}u_1),(E_{22}+p_{22})u_2\right)^\top.
\end{equation}
In this paper, the following source terms are considered:
\begin{equation*}\label{sxU}
\mathbf{S}^x(\mathbf{U})=\left(0,-\frac{1}{2}\rho\partial_xW,0,-\frac{1}{2}\rho u_1\partial_xW,-\frac{1}{4}\rho u_2\partial_xW,0\right)^\top,
\end{equation*}
\begin{equation*}
\mathbf{S}^y(\mathbf{U})=\left(0,0,-\frac{1}{2}\rho\partial_yW,0,-\frac{1}{4}\rho u_1\partial_yW,-\frac{1}{2}\rho u_2\partial_yW\right)^\top,
\end{equation*}
where the function $W(x,y,t)$ represents a known potential, which can denote the electron quiver energy in laser light (see, e.g.~\cite{morreeuw2006electron,sangam2007anisotropic}).
In physics, the density must be positive, and the symmetric pressure tensor must be positive-definite. This means $\mathbf{U}$ should stay in the physically admissible state set
\begin{equation}
\mathcal{G}:=\left\{\mathbf{U}\in\mathbb{R}^6: \rho>0,~ \mathbf{x}^\top\mathbf{p}\mathbf{x}>0 ~ \forall\mathbf{x}\in\mathbb{R}^2 \setminus \{\mathbf{0}\} \right\},
\end{equation}
or
\begin{equation}
\mathcal{G}:=\{\mathbf{U}\in\mathbb{R}^6: \rho>0,~ p_{11}>0,~ \det(\mathbf{p})>0\},
\end{equation}
where $\det(\mathbf{p}):=p_{11}p_{22}-p_{12}^2$. Based on \eqref{EOS}, it can be observed that $\mathcal{G}$ is equivalent to
\begin{equation}\label{Gast}
\mathcal{G}_1:=\left\{\mathbf{U}\in\mathbb{R}^6: \rho>0,~ g_{11}(\mathbf{U})>0,~ g_{\det}(\mathbf{U})>0\right\},
\end{equation}
where
\[
g_{11}(\mathbf{U}):=2E_{11}-\frac{m_1^2}{\rho}, \quad g_{\det}(\mathbf{U}):=\left(2E_{11}-\frac{m_1^2}{\rho}\right)\left(2E_{22}-\frac{m_2^2}{\rho}\right)-\left(2E_{12}-\frac{m_1m_2}{\rho}\right)^2.
\]

In recent decades, several numerical schemes have been proposed for solving the ten-moment equations. Brown et al. introduced a second-order upwind finite volume scheme \cite{brown1995numerical}. Berthon et al. employed relaxation numerical schemes to approximate the weak solutions of these equations \cite{berthon2006numerical, berthon2015entropy}. Notably, their relaxation schemes are first-order accurate and ensure both entropy stability and the preservation of positivity.
Meena, Kumar, and their collaborators have developed a variety of numerical schemes for ten-moment equations, including high-order positivity-preserving discontinuous Galerkin (DG) methods \cite{meena2017positivity}, finite difference weighted essentially non-oscillatory (WENO) schemes \cite{meena2020positivity}, and high-order entropy stable finite difference methods \cite{sen2018entropy} as well as DG methods \cite{biswas2021entropy}. Besides, they have formulated a second-order robust monotone upwind scheme \cite{meena2017robust}, and a second-order well-balanced (WB) scheme to handle equilibrium states \cite{meena2018well}. Additionally, Meena and  Kumar proposed a robust finite volume scheme for the two-fluid ten-moment plasma flow equations \cite{meena2019robust}. Sangam applied a Harten--Lax--van Leer-contact (HLLC) approximate Riemann solver to solve the ten-moment equations coupled with magnetic field \cite{sangam2008hllc}.


The hyperbolic balance laws (\ref{ten-moment2}), coupled with (\ref{EOS}), exhibit non-trivial hydrostatic equilibrium solutions \cite{meena2018well}, where the source term and the flux gradient mutually balance each other.
Conventional numerical methods may fail to maintain such hydrostatic equilibrium solutions, potentially leading to notable numerical error on coarser meshes when simulating these solutions or their perturbations.
To tackle this issue, one has to conduct simulations on highly refined meshes, which can be time-consuming, especially for multi-dimensional problems or long-term simulations. In order to reduce computational costs, WB methods have been proposed to accurately preserve a discrete form of these hydrostatic equilibrium solutions up to machine accuracy. This ensures the effective capture of nearly equilibrium flows even on coarse meshes.
Existing WB methods primarily focused on the shallow water equations over a non-flat bottom topology (see, e.g. \cite{greenberg1996well,audusse2004fast,xing2005high,xing2010positivity,zhang2023high}) and the Euler equations under gravitational fields (see, e.g. \cite{xing2013high,kappeli2014well,chandrashekar2015second,li2016high,klingenberg2019arbitrary,grosheintz2019high,wu2021uniformly,ren2023high}).
Besides, there exist a few WB schemes for the magnetohydrodynamic (MHD) equations with gravitational source terms \cite{kanbar2022well}, the shallow water MHD equations \cite{bouchut2017multi}, and blood flow 
 \cite{britton2020well,hernandez2021well,pimentel2023high}. 
For the ten-moment Gaussian closure equations (\ref{ten-moment2}), Meena and Kumar presented a WB second-order finite volume scheme by combining hydrostatic reconstruction with contact-preserving numerical flux and appropriate source discretization \cite{meena2018well}.

In addition to maintaining the hydrostatic equilibrium states, another crucial requirement for numerical schemes of the system (\ref{ten-moment2}) is to preserve positive density and  positive-definite pressure tensor, referred to as the positivity-preserving property in this paper.
Ensuring positivity is not only vital for maintaining physically meaningful solutions but is also essential for the stability of numerical simulations. Over the past few decades, high-order positivity-preserving, or bound-preserving in general, numerical schemes have garnered widespread attention and made significant progress. Most of these studies are rooted in two types of limiters: a local scaling limiter  (see, e.g.  \cite{zhang2010maximum,zhang2010positivity,zhang2012maximum,zhang2017positivity}) or a flux-correction limiter (see, e.g.  \cite{xu2014parametrized,hu2013positivity,wu2015high}). 
Recently, the geometric quasilinearization (GQL) approach was developed in \cite{wu2021geometric} to address general bound-preserving problems involving nonlinear constraints. The GQL approach was motivated by the bound-preserving study on the compressible MHD systems \cite{wu2017admissible,wu2018positivity,wu2019provably,wu2023provably,DingWu2023} and relativistic hydrodynamic equations \cite{wu2021minimum,wu2021provably}.
The GQL approach equivalently transforms nonlinear constraints into linear ones by introducing free auxiliary variables, thereby highly simplifying the bound-preserving analysis. Furthermore, there have been efforts to develop high-order numerical schemes that simultaneously possess WB and positivity-preserving properties for various hyperbolic systems, such as the shallow water equations (see, e.g. \cite{xing2010positivity,xing2013positivity,li2017positivity,zhang2021positivity,zhang2023high}), the Euler equations under gravity (see, e.g. \cite{wu2021uniformly,zhang2021high,jiang2022positivity,ren2023high,gu2023high}), and blood flow \cite{hernandez2021well}. However, to the best of our knowledge, there are currently no studies on positivity-preserving WB schemes for the ten-moment equations (\ref{ten-moment2}), which pose some essential difficulties not encountered in other systems due to the anisotropic effects.

This paper aims to develop high-order WB DG schemes that are provably positivity-preserving for the ten-moment Gaussian closure system with source terms. To the best of our knowledge, the sole existing WB method for this system was devised by Meena and Kumar in \cite{meena2018well}. However, it was confined to second-order accuracy, and its positivity-preserving property has not been investigated yet.
The efforts, novelty, and difficulties in this work are summarized as follows:
\begin{itemize}
\item
The proposed novel DG schemes incorporate an appropriate discretization of the source terms and suitable modification to the solution states in the HLLC numerical fluxes, which are carefully devised such that
 the desired WB and positivity-preserving properties are satisfied simultaneously.
Our modification is novel and different from the one proposed in \cite{wu2021uniformly} for the Euler equations with gravitation.
The key distinction lies in the fact that for the Euler system with gravitation, the pressure is scalar, whereas for the ten-moment system, the pressure is an anisotropic tensor.
Our novel modification technique overcomes the difficulties posed by the anisotropic effects,  maintains the high-order accuracy, and 
ensures that the modified solution state remains within the physically admissible state set.

\item We introduce several important properties of the admissible state set, the HLLC flux and the HLLC solver.
Based on these properties and the GQL approach, we prove that the resulting WB DG schemes possess a weak positivity for the updated cell averages of the DG solutions.
This property ensures that a simple scaling limiter  \cite{zhang2010positivity,wang2012robust,meena2017positivity} can effectively enforce the physical admissibility of the DG solution polynomials at certain points of interest without compromising high-order accuracy and conservation.
The modification to the solution states in the HLLC numerical flux and special discretization of the source terms make the positivity-preserving analyses more difficult than the analyses for the standard DG schemes without any WB  modifications \cite{meena2017positivity}.
Furthermore, due to the new technique for handling the anisotropic pressure tensor, the positivity-preserving design and analyses for the ten-moment equations are more intricate than those in \cite{wu2021uniformly} for the Euler equations with gravitation, where the pressure is a scalar.

\end{itemize}

This paper is organized as follows. Section \ref{sec2} will introduce the hydrostatic equilibrium solutions of (\ref{ten-moment2}) and present several important properties (including the GQL respresentation) of the admissible state set, the HLLC flux and the HLLC solver.   Section \ref{sec3}  focuses on the development of  positivity-preserving WB DG schemes for the one-dimensional (1D) ten-moment system. Section \ref{sec4} then proceeds to generalize these schemes to the 2D case. Before concluding the paper in Section \ref{sec6}, Section \ref{sec5} will present the numerical tests to validate the desired properties and robustness of the proposed positivity-preserving WB DG schemes.

\section{Auxiliary results}\label{sec2}
This section introduces the hydrostatic equilibrium solutions of (\ref{ten-moment2}), several useful properties (including the GQL representation) of the admissible state set, the contact property of the HLLC flux and the positivity-preserving property of the HLLC solver.

\subsection{Hydrostatic equilibrium solutions}\label{stationary solutions}
The steady state solutions to the one-dimensional (1D) ten-moment model satisfy
\[
\frac{\partial\mathbf{F(U)}}{\partial x}=\mathbf{S}^x(\mathbf{U}).
\]
Denote the steady state solutions by $\left\{\rho^e,u_1^e,u_2^e,p_{11}^e,p_{12}^e,p_{22}^e\right\}$.
Due to the conditions $u_1^e=u_2^e=0$ at a hydrostatic state, the following relationships hold:
\begin{align}
(p_{11}^e)_x&=-\frac{1}{2}\rho^e W_x, \label{1d steady state 1} \\
(p_{12}^e)_x&=0. \notag
\end{align}
Consequently, $p_{12}^e$ is a constant at a hydrostatic equilibrium state. However, determining $p_{11}^e$ from (\ref{1d steady state 1}) requires an additional relation. Three cases \cite{meena2018well} are considered:
\begin{itemize}
\item Polytropic case: A polytropic equilibrium is characterized by $p_{11}^e=\alpha(\rho^e)^\nu$ with constants $\alpha$ and $\nu$ ($\nu>1$). This relation, along with (\ref{1d steady state 1}), yields
    \begin{equation*}
    \rho^e(x)=\left[\frac{\nu-1}{\alpha\nu}\left(C-\frac{1}{2}W(x)\right)\right]^{\frac{1}{\nu-1}}, \quad u_1^e=u_2^e=0, \quad p_{11}^e(x)=\frac{1}{\alpha^{\frac{1}{\nu-1}}}\left[\frac{\nu-1}{\nu}\left(C-\frac{1}{2}W(x)\right)\right]^{\frac{\nu}{\nu-1}},
    \end{equation*}
    where $C$ is a constant.

\item Isentropic case: The polytropic process behaves as an isentropic process for $\nu=3$.

\item Isothermal case: Assuming the ideal gas law, i.e., $p_{11}=\rho RT_{11}$ with constant temperature $T_{11}=T_{11,0}$ and ideal gas constant $R$, one has
    \begin{equation*}
      \rho^e(x)=\rho_0\exp\left(\frac{-W(x)}{2RT_{11,0}}\right), \quad u_1^e=u_2^e=0, \quad
      p_{11}^e(x)=p_{11,0}\exp\left(\frac{-W(x)}{2RT_{11,0}}\right),
    \end{equation*}
    where both $\rho_0$ and $p_{11,0}=\rho_0RT_{11,0}$ are constants.
\end{itemize}

At hydrostatic equilibrium states, it reduces to
\[
(p_{11}^e)_x+(p_{12}^e)_y=-\frac{1}{2}\rho W_x, \quad (p_{12}^e)_x+(p_{22}^e)_y=-\frac{1}{2}\rho W_y.
\]
These conditions are not enough to determine the steady solutions. Similar to the 1D case, one can assume that the pressure component $p_{12}^e$ is a constant, yielding that
\begin{align*}
(p_{11}^e)_x=-\frac{1}{2}\rho W_x, \qquad
(p_{22}^e)_y=-\frac{1}{2}\rho W_y.
\end{align*}
The hydrostatic equilibrium solutions can then be obtained with a special relation between these quantities, akin to the 1D case. We omit the details of three special equilibriums for the 2D case and refer readers to \cite{meena2018well} for a detailed derivation.

\subsection{Properties of admissible states}
This subsection gives several useful properties of the admissible state set $\mathcal{G}$ and its GQL representation,  which are crucial for our positivity-preserving analyses.

\begin{lemma}[{\rm\cite{meena2017positivity}}]\label{scale invariance}
The set $\mathcal{G}$ is convex. Moreover, if $\mathbf{U}\in\mathcal{G}$, then $\lambda\mathbf{U}\in\mathcal{G}$ for any $\lambda>0$.
\end{lemma}

Denote the closure of $\mathcal{G}$ by $\overline{\mathcal{G}}:=\left\{\mathbf{U}\in\mathbb{R}^6: \rho\geq0,~ \mathbf{x}^\top\mathbf{p}\mathbf{x}\geq0 ~ \forall\mathbf{x}\in\mathbb{R}^2 \setminus \{\mathbf{0}\} \right\}$. Then we have the following lemma.

\begin{lemma}\label{general combination}
	For any $\lambda_0>0$, $\lambda_1\geq0$, $\mathbf{U}_0\in\mathcal{G}$, and $\mathbf{U}_1\in\overline{\mathcal{G}}$, it holds that $(\lambda_0\mathbf{U}_0+\lambda_1\mathbf{U}_1)\in\mathcal{G}$.
\end{lemma}

Lemmas \ref{scale invariance}--\ref{general combination} can be directly proven based on the definitions of $\mathcal{G}$ and $\overline{\mathcal{G}}$. Additionally, we have the following GQL representation \cite{wu2021geometric} of $\mathcal{G}$, which will be useful for simplifying the positivity-preserving analyses of the high-order DG schemes.
The remarkable advantages of the GQL techniques lie in that it equivalently transforms the intractable nonlinear constraints in $\mathcal{G}$ into linear constraints as detailed in \eqref{eq:GQL}.

\begin{lemma}[GQL representation {\rm \cite{wu2021geometric}}]\label{GQL}
The admissible state set $\mathcal{G}$ is equivalent to
\begin{equation}\label{eq:GQL}
	\mathcal{G}_{*}:=\left\{\mathbf{U}\in\mathbb{R}^6: \mathbf{U}\cdot\mathbf{e}_1>0, \varphi(\mathbf{U};\mathbf{z},\mathbf{u}_{*})>0 ~ \forall \mathbf{u}_{*}\in\mathbb{R}^2, ~ \forall \mathbf{z}\in\mathbb{R}^2\backslash\{\mathbf{0}\}\right\},
\end{equation}
where $\mathbf{e}_1:=(1,0,\cdots,0)^\top$ and the function $\varphi(\mathbf{U};\mathbf{z},\mathbf{u}_{*})$ is defined as
\[
\varphi(\mathbf{U};\mathbf{z},\mathbf{u}_{*}):=\mathbf{z}^\top\left(\mathbf{E}-\mathbf{m}\otimes\mathbf{u}_{*}+\rho\frac{\mathbf{u}_{*}\otimes\mathbf{u}_{*}}{2}\right)\mathbf{z},
\]
which is linear with respect to $\mathbf{U}$.
\end{lemma}

The proof of Lemma \ref{GQL} can be found in \cite{wu2021geometric}. Moreover, according to the definition of $\varphi(\mathbf{U};\mathbf{z},\mathbf{u}_{*})$,  $\varphi(\kappa(\mathbf{U})\mathbf{U};\mathbf{z},\mathbf{u}_{*})=\kappa(\mathbf{U})\varphi(\mathbf{U};\mathbf{z},\mathbf{u}_{*})$, where $\kappa(\mathbf{U})$ is a quantity related to $\mathbf{U}$.

\begin{lemma}\label{Uhat in G}
For any $\mathbf{U}\in\mathcal{G}$, define $\widehat{\mathbf{U}}$ as
\[
\widehat{\rho}=\rho, \quad \widehat{\mathbf{m}}=\mathbf{Qm}, \quad \widehat{\mathbf{E}}=\mathbf{QEQ}^\top,
\]
where $\mathbf{Q}\in\mathbb{R}^{2\times2}$ is non-singular, then $\widehat{\mathbf{U}}\in\mathcal{G}$.
\end{lemma}
\begin{proof}
The conclusion follows from  $\widehat{\mathbf{p}}=2\widehat{\mathbf{E}}-\frac{\widehat{\mathbf{m}}\widehat{\mathbf{m}}^\top}{\widehat{\rho}}=\mathbf{QpQ}^\top$.
\end{proof}

\begin{lemma}\label{U S positivity}
For any $\beta>0$ and $\mathbf{U}=(\rho,m_1,m_2,E_{11},E_{12},E_{22})^\top\in\mathcal{G}$, if   $|\widetilde{\delta}_1|\leq\min\left\{\sqrt{\frac{p_{11}}{\rho}},\sqrt{\frac{\det(\mathbf{p})}{\rho p_{22}}}\right\}\beta$ and $|\widetilde{\delta}_2|\leq\min\left\{\sqrt{\frac{p_{22}}{\rho}},\sqrt{\frac{\det(\mathbf{p})}{\rho p_{11}}}\right\}\beta$, then
\[
\widetilde{\mathbf{U}}_1:=\beta\mathbf{U}+\widetilde{\delta}_1\left(0,\rho,0,m_1,\frac{1}{2}m_2,0\right)^\top\in\overline{\mathcal{G}},
\]
\[
\widetilde{\mathbf{U}}_2:=\beta\mathbf{U}+\widetilde{\delta}_2\left(0,0,\rho,0,\frac{1}{2}m_1,m_2\right)^\top\in\overline{\mathcal{G}}.
\]
\end{lemma}
\begin{proof}
For $\beta>0$, one has $\widetilde{\rho}_1=\beta\rho>0$, and
\[
\widetilde{\mathbf{p}}_1=2\widetilde{\mathbf{E}}_1-\frac{\widetilde{\mathbf{m}}_1\widetilde{\mathbf{m}}_1^\top}{\widetilde{\rho}_1}
=\begin{pmatrix}
   \beta p_{11}-\frac{\widetilde{\delta}_1^2}{\beta}\rho & \beta p_{12} \\
   \beta p_{12} & \beta p_{22}
 \end{pmatrix}.
\]
For $\mathbf{U}\in\mathcal{G}$, if $|\widetilde{\delta}_1|\leq\min\left\{\sqrt{\frac{p_{11}}{\rho}},\sqrt{\frac{\det(\mathbf{p})}{\rho p_{22}}}\right\}\beta$, then
\[
\begin{cases}
  \beta p_{11}-\frac{\widetilde{\delta}_1^2}{\beta}\rho\geq0, &  \\
  \beta p_{22}>0, & \\
  \left(\beta p_{11}-\frac{\widetilde{\delta}_1^2}{\beta}\rho\right)\beta p_{22}-\beta^2p_{12}^2\geq0, &
\end{cases}
\]
which indicates that $\widetilde{\mathbf{p}}_1$ is positive-semidefinite. It follows that $\widetilde{\mathbf{U}}_1\in\overline{\mathcal{G}}$. Similar arguments imply   $\widetilde{\mathbf{U}}_2\in\overline{\mathcal{G}}$.
\end{proof}

The GQL representation (see Lemma \ref{GQL}) and Lemma \ref{U S positivity} with $\varphi(\kappa(\mathbf{U})\mathbf{U};\mathbf{z},\mathbf{u}_{*})=\kappa(\mathbf{U})\varphi(\mathbf{U};\mathbf{z},\mathbf{u}_{*})$ give the following corollary. It will play an important role in the positivity-preserving analyses of the WB schemes for the nonhomogeneous case; see Subsection \ref{pp analysis for 1d first-order schemes}, \ref{pp analysis of 1d high-order schemes}, \ref{pp analysis of 2d first-order schemes}, and \ref{pp analysis of 2d high-order schemes}.

\begin{corollary}\label{GQL corollary}
For any $\mathbf{U}\in\mathcal{G}$, one has
\begin{align*}
\left|\varphi(\mathbf{S}_1;\mathbf{z},\mathbf{u}_{*})\right|&\leq\delta_1(\mathbf{U})\varphi(\mathbf{U};\mathbf{z},\mathbf{u}_{*}),
\\
\left|\varphi(\mathbf{S}_2;\mathbf{z},\mathbf{u}_{*})\right|&\leq\delta_2(\mathbf{U})\varphi(\mathbf{U};\mathbf{z},\mathbf{u}_{*}) \quad \forall \mathbf{u}_{*}\in\mathbb{R}^2, ~ \forall \mathbf{z}\in\mathbb{R}^2\backslash\{\mathbf{0}\},
\end{align*}
where
\[
\mathbf{S}_1:=\left(0,\rho,0,m_1,\frac{1}{2}m_2,0\right)^\top, \quad
\delta_1(\mathbf{U}):=\max\left\{\sqrt{\frac{\rho}{p_{11}}},\sqrt{\frac{\rho p_{22}}{\det(\mathbf{p})}}\right\},
\]
\[
\mathbf{S}_2:=\left(0,0,\rho,0,\frac{1}{2}m_1,m_2\right)^\top, \quad
\delta_2(\mathbf{U}):={\max\left\{\sqrt{\frac{\rho}{p_{22}}},\sqrt{\frac{\rho p_{11}}{\det(\mathbf{p})}}\right\}}.
\]
\end{corollary}

\subsection{Properties of HLLC flux and HLLC solver}
This subsection presents two important properties: the contact property of the HLLC flux \cite{meena2018well} and the positivity-preserving property of the HLLC solver \cite{meena2017robust}. The readers are referred to \cite{meena2017robust,meena2017positivity} for the details of the HLLC flux of the ten-moment Gaussian closure equations.

\begin{lemma}[{\rm\cite{meena2018well}}]\label{contact property}
For any two states $\mathbf{U}_L=\left(\rho_L,0,0,\frac{p_{11}}{2},\frac{p_{12}}{2},\frac{p_{22,L}}{2}\right)^\top$ and $\mathbf{U}_R=\left(\rho_R,0,0,\frac{p_{11}}{2},\frac{p_{12}}{2},\frac{p_{22,R}}{2}\right)^\top$, the HLLC flux in the x-direction satisfies
\[
\mathbf{F}^{hllc}(\mathbf{U}_L,\mathbf{U}_R)=(0,p_{11},p_{12},0,0,0)^\top.
\]

\end{lemma}

\begin{lemma}[{\rm\cite{meena2017robust}}]\label{pp property}
For any three admissible states $\mathbf{U}_L, \mathbf{U}_M, \mathbf{U}_R\in\mathcal{G}$, and any $\lambda_1>0$ satisfying
\[
\lambda_1\max_{\mathbf{U}\in\{\mathbf{U}_L, \mathbf{U}_M, \mathbf{U}_R\}}\alpha_{1}(\mathbf{U})\leq\frac{1}{2},
\]
where $\alpha_{1}(\mathbf{U}):=|u_1|+\sqrt{\frac{3p_{11}}{\rho}}$, one has
\[
\mathbf{U}_M-\lambda_1\left(\mathbf{F}^{hllc}(\mathbf{U}_M,\mathbf{U}_R)-\mathbf{F}^{hllc}(\mathbf{U}_L,\mathbf{U}_M)\right)\in\mathcal{G}.
\]
\end{lemma}

\subsection{A property of positive-definite matrices}
\begin{lemma}\label{transfor T}
For three symmetric positive-definite $2\times2$ matrices
\[
\mathbf{A}=\begin{pmatrix}
    a_{11} & a_{12} \\
    a_{12} & a_{22}
  \end{pmatrix},
\quad
\mathbf{B}=\begin{pmatrix}
    b_{11} & b_{12} \\
    b_{12} & b_{22}
  \end{pmatrix},
\quad
\mathbf{C}=\begin{pmatrix}
    c_{11} & c_{12} \\
    c_{12} & c_{22}
  \end{pmatrix},
\]
which satisfy $a_{22}=b_{22}$ and $c_{11}=a_{11}$, there exist two $2\times2$ matrices in the following forms
\begin{align}
\mathbf{T}_1=\begin{pmatrix}
    t_1 & t_2 \\
    0 & 1
  \end{pmatrix}, \label{upper T}
\\
\mathbf{T}_2=\begin{pmatrix}
    1 & 0 \\
    t_4 & t_3
  \end{pmatrix}, \label{lower T}
\end{align}
such that $\mathbf{B}=\mathbf{T}_1\mathbf{AT}_1^\top$ and $\mathbf{C}=\mathbf{T}_2\mathbf{AT}_2^\top$.
\end{lemma}
\begin{proof}
Taking
\[
t_1=\sqrt{\frac{\det(\mathbf{B})}{\det(\mathbf{A})}}, \quad t_2=\frac{b_{12}-a_{12}t_1}{a_{22}},
\]
and
\[
t_3=\sqrt{\frac{\det(\mathbf{C})}{\det(\mathbf{A})}}, \quad t_4=\frac{c_{12}-a_{12}t_3}{a_{11}},
\]
completes the proof.
\end{proof}

\section{Positivity-preserving WB DG methods in one dimension}\label{sec3}
Assume that the target hydrostatic equilibrium solutions to be preserved are explicitly known. As discussed in Subsection \ref{stationary solutions}, the following relations hold:
\begin{equation}\label{1D steady state}
(p_{11}^e(x))_x=-\frac{1}{2}\rho^e(x)W_x, \quad p_{12}^e(x)=C_0, \quad u_1^e(x)=u_2^e(x)=0,
\end{equation}
where $C_0$ is a constant.

\subsection{WB DG discretization}
Consider a spatial domain $\Omega$ partitioned into cells ${I_j=(x_{j-\frac{1}{2}}, x_{j+\frac{1}{2}})}$. Let the mesh size be denoted by $h_j = x_{j+\frac{1}{2}} - x_{j-\frac{1}{2}}$, with $h$ representing the maximum mesh size, $h=\max_j h_j$. Define the center of each cell as $x_j=\frac{1}{2}(x_{j+\frac{1}{2}}+x_{j-\frac{1}{2}})$. The DG numerical solutions are denoted by $\mathbf{U}_h(x,t)$,  each component of which belongs to the finite-dimensional space of discontinuous piecewise polynomial functions:
\[
\mathbb{V}_h^k=\left\{v(x)\in L^2(\Omega):v(x)|_{I_j}\in\mathbb{P}^k(I_j) ~ \forall j\right\},
\]
where $\mathbb{P}^k(I_j)$ represents the space of polynomials of degree up to $k$ in cell $I_j$. The semi-discrete DG methods state that for any test function $v\in\mathbb{V}_h^k$, the solution $\mathbf{U}_h$ is determined by
\begin{equation}\label{1D semi DG}
\int_{I_j}(\mathbf{U}_h)_t v \, dx - \int_{I_j}\mathbf{F}(\mathbf{U}_h) v_x \, dx + \widehat{\mathbf{F}}_{j+\frac{1}{2}} v(x_{j+\frac{1}{2}}^{-}) - \widehat{\mathbf{F}}_{j-\frac{1}{2}} v(x_{j-\frac{1}{2}}^{+}) = \int_{I_j}\mathbf{S}^x(\mathbf{U}_h) v \, dx,
\end{equation}
where $\widehat{\mathbf{F}}_{j+\frac{1}{2}}$ denotes the numerical flux at $x_{j+\frac{1}{2}}$. The notations $x_{j+\frac{1}{2}}^{-}$ and $x_{j+\frac{1}{2}}^{+}$ refer to the left and right limits at $x_{j+\frac{1}{2}}$, respectively.

It is worth noting that if the standard Gauss quadrature is used to discretize the cell integral at both sides of (\ref{1D semi DG}) and a conventional numerical flux is applied without any modification, the resulting DG scheme is not WB in general. For example, the DG scheme with the standard HLLC flux \cite{meena2017positivity} is not WB as it will be demonstrated by the numerical experiments in Section \ref{sec5}.


In the following, we construct a WB DG method that preserves the equilibrium state (\ref{1D steady state}). Let $\rho_h^e(x)$ and
\[
\mathbf{p}_h^e(x)=\begin{pmatrix}
                    p_{11,h}^e(x) & C_0 \\
                    C_0 & p_{22,h}^e(x)
                  \end{pmatrix}
\]
denote the positive (positive-definite) projections of $\rho^e(x)$ and 
\[
\mathbf{p}^e(x)=\begin{pmatrix}
                    p_{11}^e(x) & C_0 \\
                    C_0 & p_{22}^e(x)
                  \end{pmatrix}
\]
onto the space $\mathbb{V}_h^k$, respectively.
Define
\begin{equation}\label{pj1/2pm}
\mathbf{p}_{j+\frac{1}{2}}^{e,\pm,\ast}:=\begin{pmatrix}
                                       p_{11,j+\frac{1}{2}}^{e,\ast} & C_0 \\
                                       C_0 & p_{22,h}^e(x_{j+\frac{1}{2}}^{\pm})
                                     \end{pmatrix},
\end{equation}
where
\[
p_{11,j+\frac{1}{2}}^{e,\ast}:=\max\left\{p_{11,h}^e(x_{j+\frac{1}{2}}^{-}),p_{11,h}^e(x_{j+\frac{1}{2}}^{+})\right\}.
\]
As a critical observation, we note that $\mathbf{p}_{j+\frac{1}{2}}^{e,\pm,\ast}$ is always positive-definite.
According to Lemma \ref{transfor T}, there exist two upper triangular matrices $\mathbf{T}_{j+\frac{1}{2}}^{\pm}$ in the forms of (\ref{upper T}) satisfying
\begin{equation}\label{ph-pj1/2pm}
\mathbf{p}_{j+\frac{1}{2}}^{e,\pm,\ast}=\mathbf{T}_{j+\frac{1}{2}}^{\pm}\mathbf{p}_h^e(x_{j+\frac{1}{2}}^{\pm})(\mathbf{T}_{j+\frac{1}{2}}^{\pm})^\top.
\end{equation}

\begin{remark}\label{t1t2_high_order}
If we define
\[
\mathbf{T}_{j+\frac{1}{2}}^{\pm}=\begin{pmatrix}
                                   t_1^{\pm} & t_2^{\pm} \\
                                   0 & 1
                                 \end{pmatrix},
\]
then Lemma {\rm \ref{transfor T}} implies that
\[
t_1^{\pm}=\sqrt{\frac{\det\left(\mathbf{p}_{j+\frac{1}{2}}^{e,\pm,\ast}\right)}{\det\left(\mathbf{p}_h^e(x_{j+\frac{1}{2}}^{\pm})\right)}}, \quad
t_2^{\pm}=\frac{C_0(1-t_1^{\pm})}{p_{22,h}^e(x_{j+\frac{1}{2}}^{\pm})}.
\]
Because $p_{11,h}^e(x_{j+\frac{1}{2}}^{\pm})$ and $p_{22,h}^e(x_{j+\frac{1}{2}}^{\pm})$ are $(k+1)$-order approximation to $p_{11,h}^e(x_{j+\frac{1}{2}})$ and $p_{22,h}^e(x_{j+\frac{1}{2}})$, respectively, one can prove that $t_1^{\pm}=1+O(h^{k+1})$ and $t_2^{\pm}=O(h^{k+1})$, which will be verified by our numerical experiments.
\end{remark}

To make the DG method (\ref{1D semi DG}) WB, we apply the following HLLC numerical flux with modified solution states:
\begin{equation}\label{modified HLLC flux}
\widehat{\mathbf{F}}_{j+\frac{1}{2}}=\mathbf{F}^{hllc}(\widehat{\mathbf{U}}_{j+\frac{1}{2}}^{-},\widehat{\mathbf{U}}_{j+\frac{1}{2}}^{+}),
\end{equation}
where the modified solution states $\widehat{\mathbf{U}}_{j+\frac{1}{2}}^{\pm}$ are related to $\mathbf{U}_{j+\frac{1}{2}}^{\pm}{:=\mathbf{U}_h(x_{j+\frac{1}{2}}^{\pm})}$ by
\begin{equation}\label{U-Uhat}
\widehat{\rho}_{j+\frac{1}{2}}^{\pm}=\rho_{j+\frac{1}{2}}^{\pm}, \quad \widehat{\mathbf{m}}_{j+\frac{1}{2}}^{\pm}=\mathbf{T}_{j+\frac{1}{2}}^{\pm}\mathbf{m}_{j+\frac{1}{2}}^{\pm}, \quad
\widehat{\mathbf{E}}_{j+\frac{1}{2}}^{\pm}=\mathbf{T}_{j+\frac{1}{2}}^{\pm}\mathbf{E}_{j+\frac{1}{2}}^{\pm}(\mathbf{T}_{j+\frac{1}{2}}^{\pm})^\top.
\end{equation}
According to Lemma \ref{Uhat in G}, if $\mathbf{U}_{j+\frac{1}{2}}^{\pm}\in\mathcal{G}$, then $\widehat{\mathbf{U}}_{j+\frac{1}{2}}^{\pm}\in\mathcal{G}$.
Denote the $N$-point Gauss quadrature nodes and weights in $I_j$ by $\{x_j^{(\mu)},\omega_\mu\}_{1\leq\mu\leq N}$. The cell integral of flux in (\ref{1D semi DG}) is approximated by
\begin{equation}\label{1D cell integral}
\int_{I_j}\mathbf{F}(\mathbf{U}_h)v_xdx\approx h_j\sum_{\mu=1}^N\omega_\mu\mathbf{F}(\mathbf{U}_h(x_j^{(\mu)}))v_x(x_j^{(\mu)}).
\end{equation}

\begin{remark}
The novelty of the modification {\rm (\ref{U-Uhat})} lies in three aspects. (i) It maintains the contact property of the HLLC flux when the numerical solutions achieve the hydrostatic state, which is crucial for achieving well-balancedness; see the proof of Theorem {\rm\ref{1D well-balanceness}}. (ii) It does not destroy the high-order accuracy; see Remark {\rm \ref{t1t2_high_order}}. (iii) The modified states $\widehat{\mathbf{U}}_{j+\frac{1}{2}}^{\pm}$ remain in the admissible state set $\mathcal{G}$ if $\mathbf{U}_{j+\frac{1}{2}}^{\pm}\in\mathcal{G}$; see Lemma {\rm \ref{Uhat in G}}. Note that the existing modification techniques in {\rm \cite{wu2021uniformly,ren2023high}} for the Euler equations in the isotropic case do not meet the above three requirements simultaneously due to the anisotropic effects in the ten-moment system.
\end{remark}

\begin{remark}
The HLLC flux with modified solution states \eqref{modified HLLC flux} is not rigorously consistent, since $\mathbf{U}_{j+\frac{1}{2}}^{+}=\mathbf{U}_{j+\frac{1}{2}}^{-}$ can not imply that $\widehat{\mathbf{U}}_{j+\frac{1}{2}}^{+}=\widehat{\mathbf{U}}_{j+\frac{1}{2}}^{-}$ due to $\mathbf{T}_{j+\frac{1}{2}}^{+}\neq\mathbf{T}_{j+\frac{1}{2}}^{-}$. However, because $\mathbf{T}_{j+\frac{1}{2}}^{\pm}$ are both the approximation to the identity matrix of order $k+1$, the numerical flux \eqref{modified HLLC flux} is consistent with $(k+1)$-th order accuracy.
\end{remark}

Next, we turn our attention to the discretization of the integrals of the source terms in (\ref{1D semi DG}) to ensure the WB property. Let
$\mathbf{S}^x{=:}(0,S^{[2]},0,S^{[4]},S^{[5]},0)^\top$. Motivated by the techniques from \cite{xing2013high,li2016well,li2016high,wu2021uniformly} for the Euler equations, we reformulate and decompose the integral of the second component of the source terms as follows:
\begin{align}\label{S2 integral}
\int_{I_j} S^{[2]} v \, dx &= \int_{I_j} -\frac{1}{2} \rho W_x v \, dx = \int_{I_j} \frac{\rho}{\rho^e} (p_{11}^e)_x v \, dx \notag \\
&= \int_{I_j} \left( \frac{\rho}{\rho^e} - \frac{\overline{\rho}_j}{{\overline{\rho^e}_j}} + \frac{\overline{\rho}_j}{\overline{\rho^e}_j} \right) (p_{11}^e)_x v \, dx \notag \\
&= \int_{I_j} \left( \frac{\rho}{\rho^e} - \frac{\overline{\rho}_j}{\overline{\rho^e}_j} \right) (p_{11}^e)_x v \, dx \notag \\
&\quad + \frac{\overline{\rho}_j}{\overline{\rho^e}_j} \left( p_{11}^e(x_{j+\frac{1}{2}}^{-}) v(x_{j+\frac{1}{2}}^{-}) - p_{11}^e(x_{j-\frac{1}{2}}^{+}) v(x_{j-\frac{1}{2}}^{+}) - \int_{I_j} p_{11}^e v_x \, dx \right),
\end{align}
where we have used (\ref{1D steady state}) in the second equality, and $\overline{(\cdot)}_j$ denotes the cell average of the associated quantity over $I_j$. Then this integral can be approximated as
\begin{align}\label{S2 integral approximation}
\int_{I_j}S^{[2]}vdx&\approx h_j\sum_{\mu=1}^N\omega_\mu\left(\frac{\rho_h(x_j^{(\mu)})}{\rho_h^e(x_j^{(\mu)})}-\frac{\overline{(\rho_h)}_j}{\overline{(\rho_h^e)}_j}\right)(p_{11,h}^e)_x(x_j^{(\mu)})v(x_j^{(\mu)}) \notag \\
&+\frac{\overline{(\rho_h)}_j}{\overline{(\rho_h^e)}_j}\left(p_{11,j+\frac{1}{2}}^{e,\ast}v(x_{j+\frac{1}{2}}^{-})-p_{11,j-\frac{1}{2}}^{e,\ast}v(x_{j-\frac{1}{2}}^{+})-h_j\sum_{\mu=1}^N\omega_\mu p_{11,h}^e(x_j^{(\mu)})v_x(x_j^{(\mu)}) \right) \\
&=:\langle S^{[2]},v \rangle_j.
\end{align}
Similarly, one has
\begin{align}\label{S4 integral approximation}
\int_{I_j}S^{[4]}vdx&\approx h_j\sum_{\mu=1}^N\omega_\mu\left(\frac{m_{1,h}(x_j^{(\mu)})}{\rho_h^e(x_j^{(\mu)})}-\frac{\overline{(m_{1,h})}_j}{\overline{(\rho_h^e)}_j}\right)(p_{11,h}^e)_x(x_j^{(\mu)})v(x_j^{(\mu)}) \notag \\
&+\frac{\overline{(m_{1,h})}_j}{\overline{(\rho_h^e)}_j}\left(p_{11,j+\frac{1}{2}}^{e,\ast}v(x_{j+\frac{1}{2}}^{-})-p_{11,j-\frac{1}{2}}^{e,\ast}v(x_{j-\frac{1}{2}}^{+})-h_j\sum_{\mu=1}^N\omega_\mu p_{11,h}^e(x_j^{(\mu)})v_x(x_j^{(\mu)}) \right) \\
&=:\langle S^{[4]},v \rangle_j.
\end{align}

\begin{align}\label{S5 integral approximation}
\int_{I_j}S^{[5]}vdx&\approx h_j\sum_{\mu=1}^N\omega_\mu\left(\frac{m_{2,h}(x_j^{(\mu)})}{2\rho_h^e(x_j^{(\mu)})}-\frac{\overline{(m_{2,h})}_j}{2\overline{(\rho_h^e)}_j}\right)(p_{11,h}^e)_x(x_j^{(\mu)})v(x_j^{(\mu)}) \notag \\
&+\frac{\overline{(m_{2,h})}_j}{2\overline{(\rho_h^e)}_j}\left(p_{11,j+\frac{1}{2}}^{e,\ast}v(x_{j+\frac{1}{2}}^{-})-p_{11,j-\frac{1}{2}}^{e,\ast}v(x_{j-\frac{1}{2}}^{+})-h_j\sum_{\mu=1}^N\omega_\mu p_{11,h}^e(x_j^{(\mu)})v_x(x_j^{(\mu)}) \right) \\
&=:\langle S^{[5]},v \rangle_j.
\end{align}
By combining (\ref{1D cell integral}) and (\ref{S2 integral approximation})-(\ref{S5 integral approximation}), we obtain the WB DG methods with forward Euler time discretization as
\begin{align}\label{1D semi WB DG}
\int_{I_j}\frac{\mathbf{U}_h^{\text{new}}-\mathbf{U}_h}{\Delta t}vdx=&h_j\sum_{\mu=1}^N\omega_\mu\mathbf{F}(\mathbf{U}_h(x_j^{(\mu)}))v_x(x_j^{(\mu)})-\left(\widehat{\mathbf{F}}_{j+\frac{1}{2}}v(x_{j+\frac{1}{2}}^{-})-
\widehat{\mathbf{F}}_{j-\frac{1}{2}}v(x_{j-\frac{1}{2}}^{+})\right) \notag \\
&+\left(0,\langle S^{[2]},v \rangle_j,0,\langle S^{[4]},v \rangle_j,\langle S^{[5]},v \rangle_j,0\right)^\top
\end{align}
for any $v\in\mathbb{V}_h^k$.


\begin{theorem}\label{1D well-balanceness}
For the 1D ten-moment Gaussian closure equations with source terms, the DG schemes {\rm (\ref{1D semi WB DG})} are WB for a general known hydrostatic equilibrium solution {\rm (\ref{1D steady state})}.
\end{theorem}
\begin{proof}
Assuming $\mathbf{U}_h$ reaches the equilibrium state (\ref{1D steady state}), one has $\rho_h=\rho_h^e$, $\mathbf{u}_h=\mathbf{u}_h^e=\mathbf{0}$, $E_{11,h}=\frac{1}{2}p_{11,h}^e$, $E_{12,h}=\frac{1}{2}p_{12,h}^e=\frac{1}{2}C_0$, $E_{22,h}=\frac{1}{2}p_{22,h}^e$. Thus, from (\ref{pj1/2pm}), (\ref{ph-pj1/2pm}) and (\ref{U-Uhat}), one gets
\[
\widehat{\mathbf{U}}_{j+\frac{1}{2}}^{\pm}=\left(\rho_h^e(x_{j+\frac{1}{2}}^{\pm}),0,0,\frac{1}{2}p_{11,j+\frac{1}{2}}^{e,\ast},\frac{1}{2}C_0,\frac{1}{2}p_{22,h}^e(x_{j+\frac{1}{2}}^{\pm})\right)^\top.
\]
According to the contact property of HLLC flux (Lemma \ref{contact property}), the HLLC numerical flux \eqref{modified HLLC flux} with modified solution states reduces to
\[
\widehat{\mathbf{F}}_{j+\frac{1}{2}}=\left(0,p_{11,j+\frac{1}{2}}^{e,\ast},C_0,0,0,0\right)^\top.
\]
Note that the first, fourth, fifth, and sixth components of both the flux and source terms approximation become zero. For the equation of momentum $m_1$, thanks to $\rho_h=\rho_h^e$, one has
\[
\langle S^{[2]},v\rangle_j=p_{11,j+\frac{1}{2}}^{e,\ast}v(x_{j+\frac{1}{2}}^{-})-p_{11,j-\frac{1}{2}}^{e,\ast}v(x_{j-\frac{1}{2}}^{+})-h_j\sum_{\mu=1}^N\omega_\mu(p_{11,h}^e)_x(x_j^{(\mu)})v_x(x_j^{(\mu)}).
\]
Let $F^{{[\ell]}}$ denotes the $\ell$-th component of $\mathbf{F}$. Since $\mathbf{u}_h=\mathbf{0}$, the flux term $F^{[2]}(\mathbf{U}_h(x_j^{(\mu)}))$ reduces to $p_{11,h}^e(x_j^{(\mu)})$. This implies
\begin{align*}
&h_j\sum_{\mu=1}^N\omega_\mu F^{[2]}(\mathbf{U}_h(x_j^{(\mu)}))v_x(x_j^{(\mu)})-\left(\widehat{F}_{j+\frac{1}{2}}^{[2]}v(x_{j+\frac{1}{2}}^{-})-
\widehat{F}_{j-\frac{1}{2}}^{[2]}v(x_{j-\frac{1}{2}}^{+})\right) \\
=&h_j\sum_{\mu=1}^N\omega_\mu p_{11,h}^e(x_j^{(\mu)})v_x(x_j^{(\mu)})-\left(p_{11,j+\frac{1}{2}}^{e,\ast}v(x_{j+\frac{1}{2}}^{-})-
p_{11,j-\frac{1}{2}}^{e,\ast}v(x_{j-\frac{1}{2}}^{+})\right) \\
=&-\langle S^{[2]},v\rangle_j.
\end{align*}
For the equation of momentum $m_2$, because $p_{12,h}^e=C_0$ is a constant at the equilibrium state (\ref{1D steady state}), one has
\begin{align*}
&h_j\sum_{\mu=1}^N\omega_\mu F^{[3]}(\mathbf{U}_h(x_j^{(\mu)}))v_x(x_j^{(\mu)})-\left(\widehat{F}_{j+\frac{1}{2}}^{[3]}v(x_{j+\frac{1}{2}}^{-})-
\widehat{F}_{j-\frac{1}{2}}^{[3]}v(x_{j-\frac{1}{2}}^{+})\right) \\
=&C_0h_j\sum_{\mu=1}^N\omega_\mu v_x(x_j^{(\mu)})-C_0(v(x_{j+\frac{1}{2}}^{-})-v(x_{j-\frac{1}{2}}^{+})) \\
=&C_0\left[\int_{I_j}v_xdx-(v(x_{j+\frac{1}{2}}^{-})-v(x_{j-\frac{1}{2}}^{+}))\right] \\
=&0.
\end{align*}
Hence, the right hand side of the DG methods (\ref{1D semi WB DG}) vanishes when $\mathbf{U}_h$ reaches the hydrostatic state. This implies
$\mathbf{U}_h^{\text{new}}=\mathbf{U}_h$ and completes the proof.
\end{proof}

To achieve high-order accuracy in time, some explicit strong-stability-preserving (SSP) methods \cite{gottlieb2001strong} can be used.
For example, one can utilize the third-order accurate SSP Rung-Kutta (SSP-RK) method
\begin{equation}\label{SSPRK3}
\begin{cases}
\mathbf{U}_h^{(1)}=\mathbf{U}_h^n+\Delta t\mathbf{L}(\mathbf{U}_h^n), \\
\mathbf{U}_h^{(2)}=\frac{3}{4}\mathbf{U}_h^n+\frac{1}{4}\left(\mathbf{U}_h^{(1)}+\Delta t\mathbf{L}(\mathbf{U}_h^{(1)})\right), \\
\mathbf{U}_h^{n+1}=\frac{1}{3}\mathbf{U}_h^n+\frac{2}{3}\left(\mathbf{U}_h^{(2)}+\Delta t\mathbf{L}(\mathbf{U}_h^{(2)})\right),
\end{cases}
\end{equation}
or the third-order accurate SSP multistep (SSP-MS) method
\begin{equation}\label{SSPMS3}
\mathbf{U}_h^{n+1}=\frac{16}{27}\left(\mathbf{U}_h^{n}+3\Delta t\mathbf{L}(\mathbf{U}_h^{n})\right)+
\frac{11}{27}\left(\mathbf{U}_h^{n-3}+\frac{12}{11}\Delta t\mathbf{L}(\mathbf{U}_h^{n-3})\right),
\end{equation}
where $\mathbf{U}_h^n$ denotes the DG solutions at the $n$-th time step.
Because they can be written as convex combinations of the forward Euler time discretization, they will keep the WB property and also maintain the positivity-preserving property discussed later.
\subsection{Positivity of first-order WB DG scheme}\label{pp analysis for 1d first-order schemes}
In this and the next subsections, we delve into the positivity-preserving analyses of our WB DG scheme {(\ref{1D semi WB DG})}. The modification to the solution states in the numerical flux and the specialized discretization of source terms introduce additional complexity into the positivity-preserving analyses, in comparison with the standard DG schemes \cite{meena2017positivity}. Notably, due to the presence of an anisotropic pressure tensor, our modification to the solution states in the numerical flux significantly differs from the approach in \cite{wu2021uniformly} designed for Euler equations with gravitation, where the pressure is a scalar. This distinction leads to some notable difficulties in our positivity-preserving analyses, rendering it more intricate than the corresponding analyses for the scalar pressure case in \cite{wu2021uniformly}.

Let $\overline{\mathbf{U}}_j(t){:=}\frac{1}{h_j}\int_{I_j}\mathbf{U}_h(x,t)dx$ denote the cell average of $\mathbf{U}_h$ over $I_j$. Taking $v=1$ in (\ref{1D semi WB DG}), one gets the evolution equations of the cell average as
\begin{equation}\label{1D cell average semi}
\overline{\mathbf{U}}_j^{new}=\overline{\mathbf{U}}_j-\frac{\Delta t}{h_j}\left(\widehat{\mathbf{F}}_{j+\frac{1}{2}}-\widehat{\mathbf{F}}_{j-\frac{1}{2}}\right)+\Delta t\overline{\mathbf{S}}_j^x=: \overline{\mathbf{U}}_j + \Delta t \mathbf{L}_j(\mathbf{U}_h),
\end{equation}
where $\overline{\mathbf{S}}_j^x=\left(0,\overline{S}_j^{[2]},0,\overline{S}_j^{[4]},\overline{S}_j^{[5]},0\right)^\top$ with $\overline{S}_j^{[{\ell}]}:=\frac{1}{h_j}\langle S^{[\ell]},1\rangle_j$, $\ell=2,4,5$.

When the DG polynomial degree $k=0$, one has $\mathbf{U}_h(x,t)\equiv\overline{\mathbf{U}}_j(t)$ for all $x\in I_j$, and (\ref{1D cell average semi}) reduces to the evolution of the cell average in first-order scheme with
\begin{equation}\label{1D first order flux}
\widehat{\mathbf{F}}_{j+\frac{1}{2}}=\mathbf{F}^{hllc}\left(\widehat{\overline{\mathbf{U}}}_j,\widehat{\overline{\mathbf{U}}}_{j+1}\right),
\end{equation}
where $\widehat{\overline{\mathbf{U}}}_j$ is related to $\overline{\mathbf{U}}_j$ by
\[
\widehat{\overline{\rho}}_j=\overline{\rho}_j, \quad
\widehat{\overline{\mathbf{m}}}_j=\mathbf{T}_j\overline{\mathbf{m}}_j, \quad
\widehat{\overline{\mathbf{E}}}_j=\mathbf{T}_j\overline{\mathbf{E}}_j(\mathbf{T}_j)^\top
\]
with
\[
\mathbf{T}_j=\begin{pmatrix}
               t_{1,j} & t_{2,j} \\
               0 & 1
             \end{pmatrix},
\]
where $t_{1,j}$ and $t_{2,j}$ are determined by the steady state solution; see (\ref{ph-pj1/2pm}) or Remark \ref{t1t2_high_order}. According to Lemma \ref{Uhat in G}, if $\overline{\mathbf{U}}_j\in\mathcal{G}$, then $\widehat{\overline{\mathbf{U}}}_j\in\mathcal{G}$.

We first analyze the positivity-preserving property of the homogeneous case, i.e., $\overline{\mathbf{S}}_j^x=0$. Using (\ref{1D first order flux}) gives
\begin{align}\label{1D first order scheme}
&\overline{\mathbf{U}}_j-\frac{\Delta t}{h_j}\left(\widehat{\mathbf{F}}_{j+\frac{1}{2}}-\widehat{\mathbf{F}}_{j-\frac{1}{2}}\right) \notag \\
=&\overline{\mathbf{U}}_j-\frac{\Delta t}{h_j}\left[\mathbf{F}^{hllc}\left(\widehat{\overline{\mathbf{U}}}_j,\widehat{\overline{\mathbf{U}}}_{j+1}\right)-
\mathbf{F}^{hllc}\left(\widehat{\overline{\mathbf{U}}}_{j-1},\widehat{\overline{\mathbf{U}}}_{j}\right)\right] \notag \\
=&\left(\overline{\mathbf{U}}_j-\xi\widehat{\overline{\mathbf{U}}}_j\right)+\xi\left[\widehat{\overline{\mathbf{U}}}_j-\frac{\Delta t}{\xi h_j}\left(\mathbf{F}^{hllc}\left(\widehat{\overline{\mathbf{U}}}_j,\widehat{\overline{\mathbf{U}}}_{j+1}\right)-
\mathbf{F}^{hllc}\left(\widehat{\overline{\mathbf{U}}}_{j-1},\widehat{\overline{\mathbf{U}}}_{j}\right)\right)\right] \notag \\
=&:\widetilde{\overline{\mathbf{U}}}_j+\mathbf{\Pi}_1,
\end{align}
where $\xi>0$ is a constant to be determined later. According to the positivity of the HLLC solver (see Lemma \ref{pp property}), if $\overline{\mathbf{U}}_j\in\mathcal{G}$ and the time step size satisfies
\begin{equation}\label{1D first order CFL}
\frac{\Delta t}{\xi h_j}\max\limits_{\mathbf{U}\in\{\widehat{\overline{\mathbf{U}}}_{j-1},\widehat{\overline{\mathbf{U}}}_{j},\widehat{\overline{\mathbf{U}}}_{j+1}\}}\alpha_{1}(\mathbf{U})\leq\frac{1}{2},
\end{equation}
then Lemma \ref{scale invariance} implies
\[
\mathbf{\Pi}_1=\xi\left[\widehat{\overline{\mathbf{U}}}_j-\frac{\Delta t}{\xi h_j}\left(\mathbf{F}^{hllc}\left(\widehat{\overline{\mathbf{U}}}_j,\widehat{\overline{\mathbf{U}}}_{j+1}\right)-
\mathbf{F}^{hllc}\left(\widehat{\overline{\mathbf{U}}}_{j-1},\widehat{\overline{\mathbf{U}}}_{j}\right)\right)\right]\in\mathcal{G} ~~ \forall \xi>0.
\]

Next, we derive a condition on $\xi$ such that $\widetilde{\overline{\mathbf{U}}}_j=\overline{\mathbf{U}}_j-\xi\widehat{\overline{\mathbf{U}}}_j\in\overline{\mathcal{G}}$.
 For the pressure tensor, one has
\begin{align*}
\widetilde{\overline{\mathbf{p}}}_j&=2\widetilde{\overline{\mathbf{E}}}_j-\frac{\widetilde{\overline{\mathbf{m}}}_j\widetilde{\overline{\mathbf{m}}}_j^\top}{\widetilde{\overline{\rho}}_j} \notag \\
&=2\left(\overline{\mathbf{E}}_j-\xi\mathbf{T}_j\overline{\mathbf{E}}_j\mathbf{T}_j^\top\right)-
\frac{\left(\mathbf{I}-\xi\mathbf{T}_j\right)\overline{\mathbf{m}}_j\overline{\mathbf{m}}_j^\top\left(\mathbf{I}-\xi\mathbf{T}_j\right)^\top}{(1-\xi)\overline{\rho}_j}.
\end{align*}
Substituting $\overline{\mathbf{E}}_j=\frac{1}{2}\left(\frac{\overline{\mathbf{m}}_j\overline{\mathbf{m}}_j^\top}{\overline{\rho}_j}+\overline{\mathbf{p}}_j\right)$ into above formula, one obtains
\begin{align}
\widetilde{\overline{\mathbf{p}}}_j&=\overline{\mathbf{p}}_j-\xi\mathbf{T}_j\overline{\mathbf{p}}_j\mathbf{T}_j^\top-\frac{\xi}{(1-\xi)\overline{\rho}_j}
\left(\mathbf{I}-\mathbf{T}_j\right)\overline{\mathbf{m}}_j\overline{\mathbf{m}}_j^\top\left(\mathbf{I}-\mathbf{T}_j\right)^\top \notag \\
&=\left[(1-\theta)\overline{\mathbf{p}}_j-\xi\mathbf{T}_j\overline{\mathbf{p}}_j\mathbf{T}_j^\top\right]+\left[\theta\overline{\mathbf{p}}_j-\frac{\xi}{(1-\xi)\overline{\rho}_j}\mathbf{Q}\right] \notag \\
&=:\mathbf{\Pi}_2+\mathbf{\Pi}_3, \label{WKLtemp1}
\end{align}
where $\theta \in [0,1)$ is an arbitrary parameter and
\[
\mathbf{Q}=\begin{pmatrix}
             q &0 \\
             0 & 0
           \end{pmatrix}
           :=\left(\mathbf{I}-\mathbf{T}_j\right)\overline{\mathbf{m}}_j\overline{\mathbf{m}}_j^\top\left(\mathbf{I}-\mathbf{T}_j\right)^\top
\]
with $q:=\left((1-t_{1,j})\overline{m}_{1,j}-t_{2,j}\overline{m}_{2,j}\right)^2$. If $q=0$, which implies $\mathbf{Q}\equiv0$, then $\theta$ is taken as 0; otherwise $\theta\in(0,1)$.
Denote
\begin{align}
f_1(\theta)&:=\frac{1-\theta}{\Vert\overline{\mathbf{p}}_j^{-\frac{1}{2}}\mathbf{T}_j\overline{\mathbf{p}}_j^{\frac{1}{2}}\Vert_2^2}, \label{f1_theta} \\
f_2(\theta)&:=\frac{\theta\overline{\rho}_j\overline{p}_{11,j}}{\theta\overline{\rho}_j\overline{p}_{11,j}+q}, \notag \\
f_3(\theta)&:=\frac{\theta\overline{\rho}_j\det(\overline{\mathbf{p}}_j)}{\theta\overline{\rho}_j\det(\overline{\mathbf{p}}_j)+q\overline{p}_{22,j}}. \notag
\end{align}
We then have the following conclusion.
\begin{lemma}\label{U_titlde_pp}
For $\overline{\mathbf{U}}_j\in\mathcal{G}$ and any $\theta \in [0,1)$, if
\begin{equation}\label{xi_f1_theta}
0<\xi\leq f(\theta)
\end{equation}
with
\[
f(\theta):=\begin{cases}
             \min\{f_1(\theta),f_2(\theta),f_3(\theta)\}, & \mbox{if } \theta\in(0,1), \\
             f_1(0), & \mbox{if } \theta=0,
           \end{cases}
\]
then $\widetilde{\overline{\mathbf{U}}}_j=\overline{\mathbf{U}}_j-\xi\widehat{\overline{\mathbf{U}}}_j\in\overline{\mathcal{G}}$.
\end{lemma}

\begin{proof}
	Based on the decomposition in \eqref{WKLtemp1}, it suffices to show that
	$\mathbf{\Pi}_2$ and $\mathbf{\Pi}_3$ are positive-semidefinite and
	that the first component $\widetilde{\overline{\rho}}_j$ of $\widetilde{\overline{\mathbf{U}}}_j$ is nonnegative.

Because $\overline{\mathbf{p}}_j$ is positive-definite, one has
\begin{align*}
\mathbf{\Pi}_2&=\overline{\mathbf{p}}_j^{\frac{1}{2}}\left[(1-\theta)\mathbf{I}-\xi\overline{\mathbf{p}}_j^{-\frac{1}{2}}\mathbf{T}_j\overline{\mathbf{p}}_j\mathbf{T}_j^\top\overline{\mathbf{p}}_j^{-\frac{1}{2}}\right]\overline{\mathbf{p}}_j^{\frac{1}{2}} \notag \\
&=:\overline{\mathbf{p}}_j^{\frac{1}{2}}\widehat{\mathbf{\Pi}}_2\overline{\mathbf{p}}_j^{\frac{1}{2}}.
\end{align*}
 Hence, $\mathbf{\Pi}_2$ is positive-semidefinite if and only if $\widehat{\mathbf{\Pi}}_2$ is positive-semidefinite, namely, if and only if all eigenvalues of $\widehat{\mathbf{\Pi}}_2$ are nonnegative. Because $\widehat{\mathbf{\Pi}}_2=(1-\theta)\mathbf{I}-\xi\overline{\mathbf{p}}_j^{-\frac{1}{2}}\mathbf{T}_j\overline{\mathbf{p}}_j^{\frac{1}{2}}(\overline{\mathbf{p}}_j^{-\frac{1}{2}}\mathbf{T}_j\overline{\mathbf{p}}_j^{\frac{1}{2}})^\top$, if $(1-\theta)-\xi\Vert\overline{\mathbf{p}}_j^{-\frac{1}{2}}\mathbf{T}_j\overline{\mathbf{p}}_j^{\frac{1}{2}}\Vert_2^2\geq0$ or equivalently if
\begin{equation}\label{positive definite condition 1}
\xi\leq\frac{1-\theta}{\Vert\overline{\mathbf{p}}_j^{-\frac{1}{2}}\mathbf{T}_j\overline{\mathbf{p}}_j^{\frac{1}{2}}\Vert_2^2}=f_1(\theta),
\end{equation}
then all eigenvalues of $\widehat{\mathbf{\Pi}}_2$ are nonnegative. This implies that $\mathbf{\Pi}_2$ is positive-semidefinite.

If $q=0$, then $\mathbf{Q}\equiv\mathbf{0}$. In this case, we take $\theta=0$, and then $\mathbf{\Pi}_3={\bf 0}$. Hence, if $\xi$ satisfies (\ref{positive definite condition 1}), then $\widetilde{\overline{\mathbf{p}}}_j=\mathbf{\Pi}_2\in\overline{\mathcal{G}}$. Note that
  $\Vert\overline{\mathbf{p}}_j^{-\frac{1}{2}}\mathbf{T}_j\overline{\mathbf{p}}_j^{\frac{1}{2}}\Vert_2$ is equal to or larger than the spectral radius of $\overline{\mathbf{p}}_j^{-\frac{1}{2}}\mathbf{T}_j\overline{\mathbf{p}}_j^{\frac{1}{2}}$,
  which equals the spectral radius of $\mathbf{T}_j$.
  Thus, $\Vert\overline{\mathbf{p}}_j^{-\frac{1}{2}}\mathbf{T}_j\overline{\mathbf{p}}_j^{\frac{1}{2}}\Vert_2 \ge \max \{|t_{1,j}|,1\} \ge 1$, and it follows that $f_1(0)\in(0,1]$. Hence, under the condition (\ref{xi_f1_theta}), one has $\widetilde{\overline{\rho}}_j=(1-\xi)\overline{\rho}_j\geq0$. In conclusion,  $\widetilde{\overline{\mathbf{U}}}_j\in\overline{\mathcal{G}}$.

If $q>0$, then $\mathbf{Q}\neq\mathbf{0}$. In this case, we require that $\theta\in(0,1)$ and have
\[
\mathbf{\Pi}_3=\begin{pmatrix}
        \theta\overline{p}_{11,j}-\frac{\xi}{(1-\xi)\overline{\rho}_j}q & \theta\overline{p}_{12,j} \\
        \theta\overline{p}_{12,j} & \theta\overline{p}_{22,j}
      \end{pmatrix}.
\]
Note that $\mathbf{\Pi}_3$ is positive-semidefinite under the following conditions
\begin{align}
\xi&\leq\frac{\theta\overline{\rho}_j\overline{p}_{11,j}}{\theta\overline{\rho}_j\overline{p}_{11,j}+q}=f_2(\theta), \label{positive definite condition 2} \\
\xi&\leq\frac{\theta\overline{\rho}_j\det(\overline{\mathbf{p}}_j)}{\theta\overline{\rho}_j\det(\overline{\mathbf{p}}_j)+q\overline{p}_{22,j}}=f_3(\theta). \label{positive definite condition 4}
\end{align}
Hence, if $\xi$ satisfies (\ref{positive definite condition 1})--(\ref{positive definite condition 4}) simultaneously, i.e.,
\begin{equation*}
\xi\leq\min\{f_1(\theta),f_2(\theta),f_3(\theta)\}=f(\theta), \quad \theta\in(0,1),
\end{equation*}
then $\widetilde{\overline{\mathbf{p}}}_j=\mathbf{\Pi}_2+\mathbf{\Pi}_3$ is positive-semidefinite. When $q>0$, one has $f_2(\theta), f_3(\theta)\in(0,1)$ for any $\theta\in(0,1)$. Because $\Vert\overline{\mathbf{p}}_j^{-\frac{1}{2}}\mathbf{T}_j\overline{\mathbf{p}}_j^{\frac{1}{2}}\Vert_2\geq1$, one has $f_1(\theta)\in(0,1)$ for any $\theta\in(0,1)$. Hence, under the condition (\ref{xi_f1_theta}), one has $\xi\in(0,1)$ and then $\widetilde{\overline{\rho}}_j=(1-\xi)\overline{\rho}_j>0$. This, together with the positive-semidefiniteness of $\widetilde{\overline{\mathbf{p}}}_j$, implies that $\widetilde{\overline{\mathbf{U}}}_j\in\overline{\mathcal{G}}$.
\end{proof}

Combining (\ref{1D first order scheme}), Lemma \ref{U_titlde_pp} with Lemma \ref{general combination} gives the following result: for any given parameter $\theta \in [0,1)$,
 if $\xi$ satisfies (\ref{xi_f1_theta}) and the time step size satisfies (\ref{1D first order CFL}), then
\[
\overline{\mathbf{U}}_j-\frac{\Delta t}{h_j}\left(\widehat{\mathbf{F}}_{j+\frac{1}{2}}-\widehat{\mathbf{F}}_{j-\frac{1}{2}}\right)=\widetilde{\overline{\mathbf{U}}}_j+\mathbf{\Pi}_1\in\mathcal{G}.
\]
It is worth exploring the ``optimal'' parameter $\theta_{\ast}$ such that the positivity-preserving CFL condition (\ref{1D first order CFL}) is as mild as possible, or equivalently, $\xi$ as large as possible.
Because of the constraint \eqref{xi_f1_theta}, the task is to find the ``optimal'' parameter $\theta_{\ast}$ that maximizes $f(\theta)$ .
Observe that
\[
\max_{\theta\in[0,1)}f(\theta)=f(\theta_{\ast})=f_1(\theta_{\ast}),
\]
where
\begin{equation}\label{theta_ast}
\theta_{\ast}:=\max\{\theta_{1,2},\theta_{1,3}\}
\end{equation}
with
\[
\theta_{1,\ell}:=\frac{\sqrt{(f^{(1)}+f^{(\ell)}-1)^2+4f^{(\ell)}}-(f^{(1)}+f^{(\ell)}-1)}{2},
\quad \ell=2,3,
\]
and
\[
f^{(1)}:=\Vert\overline{\mathbf{p}}_j^{-\frac{1}{2}}\mathbf{T}_j\overline{\mathbf{p}}_j^{\frac{1}{2}}\Vert_2^2,
\quad
f^{(2)}:=\frac{q}{\overline{\rho}_j\overline{p}_{11,j}},
\quad
f^{(3)}:=\frac{q\overline{p}_{22,j}}{\overline{\rho}_j\det(\overline{\mathbf{p}}_j)}.
\]
In fact, when $q>0$, $\theta_{1,\ell}$ is the unique intersection point of the graphs of $f_1(\theta)$ and $f_\ell(\theta)$ for $\ell=2,3$, respectively; when $q=0$, $\theta_{1,\ell}=0$ for $\ell=2,3$ and $\theta_\ast=0$.
In summary, we have the following theorem.

%
%

\begin{theorem}\label{1D first order homogeneous positivity}
If the DG polynomial degree $k=0$, $\overline{\mathbf{U}}_j\in\mathcal{G}$ for all $j$, and the time step size satisfies the CFL-type condition
\begin{equation}\label{k0 PP CFL}
\frac{\Delta t}{h_j}\left(\frac{1}{\xi_\ast(\overline{\mathbf{U}}_j)}\max\limits_{\mathbf{U}\in\{\widehat{\overline{\mathbf{U}}}_{j-1},\widehat{\overline{\mathbf{U}}}_{j},\widehat{\overline{\mathbf{U}}}_{j+1}\}}\alpha_{1}(\mathbf{U})\right)\leq\frac{1}{2},
\end{equation}
where $\xi_\ast(\overline{\mathbf{U}}_j):=f_1(\theta_\ast)$,
then
\[
\overline{\mathbf{U}}_j-\frac{\Delta t}{h_j}\left(\widehat{\mathbf{F}}_{j+\frac{1}{2}}-\widehat{\mathbf{F}}_{j-\frac{1}{2}}\right)\in\mathcal{G}.
\]
Here the function $f_1(\theta)$ is defined in {\rm (\ref{f1_theta})}, and $\theta_\ast$ is defined in {\rm (\ref{theta_ast})}.
\end{theorem}

\begin{remark}
It is worth noting that $\xi_\ast=f_1(\theta_\ast)=1+O(h^{k+1})$, which can be proven as follows. Remark {\rm \ref{t1t2_high_order}} indicates that $t_{1,j}=1+O(h^{k+1})$ and $t_{2,j}=O(h^{k+1})$, which implies  $q=\left((1-t_{1,j})\overline{m}_{1,j}-t_{2,j}\overline{m}_{2,j}\right)^2=O(h^{2(k+1)})$. It follows that $f^{(2)}=O(h^{2(k+1)})$ and $f^{(3)}=O(h^{2(k+1)})$. Since $\mathbf{T}_j$ is an approximation to the identity matrix of order $k+1$, one has $f^{(1)}=1+O(h^{k+1})$. Hence $\theta_\ast=\max\{\theta_{1,2},\theta_{1,3}\}=O(h^{k+1})$. Consequently, $\xi_\ast=f_1(\theta_\ast)=\frac{1-\theta_\ast}{f^{(1)}}=1+O(h^{k+1})$.
In view of this fact, the CFL-type condition \eqref{k0 PP CFL} is not restrictive.
\end{remark}


Combining Theorem \ref{1D first order homogeneous positivity} and the GQL representation of the admissible set $\mathcal{G}$ (see Lemma \ref{GQL}) leads to the following corollary. It will play a vital role in the subsequent positivity-preserving analyses.
\begin{corollary}\label{1d_first_order_corollary}
For any $\mathbf{U}_1,\mathbf{U}_2,\mathbf{U}_3\in\mathcal{G}$, it holds that
\begin{align*}
{\bf \Pi}(\widehat{\mathbf{U}}_1,\widehat{\mathbf{U}}_2,\widehat{\mathbf{U}}_3)\cdot\mathbf{e}_1&>- \eta_1^\ast \mathbf{U}_1\cdot\mathbf{e}_1,
\\
\varphi\left({\bf \Pi}(\widehat{\mathbf{U}}_1,\widehat{\mathbf{U}}_2,\widehat{\mathbf{U}}_3);\mathbf{z},\mathbf{u}_\ast\right)&>- \eta_1^\ast \varphi(\mathbf{U}_1;\mathbf{z},\mathbf{u}_\ast) \quad \forall \mathbf{u}_{*}\in\mathbb{R}^2, ~ \forall \mathbf{z}\in\mathbb{R}^2\backslash\{\mathbf{0}\},
\end{align*}
where
\begin{align*}
&	{\bf \Pi}(\widehat{\mathbf{U}}_1,\widehat{\mathbf{U}}_2,\widehat{\mathbf{U}}_3):=-\left[\mathbf{F}^{hllc}(\widehat{\mathbf{U}}_{1},\widehat{\mathbf{U}}_{2})
-\mathbf{F}^{hllc}(\widehat{\mathbf{U}}_{3},\widehat{\mathbf{U}}_{1})\right],
\\
&\eta_1^\ast := \frac{2}{\xi_\ast(\mathbf{U}_1)}\max\limits_{\mathbf{U}\in\{\widehat{\mathbf{U}}_{1},\widehat{\mathbf{U}}_{2},\widehat{\mathbf{U}}_{3}\}}\alpha_{1}(\mathbf{U}).
\end{align*}
\end{corollary}

Next, we discuss the nonhomogeneous case. When $k=0$, it follows from (\ref{S2 integral approximation})--(\ref{S5 integral approximation}) that
\begin{equation*}
\overline{\mathbf{S}}_j^x=\overline{\alpha}_{11,j}\overline{\mathbf{S}}_{1,j}
\end{equation*}
with
\[
\overline{\alpha}_{11,j}:=\frac{p_{11,j+\frac{1}{2}}^{e,\ast}-p_{11,j-\frac{1}{2}}^{e,\ast}}{h_j{\overline{(\rho_h^e)}_j}}, \quad
\overline{\mathbf{S}}_{1,j}:=\left(0,\overline{\rho}_j,0,\overline{m}_{1,j},\frac{1}{2}\overline{m}_{2,j},0\right)^\top.
\]
The scheme (\ref{1D cell average semi}) can be rewritten as
\begin{equation}\label{eq:WKLtmp}
	\overline{\mathbf{U}}_j^{\text{new}} = \overline{\mathbf{U}}_j+\Delta t\mathbf{L}_j(\mathbf{U}_h)=\overline{\mathbf{U}}_j + \Delta t \left(  \frac{1}{h_j} {\bf \Pi}_4 +  \overline{\mathbf{S}}_j^x \right)
\end{equation}
with
\[
{\bf \Pi}_4:=-\left[\mathbf{F}^{hllc}\left(\widehat{\overline{\mathbf{U}}}_j,\widehat{\overline{\mathbf{U}}}_{j+1}\right)-
\mathbf{F}^{hllc}\left(\widehat{\overline{\mathbf{U}}}_{j-1},\widehat{\overline{\mathbf{U}}}_{j}\right)\right].
\]

\begin{theorem}\label{1D first order scheme PP theorem}
	If the DG polynomial degree $k=0$ and $\overline{\mathbf{U}}_j\in\mathcal{G}$ for all $j$, then
	\[
	\overline{\mathbf{U}}_j^{\text{new}} =  \overline{\mathbf{U}}_j+\Delta t\mathbf{L}_j(\mathbf{U}_h)\in\mathcal{G}
	\]
	under the CFL-type condition
\begin{equation}\label{eq:CFLWKL1}
		\Delta t \left( \frac{\eta_{1,j}^\ast}{\Delta x}+\overline{\beta}_{11,j} \right) \leq 1,
\end{equation}
	where 
	\begin{equation}
		\overline{\beta}_{11,j}:=\left|{\overline{\alpha}_{11,j}\delta_1(\overline{\mathbf{U}}_j)}\right|, \qquad \eta_{1,j}^\ast:=\frac{2}{\xi_\ast ( {\overline{\mathbf{U}}}_{j}) }\max\limits_{\mathbf{U}\in\{\widehat{\overline{\mathbf{U}}}_{j-1},\widehat{\overline{\mathbf{U}}}_{j},\widehat{\overline{\mathbf{U}}}_{j+1}\}}\alpha_{1}(\mathbf{U}).
	\end{equation}

\end{theorem}

\begin{proof}
Since the first component of $\overline{\mathbf{S}}_j^x$ is zero, Theorem \ref{1D first order homogeneous positivity} implies that
$$
\overline{\mathbf{U}}_j^{\text{new}} \cdot {\bf e}_1 > 0.
$$
Thanks to the linearity of $\varphi(\mathbf{U};\mathbf{z},\mathbf{u}_\ast)$ with respect to $\bf U$, it follows from \eqref{eq:WKLtmp} that
\begin{align*}
	\varphi( \overline{\mathbf{U}}_j^{\text{new}};\mathbf{z},\mathbf{u}_\ast) &= \varphi(  \overline{\mathbf{U}}_j ;\mathbf{z},\mathbf{u}_\ast) +    \frac{\Delta t}{h_j} \varphi(  {\bf \Pi}_4 ;\mathbf{z},\mathbf{u}_\ast) + \Delta t \varphi(  \overline{\mathbf{S}}_j^x  ;\mathbf{z},\mathbf{u}_\ast)
	\\
	& > \varphi(  \overline{\mathbf{U}}_j ;\mathbf{z},\mathbf{u}_\ast) - \frac{\Delta t}{h_j} \eta_{1,j}^\ast
	\varphi(  \overline{\mathbf{U}}_j ;\mathbf{z},\mathbf{u}_\ast) - \Delta t \overline{\beta}_{11,j} \varphi(\overline{\mathbf{U}}_j;\mathbf{z},\mathbf{u}_\ast)
	\\
	& = \left( 1 - \frac{\Delta t}{h_j} \eta_{1,j}^\ast  - \Delta t \overline{\beta}_{11,j} \right) \varphi(\overline{\mathbf{U}}_j;\mathbf{z},\mathbf{u}_\ast) \ge 0,
\end{align*}
where Corollaries \ref{GQL corollary} and \ref{1d_first_order_corollary} have been used in the second step, and the   CFL-type condition \ref{eq:CFLWKL1} has been used in the last step.
According to the GQL representation in Lemma \ref{GQL}, $\overline{\mathbf{U}}_j^{\text{new}} \in {\mathcal G}$.
The proof is completed.
\end{proof}

\subsection{Positivity-preserving high-order WB DG schemes}\label{pp analysis of 1d high-order schemes}
When the DG polynomial degree $k\geq1$, we derive a weak positivity for the cell averages of the high-order WB DG method (\ref{1D semi WB DG}); see Theorem \ref{1D high-order positivity}. Based on such weak positivity, one can apply a simple limiter to enforce the physical admissibility of the DG solution polynomials at certain points of interest without loss of conservation and high-order accuracy.

\subsubsection{Theoretical positivity-preserving analysis}\label{1d high-order pp analysis}
If the DG polynomial degree $k\geq1$, then the scheme (\ref{1D cell average semi}) can be rewritten as
\begin{equation}\label{1D high order CA evolution}
\overline{\mathbf{U}}_j^{\text{new}}=\overline{\mathbf{U}}_j+\Delta t\mathbf{L}_j(\mathbf{U}_h)=\overline{\mathbf{U}}_j+\Delta t\left[\frac{1}{h_j}(\mathbf{\Pi}_5+\mathbf{\Pi}_6)+\overline{\mathbf{S}}_j^x\right],
\end{equation}
where
\begin{align*}
\mathbf{\Pi}_5&:=-\left[\mathbf{F}^{hllc}\left(\widehat{\mathbf{U}}_{j+\frac{1}{2}}^{-},\widehat{\mathbf{U}}_{j+\frac{1}{2}}^{+}\right)-
\mathbf{F}^{hllc}\left(\widehat{\mathbf{U}}_{j-\frac{1}{2}}^{+},\widehat{\mathbf{U}}_{j+\frac{1}{2}}^{-}\right)\right], \\
\mathbf{\Pi}_6&:=-\left[\mathbf{F}^{hllc}\left(\widehat{\mathbf{U}}_{j-\frac{1}{2}}^{+},\widehat{\mathbf{U}}_{j+\frac{1}{2}}^{-}\right)-
\mathbf{F}^{hllc}\left(\widehat{\mathbf{U}}_{j-\frac{1}{2}}^{-},\widehat{\mathbf{U}}_{j-\frac{1}{2}}^{+}\right)\right].
\end{align*}
Recall that $\overline{\mathbf{S}}_j^x=\left(0,\overline{S}_j^{[2]},0,\overline{S}_j^{[4]},\overline{S}_j^{[5]},0\right)^\top$ with $\overline{S}_j^{[{\ell}]}=\frac{1}{h_j}\langle S^{[\ell]},1\rangle$, $\ell=2,4,5$. Taking $v=1$ in (\ref{S2 integral approximation})--(\ref{S5 integral approximation}) gives
\begin{equation}\label{1D high order source}
\overline{\mathbf{S}}_j^x=\sum_{\mu=1}^{N}\omega_\mu\frac{(p_{11,h}^e)_x(x_j^{(\mu)})}{\rho_h^e(x_j^{(\mu)})}\mathbf{S}_{1,h}(x_j^{(\mu)})
+\alpha_{11,j}\overline{\mathbf{S}}_{1,j},
\end{equation}
where
\begin{align*}
&\mathbf{S}_{1,h}(x_j^{(\mu)}):=\left(0,\rho_h(x_j^{(\mu)}),0,m_{1,h}(x_j^{(\mu)}),\frac{1}{2}
m_{2,h}(x_j^{(\mu)}),0\right)^\top, \\
&\alpha_{11,j}:=\frac{p_{11,j+\frac{1}{2}}^{e,\ast}-p_{11,j-\frac{1}{2}}^{e,\ast}-p_{11,h}^e(x_{j+\frac{1}{2}}^-)+p_{11,h}^e(x_{j-\frac{1}{2}}^+)}{h_j{\overline{(\rho_h^e)}_j}}, \end{align*}

Denote
\begin{equation}\label{Sj}
\mathbb{S}_j:=\{\widehat{x}_j^{(\nu)}\}_{\nu=1}^L\cup\{x_j^{(\mu)}\}_{\mu=1}^N,
\end{equation}
where $L=\lceil\frac{k+3}{2}\rceil$, and $\{\widehat{x}_j^{(\nu)}\}_{\nu=1}^L$ are the $L$-point Gauss--Lobatto quadrature nodes scaled in $I_j$ with $\widehat{x}_j^{(1)}=x_{j-\frac{1}{2}}$ and $\widehat{x}_j^{(L)}=x_{j+\frac{1}{2}}$. Let $\{\widehat{\omega}_{\nu}\}_{\nu=1}^L$ be the corresponding (positive) quadrature weights satisfying $\sum_{\nu=1}^L\widehat{\omega}_{\nu}=1$ and $\widehat{\omega}_{1}=\widehat{\omega}_{L}=\frac{1}{L(L-1)}$. Then one has
\begin{equation}\label{GL quadrature}
\overline{\mathbf{U}}_j=\frac{1}{h_j}\int_{I_j}\mathbf{U}_h{(x)}dx=\sum_{\nu=1}^{L}\widehat{\omega}_{\nu}\mathbf{U}_h(\widehat{x}_j^{(\nu)}).
\end{equation}

\begin{theorem}\label{1D high-order positivity}
Assume that the projected hydrostatic equilibrium solution satisfies
\begin{equation}\label{PP of projection}
\rho_h^e(x)>0, \quad \mathbf{z}^\top\mathbf{p}_h^e(x)\mathbf{z}>0 \quad \forall \mathbf{z}\in\mathbb{R}^2\setminus\{\mathbf{0}\}, ~ \forall x\in\mathbb{S}_j, ~ \forall j,
\end{equation}
and the numerical solution $\mathbf{U}_h$ satisfies
\begin{equation}\label{PP of quadrature node 2}
\mathbf{U}_h(x)\in\mathcal{G} \quad \forall x\in\mathbb{S}_j, ~ \forall j.
\end{equation}
Then
\begin{equation}\label{1D high order weak PP}
\overline{\mathbf{U}}_j^{\text{new}}=\overline{\mathbf{U}}_j+\Delta t\mathbf{L}_j(\mathbf{U}_h)\in\mathcal{G} \quad \forall j,
\end{equation}
under the CFL-type condition
\begin{equation}\label{1D high order PP CFL}
\Delta t\left(\frac{\max\{\eta_{1,j+\frac{1}{2}}^{\ast,-},\eta_{1,j-\frac{1}{2}}^{\ast,+}\}}{\widehat{\omega}_1h_j}+\beta_{11,j}\right)\leq1,
\end{equation}
where
\begin{align*}
&\eta_{1,j+\frac{1}{2}}^{\ast,-}:=\frac{2}{\xi_\ast(\mathbf{U}_{j+\frac{1}{2}}^{-})}\max\limits_{\mathbf{U}\in\{\widehat{\mathbf{U}}_{j+\frac{1}{2}}^{-},\widehat{\mathbf{U}}_{j+\frac{1}{2}}^{+},\widehat{\mathbf{U}}_{j-\frac{1}{2}}^{+}\}}\alpha_{1}(\mathbf{U}),
\\
&\eta_{1,j-\frac{1}{2}}^{\ast,+}:=\frac{2}{\xi_\ast(\mathbf{U}_{j-\frac{1}{2}}^{+})}\max\limits_{\mathbf{U}\in\{\widehat{\mathbf{U}}_{j-\frac{1}{2}}^{+},\widehat{\mathbf{U}}_{j+\frac{1}{2}}^{-},\widehat{\mathbf{U}}_{j-\frac{1}{2}}^{-}\}}\alpha_{1}(\mathbf{U}),
 \\
&\beta_{11,j}:={\max_{1\leq\mu\leq N}\left\{\left|\frac{(p_{11,h}^e)_x(x_j^{(\mu)})}{\rho_h^e(x_j^{(\mu)})}\delta_1(\mathbf{U}_{h}(x_j^{(\mu)}))\right|\right\}+\left|\alpha_{11,j}\delta_1(\overline{\mathbf{U}}_j)\right|}.
\end{align*}
\end{theorem}

\begin{proof}
Since $\overline{\mathbf{S}}_j^x\cdot {\bf e}_1=0$, (\ref{1D high order CA evolution}) implies that
\begin{equation*}
	\begin{aligned}
\overline{\mathbf{U}}_j^{\text{new}} \cdot {\bf e}_1&=\overline{\mathbf{U}}_j \cdot {\bf e}_1+\frac{\Delta t}{h_j}({\bf \Pi}_5\cdot{\bf e}_1+{\bf \Pi}_6\cdot{\bf e}_1) \\
&>\sum_{\nu=1}^{L}\widehat{\omega}_{\nu}\mathbf{U}_h(\widehat{x}_j^{(\nu)})\cdot {\bf e}_1-\frac{\Delta t}{h_j}\left(\eta_{1,j+\frac{1}{2}}^{\ast,-}\mathbf{U}_{j+\frac{1}{2}}^{-}\cdot {\bf e}_1+\eta_{1,j-\frac{1}{2}}^{\ast,+}\mathbf{U}_{j-\frac{1}{2}}^{+}\cdot {\bf e}_1\right) \\
&=\left(\widehat{\omega}_{L}-\frac{\Delta t}{h_j}\eta_{1,j+\frac{1}{2}}^{\ast,-}\right)\mathbf{U}_{j+\frac{1}{2}}^{-}\cdot {\bf e}_1+
\left(\widehat{\omega}_{1}-\frac{\Delta t}{h_j}\eta_{1,j-\frac{1}{2}}^{\ast,+}\right)\mathbf{U}_{j-\frac{1}{2}}^{+}\cdot {\bf e}_1+
\sum_{\nu=2}^{L-1}\widehat{\omega}_{\nu}\mathbf{U}_h(\widehat{x}_j^{(\nu)})\cdot {\bf e}_1 \\
&>0,
\end{aligned}
\end{equation*}
where the cell average decomposition (\ref{GL quadrature}) and Corollary \ref{1d_first_order_corollary} have been used in the second step, and the assumption (\ref{PP of quadrature node 2}) and the CFL condition (\ref{1D high order PP CFL}) have been used in the last step.

Thanks to the linearity of $\varphi(\mathbf{U};\mathbf{z},\mathbf{u}_\ast)$ with respect to $\bf U$, it follows from (\ref{1D high order CA evolution}) that
\begin{align*}
\varphi( \overline{\mathbf{U}}_j^{\text{new}};\mathbf{z},\mathbf{u}_\ast) &= \varphi(  \overline{\mathbf{U}}_j ;\mathbf{z},\mathbf{u}_\ast) +    \frac{\Delta t}{h_j} \varphi(  {\bf \Pi}_5 ;\mathbf{z},\mathbf{u}_\ast) + \frac{\Delta t}{h_j} \varphi(  {\bf \Pi}_6 ;\mathbf{z},\mathbf{u}_\ast) + \Delta t \varphi(  \overline{\mathbf{S}}_j^x  ;\mathbf{z},\mathbf{u}_\ast) \\
&>\varphi(  \overline{\mathbf{U}}_j ;\mathbf{z},\mathbf{u}_\ast)-\frac{\Delta t}{h_j}\eta_{1,j+\frac{1}{2}}^{\ast,-}\varphi(\mathbf{U}_{j+\frac{1}{2}}^{-};\mathbf{z},\mathbf{u}_\ast)-\frac{\Delta t}{h_j}\eta_{1,j-\frac{1}{2}}^{\ast,+}\varphi(\mathbf{U}_{j-\frac{1}{2}}^{+};\mathbf{z},\mathbf{u}_\ast) \\
&\quad-\Delta t\sum_{\mu=1}^{N}\omega_\mu\left|\frac{(p_{11,h}^e)_x(x_j^{(\mu)})}{\rho_h^e(x_j^{(\mu)})}{\delta_{1}(\mathbf{U}_{h}(x_j^{(\mu)}))}\right|\varphi(\mathbf{U}_{h}(x_j^{(\mu)});\mathbf{z},\mathbf{u}_\ast)-
\Delta t\left|{\alpha_{11,j}\delta_{1}(\overline{\mathbf{U}}_{j})}\right|\varphi(\overline{\mathbf{U}}_{j};\mathbf{z},\mathbf{u}_\ast) \\
&\geq(1-\Delta t\beta_{11,j})\varphi(\overline{\mathbf{U}}_{j};\mathbf{z},\mathbf{u}_\ast)-\frac{\Delta t}{h_j}\eta_{1,j+\frac{1}{2}}^{\ast,-}\varphi(\mathbf{U}_{j+\frac{1}{2}}^{-};\mathbf{z},\mathbf{u}_\ast)-\frac{\Delta t}{h_j}\eta_{1,j-\frac{1}{2}}^{\ast,+}\varphi(\mathbf{U}_{j-\frac{1}{2}}^{+};\mathbf{z},\mathbf{u}_\ast) \\
&=\left[\widehat{\omega}_L(1-\Delta t\beta_{11,j})-\frac{\Delta t}{h_j}\eta_{1,j+\frac{1}{2}}^{\ast,-}\right]\varphi(\mathbf{U}_{j+\frac{1}{2}}^{-};\mathbf{z},\mathbf{u}_\ast) \\
&\quad+\left[\widehat{\omega}_1(1-\Delta t\beta_{11,j})-\frac{\Delta t}{h_j}\eta_{1,j-\frac{1}{2}}^{\ast,+}\right]\varphi(\mathbf{U}_{j-\frac{1}{2}}^{+};\mathbf{z},\mathbf{u}_\ast)
+(1-\Delta t\beta_{11,j})\sum_{\nu=2}^{L-1}\widehat{\omega}_{\nu}\varphi(\mathbf{U}_h(\widehat{x}_j^{(\nu)}) ;\mathbf{z},\mathbf{u}_\ast) \\
&>0,
\end{align*}
where Corollary \ref{GQL corollary}, Corollary \ref{1d_first_order_corollary} and (\ref{1D high order source}) have been used in the second step, the cell average decomposition (\ref{GL quadrature}) has been used in the fourth step, and the assumption (\ref{PP of quadrature node 2}) and the CFL condition (\ref{1D high order PP CFL}) have been used in the last step.
According to the GQL representation (Lemma \ref{GQL}), $\overline{\mathbf{U}}_j^{\text{new}} \in {\mathcal G}$.
The proof is completed.
\end{proof}

\subsubsection{Positivity-preserving limiter}
In general, a high-order DG scheme does not automatically satisfy the condition (\ref{PP of quadrature node 2}). In such cases, one can use a simple positivity-preserving limiter (cf. \cite{zhang2010positivity,wang2012robust,meena2017positivity}) to enforce the condition (\ref{PP of quadrature node 2}) without losing conservation and high-order accuracy.

Denote
\begin{align*}
\overline{\mathbb{G}}_h^k&:=\left\{\mathbf{v}\in[\mathbb{V}_h^k]^6: \frac{1}{h_j}\int_{I_j}\mathbf{v}(x)dx\in\mathcal{G} ~~ \forall j\right\}, \\
\widetilde{\mathbb{G}}_h^k&:=\left\{\mathbf{v}\in[\mathbb{V}_h^k]^6: \mathbf{v}|_{I_j}(x)\in\mathcal{G} ~~ \forall x\in\mathbb{S}_j, ~ \forall j\right\}.
\end{align*}
For any $\mathbf{U}_h\in\overline{\mathbb{G}}_h^k$ with $\mathbf{U}_h|_{I_j}=:\mathbf{U}_j(x)=\left(\rho_j(x),m_{1,j}(x),m_{2,j}(x),E_{11,j}(x),E_{12,j}(x),E_{22,j}(x)\right)^\top$, we define the positivity-preserving limiting operator ${\mathcal{T}}_h:\overline{\mathbb{G}}_h^k\rightarrow\widetilde{\mathbb{G}}_h^k$ by
\begin{equation}\label{PP operator}
{\mathcal{T}}_h\mathbf{U}_h|_{I_j}=\mathbf{U}_j^{{\rm ({\textsc{iii}})}}(x) \quad \forall j
\end{equation}
with the limited polynomial vector function $\mathbf{U}_j^{({\textsc{iii}})}(x)$ constructed through the following three steps.

\textbf{Step (i)}: First, if ${\min_{x\in\mathbb{S}_j}}\rho_j(x)<\varepsilon_1$, modify the density to enforce its positivity via
\begin{equation}\label{pp for rho}
\rho_j^{(\textsc{i})}(x)=\theta_1\left(\rho_j(x)-\overline{\rho}_j\right)+\overline{\rho}_j, \quad
\theta_1:=\min\left\{1,\frac{\overline{\rho}_j-\varepsilon_1}{\overline{\rho}_j-\min_{x\in\mathbb{S}_j}\rho_j(x)}\right\},
\end{equation}
where $\varepsilon_1$ is a small positive number as the desired lower bound of the density  and may be taken as $\varepsilon_1=\min\{10^{-13},\overline{\rho}_j\}$.

\textbf{Step (ii)}: Then, if $\min_{x\in\mathbb{S}_j}g_{11}(\mathbf{U}_j^{(\textsc{i})}(x))<\varepsilon_2$, modify
\[
\mathbf{U}_j^{(\textsc{i})}(x):=\left(\rho_j^{(\textsc{i})}(x),m_{1,j}(x),m_{2,j}(x),E_{11,j}(x),E_{12,j}(x),E_{22,j}(x)\right)^\top
\]
 to enforce the positivity of $g_{11}(\mathbf{U})$ {as follows:}
\begin{equation}\label{pp for g11}
\mathbf{U}_j^{(\textsc{ii})}(x)=\theta_2\left(\mathbf{U}_j^{(\textsc{i})}(x)-\overline{\mathbf{U}}_j\right)+\overline{\mathbf{U}}_j, \quad
\theta_2:=\min\left\{1,\frac{g_{11}(\overline{\mathbf{U}}_j)-\varepsilon_2}{g_{11}(\overline{\mathbf{U}}_j)-\min_{x\in\mathbb{S}_j}g_{11}(\mathbf{U}_j^{(\textsc{i})}(x))}\right\},
\end{equation}
where $\varepsilon_2$ is also a small positive number as the desired lower bound of $g_{11}(\mathbf{U})$ and may be taken as $\varepsilon_2=\min\{10^{-13},g_{11}(\overline{\mathbf{U}}_j)\}$.

\textbf{Step (iii)}: Finally, 
modify $\mathbf{U}_j^{(\textsc{ii})}(x)$ to enforce the positivity of $g_{\det}(\mathbf{U})$ as follows:
\begin{equation}\label{pp for gdet}
\mathbf{U}_j^{(\textsc{iii})}(x)=\theta_3\left(\mathbf{U}_j^{(\textsc{ii})}(x)-\overline{\mathbf{U}}_j\right)+\overline{\mathbf{U}}_j, \quad
\theta_3:=\min_{x\in\mathbb{S}_j}\widetilde{\theta}(x),
\end{equation}
where, for $x\in\{x\in\mathbb{S}_j:g_{\det}(\mathbf{U}_j^{(\textsc{ii})}(x))\geq\varepsilon_3\}$, $\widetilde{\theta}(x)=1$, and, for $x\in\{x\in\mathbb{S}_j:g_{\det}(\mathbf{U}_j^{(\textsc{ii})}(x))<\varepsilon_3\}$, $\widetilde{\theta}(x)$ is the solution to the equation
\[
g_{\det}\left((1-\widetilde{\theta})\overline{\mathbf{U}}_j+\widetilde{\theta}\mathbf{U}_j^{(\textsc{ii})}(x)\right)=\varepsilon_3,
\quad
\widetilde{\theta}\in[0,1),
\]
where $\varepsilon_3$ is a small positive number as the desired lower bound of $g_{\det}(\mathbf{U})$ and may be taken as $\varepsilon_3=\min\{10^{-13},g_{\det}(\overline{\mathbf{U}}_j)\}$.

According to the above definition of the limiter ${\mathcal{T}}_h$ and the Jensen's inequality for the concave function $g_{11}(\mathbf{U})$, we immediately obtain the following proposition.

\begin{proposition}\label{limit proposition}
For any $\mathbf{U}_h\in\overline{\mathbb{G}}_h^k$, one has ${\mathcal{T}}_h\mathbf{U}_h\in\widetilde{\mathbb{G}}_h^k$.
\end{proposition}

Proposition \ref{limit proposition} indicates that the limited solution (\ref{PP operator}) satisfies the condition (\ref{PP of quadrature node 2}). Note that this type of local scaling limiters keep the local conservation 
and do not destroy the high-order accuracy; see \cite{zhang2010maximum,zhang2010positivity,zhang2017positivity} for details.

Define the initial numerical solutions as $\mathbf{U}_h^0(x):={\mathcal{T}}_h{\mathcal{P}}_h\mathbf{U}(x,0)$, {where $\mathcal{P}_h$ denotes the $L^2$-projection onto the space $[\mathbb{V}_h^k]^6$}. For the WB DG schemes with the SSP-RK time discretization, if the limiter (\ref{PP operator}) is used at each RK stage, then the resulting fully discrete DG methods are positivity-preserving.

\begin{remark}
If the projected hydrostatic equilibrium solutions $\rho_h^e$ and $\mathbf{p}_h^e$ do not satisfy the condition {\rm (\ref{PP of projection})} in Theorem {\rm \ref{1D high-order positivity}}, then we can redefine $\rho_h^e,p_{11,h}^e,p_{12,h}^e,p_{22,h}^e\in\mathbb{V}_h^k$ as
\begin{equation}\label{PP limit projection}
\left(\rho_h^e(x),0,0,\frac{p_{11,h}^e(x)}{2},\frac{p_{12,h}^e(x)}{2},\frac{p_{22,h}^e(x)}{2}\right)^\top:={\mathcal{T}}_h{\mathcal{P}}_h
\left(\rho^e(x),0,0,\frac{p_{11}^e(x)}{2},\frac{p_{12}^e(x)}{2},\frac{p_{22}^e(x)}{2}\right)^\top.
\end{equation}
One can verify that $\rho_h^e$ and $\mathbf{p}_h^e$ defined by {\rm (\ref{PP limit projection})} always satisfy {\rm (\ref{PP of projection})}.
\end{remark}



\section{Positivity-preserving WB DG methods in two dimensions}\label{sec4}
In this section, we extend the proposed 1D positivity-preserving WB DG methods to two dimensions. For the sake of clarity, we shall focus on the 2D Cartesian meshes.
Assume that the target hydrostatic equilibrium solutions to be preserved are explicitly known and denoted by $\{\rho^e(x,y),u_1^e(x,y)=0,u_2^e(x,y)=0,p_{11}^e(x,y),p_{12}^e(x,y),p_{22}^e(x,y)\}$.
As discussed in Subsection \ref{stationary solutions}, one has
\begin{equation}\label{2D steady state 1}
(p_{11}^e)_x=-\frac{1}{2}\rho^e W_x, \quad p_{12}^e=C_0, \quad
(p_{22}^e)_y=-\frac{1}{2}\rho^e W_y, \quad \mathbf{u}_h^e=\mathbf{0}.
\end{equation}

\subsection{WB DG discretization}
Let ${\mathcal{K}}_h$ be a uniform Cartesian partition of the spatial domain {$\Omega$} with $K:=[x_{i-\frac{1}{2}},x_{i+\frac{1}{2}}]\times[y_{l-\frac{1}{2}},y_{l+\frac{1}{2}}]$ being a representative rectangular cell in ${\mathcal{K}}_h$. Denote $\Delta x=x_{i+\frac{1}{2}}-x_{i-\frac{1}{2}}$ and $\Delta y=y_{l+\frac{1}{2}}-y_{l-\frac{1}{2}}$. The DG numerical solutions, denoted by $\mathbf{U}_h(x,y,t)$, each component of which is an element of the finite-dimensional space of discontinuous piecewise polynomial functions:
\[
\mathbb{V}_h^k=\left\{v(x,y)\in L^2(\Omega):v(x,y)|_K\in\mathbb{P}^k(K) ~~ \forall K\in\mathcal{K}_h\right\},
\]
where $\mathbb{P}^k(K)$ denotes the space of polynomials of degree up to $k$ in cell $K$.
Then the 2D semi-discrete DG methods are {formulated} as follows: for any test function $v(x,y)\in\mathbb{V}_h^k$, {the solution $\mathbf{U}_h$ is computed by}
\begin{align}
\int_K(\mathbf{U}_h)_tv(x,y) \, dx \, dy&-\int_K\mathbf{F}(\mathbf{U}_h)v_x \, dx \, dy+\int_{y_{l-\frac{1}{2}}}^{y_{l+\frac{1}{2}}}\left(\widehat{\mathbf{F}}_{i+\frac{1}{2}}
v(x_{i+\frac{1}{2}}^{-},y)-\widehat{\mathbf{F}}_{i-\frac{1}{2}}v(x_{i-\frac{1}{2}}^{+},y)\right)dy \notag \\
&-\int_K\mathbf{G}(\mathbf{U}_h)v_y \, dx \, dy+\int_{x_{i-\frac{1}{2}}}^{x_{i+\frac{1}{2}}}\left(\widehat{\mathbf{G}}_{l+\frac{1}{2}}
v(x,y_{l+\frac{1}{2}}^{-})-\widehat{\mathbf{G}}_{l-\frac{1}{2}}v(x,y_{l-\frac{1}{2}}^{+})\right)dx \notag \\
&=\int_K\left(\mathbf{S}^x(\mathbf{U}_h)+\mathbf{S}^y(\mathbf{U}_h)\right)v(x,y) \, dx \, dy \quad \forall K\in{\mathcal{K}}_h,\label{2D semi DG}
\end{align}
where $\widehat{\mathbf{F}}_{i+\frac{1}{2}}$ and $\widehat{\mathbf{G}}_{l+\frac{1}{2}}$ represent the numerical fluxes. Let $\rho_h^e(x,y)$, $p_{11,h}^e(x,y)$, $p_{12,h}^e(x,y)$, and $p_{22,h}^e(x,y)$ be the $L^2$-projections of $\rho^e(x,y)$, $p_{11}^e(x,y)$, $p_{12}^e(x,y)$, and $p_{22}^e(x,y)$ onto $\mathbb{V}_h^k$, respectively. Since $p_{12}^e(x,y)=C_0$ is a constant, one has $p_{12,h}^e(x,y)=C_0$.

Denote
\begin{align}
\mathbf{p}_{i+\frac{1}{2},l}^{e,\pm,\ast}(y)&:=\begin{pmatrix}
                                               p_{11,i+\frac{1}{2},l}^{e,\ast}(y) & C_0 \\
                                               C_0 &
                                               p_{22,h}^e(x_{i+\frac{1}{2}}^{\pm},y)
                                             \end{pmatrix},
\quad y\in[y_{l-\frac{1}{2}},y_{l+\frac{1}{2}}], \label{2d_pstar_x}\\
\mathbf{p}_{i,l+\frac{1}{2}}^{e,\ast,\pm}(x)&:=\begin{pmatrix}
                                                p_{11,h}^e(x,y_{l+\frac{1}{2}}^{\pm}) & C_0 \\
                                                C_0 & p_{22,i,l+\frac{1}{2}}^{e,\ast}(x)
                                              \end{pmatrix},
\quad x\in[x_{i-\frac{1}{2}},x_{i+\frac{1}{2}}], \label{2d_pstar_y}
\end{align}
where
\[
p_{11,i+\frac{1}{2},l}^{e,\ast}(y):=\max\left\{{p_{11,h}^{e}(x_{i+\frac{1}{2}}^{-},y),p_{11,h}^{e}(x_{i+\frac{1}{2}}^{+},y)}\right\},
\quad
 y\in[y_{l-\frac{1}{2}},y_{l+\frac{1}{2}}],
\]
\[
p_{22,i,l+\frac{1}{2}}^{e,\ast}(x):=\max\left\{{p_{22,h}^{e}(x,y_{l+\frac{1}{2}}^{-}),p_{22,h}^{e}(x,y_{l+\frac{1}{2}}^{+})}\right\},
\quad
 x\in[x_{i-\frac{1}{2}},x_{i+\frac{1}{2}}].
\]
Then by Lemma \ref{transfor T}, there exist two upper triangular {matrices} $\mathbf{T}_{i+\frac{1}{2},l}^{\pm}(y)$ of form (\ref{upper T}) and two lower triangular matrices $\mathbf{T}_{i,l+\frac{1}{2}}^{\pm}(x)$ of form (\ref{lower T}) such that
\begin{align}
\mathbf{p}_{i+\frac{1}{2},l}^{e,\pm,\ast}(y)&=\mathbf{T}_{i+\frac{1}{2},l}^{\pm}(y)\mathbf{p}_h^e(x_{i+\frac{1}{2}}^{\pm},y)(\mathbf{T}_{i+\frac{1}{2},l}^{\pm}(y))^\top, \label{2d_pstar_T_x}\\
\mathbf{p}_{i,l+\frac{1}{2}}^{e,\ast,\pm}(x)&=\mathbf{T}_{i,l+\frac{1}{2}}^{\pm}(x)\mathbf{p}_h^e(x,y_{l+\frac{1}{2}}^{\pm})(\mathbf{T}_{i,l+\frac{1}{2}}^{\pm}(x))^\top.
\label{2d_pstar_T_y}
\end{align}

To {derive} 2D WB DG schemes, we employ the following HLLC numerical fluxes with modified solution states:
\[
\widehat{\mathbf{F}}_{i+\frac{1}{2}}=\mathbf{F}^{hllc}\left(\widehat{\mathbf{U}}_{i+\frac{1}{2},l}^{-}(y),\widehat{\mathbf{U}}_{i+\frac{1}{2},l}^{+}(y)\right), \quad
\widehat{\mathbf{G}}_{l+\frac{1}{2}}=\mathbf{G}^{hllc}\left(\widehat{\mathbf{U}}_{i,l+\frac{1}{2}}^{-}(x),\widehat{\mathbf{U}}_{i,l+\frac{1}{2}}^{+}(x)\right),
\]
where $\widehat{\mathbf{U}}_{i+\frac{1}{2},l}^{\pm}(y)$ and $\widehat{\mathbf{U}}_{i,l+\frac{1}{2}}^{\pm}(x)$ are the states obtained by some modifications to the DG numerical solutions $\mathbf{U}_{i+\frac{1}{2},l}^{\pm}(y):={\mathbf{U}_h(x_{i+\frac{1}{2}}^{\pm},y)}$ {within the interval} $y\in[y_{l-\frac{1}{2}},y_{l+\frac{1}{2}}]$ and $\mathbf{U}_{i,l+\frac{1}{2}}^{\pm}(x):={\mathbf{U}_h(x,y_{l+\frac{1}{2}}^{\pm})}$ {within the interval} $x\in[x_{i-\frac{1}{2}},x_{i+\frac{1}{2}}]$, respectively. Specifically, the modified solution states are defined as
\begin{equation}
\widehat{\rho}_{i+\frac{1}{2},l}^{\pm}(y)=\rho_{i+\frac{1}{2},l}^{\pm}(y), ~~
\widehat{\mathbf{m}}_{i+\frac{1}{2},l}^{\pm}(y)=\mathbf{T}_{i+\frac{1}{2},l}^{\pm}(y)\mathbf{m}_{i+\frac{1}{2},l}^{\pm}(y), ~~
\widehat{\mathbf{E}}_{i+\frac{1}{2},l}^{\pm}(y)=\mathbf{T}_{i+\frac{1}{2},l}^{\pm}(y)\mathbf{E}_{i+\frac{1}{2},l}^{\pm}(y)(\mathbf{T}_{i+\frac{1}{2},l}^{\pm}(y))^\top
\label{2d_transformation_x}
\end{equation}
\begin{equation}
\widehat{\rho}_{i,l+\frac{1}{2}}^{\pm}(x)=\rho_{i,l+\frac{1}{2}}^{\pm}(x), ~~
\widehat{\mathbf{m}}_{i,l+\frac{1}{2}}^{\pm}(x)=\mathbf{T}_{i,l+\frac{1}{2}}^{\pm}(x)\mathbf{m}_{i,l+\frac{1}{2}}^{\pm}(x), ~~
\widehat{\mathbf{E}}_{i,l+\frac{1}{2}}^{\pm}(x)=\mathbf{T}_{i,l+\frac{1}{2}}^{\pm}(x)\mathbf{E}_{i,l+\frac{1}{2}}^{\pm}(x)(\mathbf{T}_{i,l+\frac{1}{2}}^{\pm}(x))^\top.
\label{2d_transformation_y}
\end{equation}
Let $\{x_i^{(\mu)}\}_{\mu=1}^N$ and $\{y_l^{(\mu)}\}_{\mu=1}^N$ denote the Gauss quadrature points in the intervals $[x_{i-\frac{1}{2}},x_{i+\frac{1}{2}}]$ and $[y_{l-\frac{1}{2}},y_{l+\frac{1}{2}}]$, respectively. Then one can approximate the integrals over the cell edges in (\ref{2D semi DG}) as follows:
\begin{align}
&\int_{y_{l-\frac{1}{2}}}^{y_{l+\frac{1}{2}}}\left(\widehat{\mathbf{F}}_{i+\frac{1}{2}}v(x_{i+\frac{1}{2}}^{-},y)-\widehat{\mathbf{F}}_{i-\frac{1}{2}}v(x_{i-\frac{1}{2}}^{+},y)\right)dy \notag \\
\approx&\Delta y\sum_{\mu=1}^{N}\omega_\mu\left[\mathbf{F}^{hllc}\left(\widehat{\mathbf{U}}_{i+\frac{1}{2},l}^{-,\mu},\widehat{\mathbf{U}}_{i+\frac{1}{2},l}^{+,\mu}\right)v_{i+\frac{1}{2},l}^{-,\mu}-
\mathbf{F}^{hllc}\left(\widehat{\mathbf{U}}_{i-\frac{1}{2},l}^{-,\mu},\widehat{\mathbf{U}}_{i-\frac{1}{2},l}^{+,\mu}\right)v_{i-\frac{1}{2},l}^{+,\mu}
\right] \notag \\
=&:\left\langle\widehat{\mathbf{F}}_{i\pm\frac{1}{2}},v\right\rangle_y, \label{edge integral y} \\
&\int_{x_{i-\frac{1}{2}}}^{x_{i+\frac{1}{2}}}\left(\widehat{\mathbf{G}}_{l+\frac{1}{2}}v(x,y_{l+\frac{1}{2}}^{-})-\widehat{\mathbf{G}}_{l-\frac{1}{2}}v(x,y_{l-\frac{1}{2}}^{+})\right)dx \notag \\
\approx&\Delta x\sum_{\mu=1}^{N}\omega_\mu\left[\mathbf{G}^{hllc}\left(\widehat{\mathbf{U}}_{i,l+\frac{1}{2}}^{\mu,-},\widehat{\mathbf{U}}_{i,l+\frac{1}{2}}^{\mu,+}\right)v_{i,l+\frac{1}{2}}^{\mu,-}-
\mathbf{G}^{hllc}\left(\widehat{\mathbf{U}}_{i,l-\frac{1}{2}}^{\mu,-},\widehat{\mathbf{U}}_{i,l-\frac{1}{2}}^{\mu,+}\right)v_{i,l-\frac{1}{2}}^{\mu,+}
\right] \notag \\
=&:\left\langle\widehat{\mathbf{G}}_{l\pm\frac{1}{2}},v\right\rangle_x, \label{edge integral x}
\end{align}
where $\widehat{\mathbf{U}}_{i+\frac{1}{2},l}^{\pm,\mu}:=\widehat{\mathbf{U}}_{i+\frac{1}{2},l}^{\pm}(y_l^{(\mu)})$, $\widehat{\mathbf{U}}_{i,l+\frac{1}{2}}^{\mu,\pm}:=\widehat{\mathbf{U}}_{i,l+\frac{1}{2}}^{\pm}(x_i^{(\mu)})$, $v_{i+\frac{1}{2},l}^{\pm,\mu}:=v(x_{i+\frac{1}{2}}^{\pm},y_l^{(\mu)})$ and $v_{i,l+\frac{1}{2}}^{\mu,\pm}:=v(x_i^{(\mu)},y_{l+\frac{1}{2}}^{\pm})$. The {integrals over the cell $K$} in (\ref{2D semi DG}) are discretized as
\begin{align}
\int_K\mathbf{F}(\mathbf{U}_h)v_x \, dx \, dy&\approx|K|\sum_{1\leq\mu_1,\mu_2\leq N}\omega_{\mu_1}\omega_{\mu_2}(\mathbf{F}(\mathbf{U}_h)v_x){_{i,l}^{\mu_1,\mu_2}}=:\left\langle\mathbf{F}(\mathbf{U}_h),v_x\right\rangle_K, \label{element integral F} \\
\int_K\mathbf{G}(\mathbf{U}_h)v_y \, dx \, dy&\approx|K|\sum_{1\leq\mu_1,\mu_2\leq N}\omega_{\mu_1}\omega_{\mu_2}(\mathbf{G}(\mathbf{U}_h)v_y){_{i,l}^{\mu_1,\mu_2}}=:\left\langle\mathbf{G}(\mathbf{U}_h),v_y\right\rangle_K, \label{element integral G}
\end{align}
where $|K|=\Delta x\Delta y$ is the area of cell $K$, {and $(\cdot)_{i,l}^{\mu_1,\mu_2}$ denotes the value of the associated quantity at the point $(x_i^{(\mu_1)},y_l^{(\mu_2)})$}.

Next, we address how to discretize the integrals involving source terms in (\ref{2D semi DG}) to achieve the WB property. Denote $\mathbf{S}^x=:\left(0,S^{[2]},0,S^{[4]},S^{[5,x]},0\right)^\top$ and $\mathbf{S}^y=:\left(0,0,S^{[3]},0,S^{[5,y]},S^{[6]}\right)^\top$. {Employing a technique analogous to that used in the 1D case, we reformulate and decompose the integral of the second component of the source terms as follows:}
\begin{align*}
\int_K S^{[2]}v \, dx \, dy &= \int_K -\frac{1}{2}\rho W_xv \, dx \, dy = \int_K \frac{\rho}{\rho^e}(p_{11}^e)_xv \, dx \, dy \\
&= \int_K \left(\frac{\rho}{\rho^e} - \frac{\overline{\rho}_K}{{\overline{\rho^e}_K}} + \frac{\overline{\rho}_K}{\overline{\rho^e}_K}\right)(p_{11}^e)_xv \, dx \, dy \\
&= \int_K \left(\frac{\rho}{\rho^e} - \frac{\overline{\rho}_K}{\overline{\rho^e}_K}\right)(p_{11}^e)_xv \, dx \, dy \\
&\quad + \frac{\overline{\rho}_K}{\overline{\rho^e}_K}\left[\int_{y_{l-\frac{1}{2}}}^{y_{l+\frac{1}{2}}} \left((p_{11}^ev)(x_{i+\frac{1}{2}}^{-},y) - (p_{11}^ev)(x_{i-\frac{1}{2}}^{+},y)\right)dy - \int_K p_{11}^ev_x \, dx \, dy\right],
\end{align*}
where (\ref{2D steady state 1}) has been utilized in the second equality. This integral can then be approximated by
\begin{align}
\int_K S^{[2]}v \, dx \, dy &\approx |K|\sum_{1\leq\mu_1,\mu_2\leq N}\omega_{\mu_1}\omega_{\mu_2}\left[\left(\frac{\rho_h}{\rho_h^e}\right)_{i,l}^{\mu_1,\mu_2} - \frac{\overline{(\rho_h)}_K}{\overline{(\rho_h^e)}_K}\right]((p_{11,h}^e)_xv)_{i,l}^{\mu_1,\mu_2} \notag \\
&\quad + \frac{\overline{(\rho_h)}_K}{\overline{(\rho_h^e)}_K}\left[\Delta y\sum_{\mu=1}^N\omega_\mu\left(p_{11,i+\frac{1}{2},l}^{e,\ast,\mu}v_{i+\frac{1}{2},l}^{-,\mu} - p_{11,i-\frac{1}{2},l}^{e,\ast,\mu}v_{i-\frac{1}{2},l}^{+,\mu}\right) \right. \notag \\
&\quad \left. - |K|\sum_{1\leq\mu_1,\mu_2\leq N}\omega_{\mu_1}\omega_{\mu_2}(p_{11,h}^ev_x)_{i,l}^{\mu_1,\mu_2}\right] \notag \\
&=: \left\langle S^{[2]},v\right\rangle_K, \label{S2 v}
\end{align}
where $p_{11,i\pm\frac{1}{2},l}^{e,\ast,\mu}:=p_{11,i\pm\frac{1}{2},l}^{e,\ast}(y_l^{(\mu)})$.
Similarly, one can approximate other integrals of source terms as
\begin{align}
\int_K S^{[3]}v \, dx \, dy &\approx |K| \sum_{1 \leq \mu_1, \mu_2 \leq N} \omega_{\mu_1} \omega_{\mu_2} \left[ \left(\frac{\rho_h}{\rho_h^e}\right)_{i,l}^{\mu_1,\mu_2} - \frac{\overline{(\rho_h)}_K}{\overline{(\rho_h^e)}_K} \right] ((p_{22,h}^e)_y v)_{i,l}^{\mu_1,\mu_2} \notag \\
&\quad + \frac{\overline{\rho_h}_K}{\overline{(\rho_h^e)}_K} \left[ \Delta x \sum_{\mu=1}^N \omega_\mu \left( p_{22,i,l+\frac{1}{2}}^{e,\mu,\ast} v_{i,l+\frac{1}{2}}^{\mu,-} - p_{22,i,l-\frac{1}{2}}^{e,\mu,\ast} v_{i,l-\frac{1}{2}}^{\mu,+} \right) \right. \notag \\
&\quad \left. - |K| \sum_{1 \leq \mu_1, \mu_2 \leq N} \omega_{\mu_1} \omega_{\mu_2} (p_{22,h}^e v_y)_{i,l}^{\mu_1,\mu_2} \right] \notag \\
&=: \left\langle S^{[3]},v \right\rangle_K, \label{S3 v}
\end{align}
\begin{align}
\int_K S^{[4]}v \, dx \, dy &\approx |K| \sum_{1 \leq \mu_1, \mu_2 \leq N} \omega_{\mu_1} \omega_{\mu_2} \left[ \left(\frac{m_{1,h}}{\rho_h^e}\right)_{i,l}^{\mu_1,\mu_2} - \frac{\overline{(m_{1,h})}_K}{\overline{(\rho_h^e)}_K} \right] ((p_{11,h}^e)_x v)_{i,l}^{\mu_1,\mu_2} \notag \\
&\quad + \frac{\overline{(m_{1,h})}_K}{\overline{(\rho_h^e)}_K} \left[ \Delta y \sum_{\mu=1}^N \omega_\mu \left( p_{11,i+\frac{1}{2},l}^{e,\ast,\mu} v_{i+\frac{1}{2},l}^{-,\mu} - p_{11,i-\frac{1}{2},l}^{e,\ast,\mu} v_{i-\frac{1}{2},l}^{+,\mu} \right) \right. \notag \\
&\quad \left. - |K| \sum_{1 \leq \mu_1, \mu_2 \leq N} \omega_{\mu_1} \omega_{\mu_2} (p_{11,h}^e v_x)_{i,l}^{\mu_1,\mu_2} \right] \notag \\
&=: \left\langle S^{[4]},v \right\rangle_K, \label{S4 v}
\end{align}
\begin{align}
\int_K S^{[5,x]}v \, dx \, dy &\approx |K| \sum_{1 \leq \mu_1, \mu_2 \leq N} \omega_{\mu_1} \omega_{\mu_2} \left[ \left(\frac{m_{2,h}}{2\rho_h^e}\right)_{i,l}^{\mu_1,\mu_2} - \frac{\overline{(m_{2,h})}_K}{2\overline{(\rho_h^e)}_K} \right] ((p_{11,h}^e)_x v)_{i,l}^{\mu_1,\mu_2} \notag \\
&\quad + \frac{\overline{(m_{2,h})}_K}{2\overline{(\rho_h^e)}_K} \left[ \Delta y \sum_{\mu=1}^N \omega_\mu \left( p_{11,i+\frac{1}{2},l}^{e,\ast,\mu} v_{i+\frac{1}{2},l}^{-,\mu} - p_{11,i-\frac{1}{2},l}^{e,\ast,\mu} v_{i-\frac{1}{2},l}^{+,\mu} \right) \right. \notag \\
&\quad \left. - |K| \sum_{1 \leq \mu_1, \mu_2 \leq N} \omega_{\mu_1} \omega_{\mu_2} (p_{11,h}^e v_x)_{i,l}^{\mu_1,\mu_2} \right] \notag \\
&=: \left\langle S^{[5,x]},v \right\rangle_K, \label{S5x v}
\end{align}
\begin{align}
\int_K S^{[5,y]}v \, dx \, dy &\approx |K| \sum_{1 \leq \mu_1, \mu_2 \leq N} \omega_{\mu_1} \omega_{\mu_2} \left[ \left(\frac{m_{1,h}}{2\rho_h^e}\right)_{i,l}^{\mu_1,\mu_2} - \frac{\overline{m_{1,h}}_K}{2\overline{(\rho_h^e)}_K} \right] ((p_{22,h}^e)_y v)_{i,l}^{\mu_1,\mu_2} \notag \\
&\quad + \frac{\overline{m_{1,h}}_K}{2\overline{(\rho_h^e)}_K} \left[ \Delta x \sum_{\mu=1}^N \omega_\mu \left( p_{22,i,l+\frac{1}{2}}^{e,\mu,\ast} v_{i,l+\frac{1}{2}}^{\mu,-} - p_{22,i,l-\frac{1}{2}}^{e,\mu,\ast} v_{i,l-\frac{1}{2}}^{\mu,+} \right) \right. \notag \\
&\quad \left. - |K| \sum_{1 \leq \mu_1, \mu_2 \leq N} \omega_{\mu_1} \omega_{\mu_2} (p_{22,h}^e v_y)_{i,l}^{\mu_1,\mu_2} \right] \notag \\
&=:\left\langle S^{[5,y]},v\right\rangle_K, \label{S5y v}
\end{align}
\begin{align}
\int_K S^{[6]}v \, dx \, dy &\approx |K| \sum_{1 \leq \mu_1, \mu_2 \leq N} \omega_{\mu_1} \omega_{\mu_2} \left[ \left(\frac{m_{2,h}}{\rho_h^e}\right)_{i,l}^{\mu_1,\mu_2} - \frac{\overline{m_{2,h}}_K}{\overline{(\rho_h^e)}_K} \right] ((p_{22,h}^e)_y v)_{i,l}^{\mu_1,\mu_2} \notag \\
&\quad + \frac{\overline{m_{2,h}}_K}{\overline{(\rho_h^e)}_K} \left[ \Delta x \sum_{\mu=1}^N \omega_\mu \left( p_{22,i,l+\frac{1}{2}}^{e,\mu,\ast} v_{i,l+\frac{1}{2}}^{\mu,-} - p_{22,i,l-\frac{1}{2}}^{e,\mu,\ast} v_{i,l-\frac{1}{2}}^{\mu,+} \right) \right. \notag \\
&\quad \left. - |K| \sum_{1 \leq \mu_1, \mu_2 \leq N} \omega_{\mu_1} \omega_{\mu_2} (p_{22,h}^e v_y)_{i,l}^{\mu_1,\mu_2} \right] \notag \\
&=:\left\langle S^{[6]},v\right\rangle_K, \label{S6 v}
\end{align}
where $p_{22,i,l\pm\frac{1}{2}}^{e,\mu,\ast}:=p_{22,i,l\pm\frac{1}{2}}^{e,\ast}(x_i^{(\mu)})$. 
Substituting (\ref{edge integral y})--(\ref{S6 v}) into (\ref{2D semi DG}), one obtains the following WB DG methods with the forward Euler time discretization:
\begin{align}
\int_K\frac{\mathbf{U}_h^{\text{new}}-\mathbf{U}_h}{\Delta t}v \, dx \, dy
=&\left\langle\mathbf{F}(\mathbf{U}_h),v_x\right\rangle_K+\left\langle\mathbf{G}(\mathbf{U}_h),v_y\right\rangle_K-
\left\langle\widehat{\mathbf{F}}_{i\pm\frac{1}{2}},v\right\rangle_y-\left\langle\widehat{\mathbf{G}}_{l\pm\frac{1}{2}},v\right\rangle_x \notag \\
&+\left(0,\langle S^{[2]},v\rangle_K,0,\langle S^{[4]},v\rangle_K,\langle S^{[5,x]},v\rangle_K,0\right)^\top \notag \\
&+\left(0,0,\langle S^{[3]},v\rangle_K,0,\langle S^{[5,y]},v\rangle_K,\langle S^{[6]},v\rangle_K\right)^\top \quad {\forall K\in\mathcal{K}_h}. \label{2D WB semi DG}
\end{align}

\begin{theorem}\label{2D WB theorem}
For the 2D ten-moment Gauss closure equations with source terms, the DG methods {\rm (\ref{2D WB semi DG})} are WB for a general known hydrostatic state {\rm (\ref{2D steady state 1})}.
\end{theorem}

\begin{proof}
{Assuming $\mathbf{U}_h$ reaches the equilibrium state (\ref{2D steady state 1}), one has $\rho_h=\rho_h^e$, $\mathbf{u}_h=\mathbf{u}_h^e=\mathbf{0}$, $E_{11,h}=\frac{1}{2}p_{11,h}^e$, $E_{12,h}=\frac{1}{2}p_{12,h}^e=\frac{1}{2}C_0$, $E_{22,h}=\frac{1}{2}p_{22,h}^e$.}
Hence, {from (\ref{2d_pstar_x})--(\ref{2d_transformation_y}), one obtains}
\begin{align*}
\widehat{\mathbf{U}}_{i+\frac{1}{2},l}^{\pm,\mu}&=\left(\rho_h^e(x_{i+\frac{1}{2}}^\pm,y_l^{(\mu)}),0,0,\frac{1}{2}p_{11,i+\frac{1}{2},l}^{e,\ast,\mu},
\frac{1}{2}C_0,\frac{1}{2}p_{22,i+\frac{1}{2},l}^{\pm,\mu}\right)^\top, \\
\widehat{\mathbf{U}}_{i,l+\frac{1}{2}}^{\mu,\pm}&=\left(\rho_h^e(x_i^{(\mu)},y_{l+\frac{1}{2}}^\pm),0,0,\frac{1}{2}p_{11,i,l+\frac{1}{2}}^{\mu,\pm},
\frac{1}{2}C_0,\frac{1}{2}p_{22,i,l+\frac{1}{2}}^{e,\mu,\ast}\right)^\top.
\end{align*}
According to the contact property of HLLC flux (see Lemma \ref{contact property}), the HLLC numerical fluxes with modified solution states reduce to
\begin{align*}
\widehat{\mathbf{F}}_{i+\frac{1}{2}}&=\mathbf{F}^{hllc}\left(\widehat{\mathbf{U}}_{i+\frac{1}{2},l}^{-,\mu},\widehat{\mathbf{U}}_{i+\frac{1}{2},l}^{+,\mu}\right)
=\left(0,p_{11,i+\frac{1}{2},l}^{e,\ast,\mu},C_0,0,0,0\right)^\top, \\
\widehat{\mathbf{G}}_{l+\frac{1}{2}}&=\mathbf{G}^{hllc}\left(\widehat{\mathbf{U}}_{i,l+\frac{1}{2}}^{\mu,-},\widehat{\mathbf{U}}_{i,l+\frac{1}{2},l}^{\mu,+}\right)
=\left(0,C_0,p_{22,i,l+\frac{1}{2}}^{e,\mu,\ast},0,0,0\right)^\top.
\end{align*}
Note that the first, fourth, fifth, and sixth {components} of flux and source approximations all become zero.
{For the equation of momentum $m_1$,}
because $\rho_h=\rho_h^e$, one has
\[
\left\langle S^{[2]},v\right\rangle_K=\Delta y\sum_{\mu=1}^N\omega_\mu\left(p_{11,i+\frac{1}{2},l}^{e,\ast,\mu}v_{i+\frac{1}{2},l}^{-,\mu}-p_{11,i-\frac{1}{2},l}^{e,\ast,\mu}v_{i-\frac{1}{2},l}^{+,\mu}\right)-
|K|\sum_{1\leq\mu_1,\mu_2\leq N}\omega_{\mu_1}\omega_{\mu_2}(p_{11,h}^ev_x){_{i,l}^{\mu_1,\mu_2}}.
\]
Denote the $\ell$-th component of $\mathbf{F}$ and $\mathbf{G}$ by $F^{[\ell]}$ and $G^{[\ell]}$, respectively. Since $\mathbf{u}_h^e=0$ and $p_{12,h}^e=C_0$, one has
\begin{align*}
&\left\langle F^{[2]}(\mathbf{U}_h),v_x\right\rangle_K-\left\langle \widehat{F}_{i\pm\frac{1}{2}}^{[2]},v\right\rangle_y \\
=&|K|\sum_{1\leq\mu_1,\mu_2\leq N}\omega_{\mu_1}\omega_{\mu_2}(p_{11,h}^ev_x)_{i,l}^{\mu_1,\mu_2}-\Delta y\sum_{\mu=1}^N\omega_\mu\left(p_{11,i+\frac{1}{2},l}^{e,\ast,\mu}v_{i+\frac{1}{2},l}^{-,\mu}-p_{11,i-\frac{1}{2},l}^{e,\ast,\mu}v_{i-\frac{1}{2},l}^{+,\mu}\right) \\
=&-\left\langle S^{[2]},v\right\rangle_K,
\end{align*}
and
\begin{align*}
&\left\langle G^{[2]}(\mathbf{U}_h),v_y\right\rangle_K-\left\langle \widehat{G}_{l\pm\frac{1}{2}}^{[2]},v\right\rangle_x \\
=&|K|\sum_{1\leq\mu_1,\mu_2\leq N}\omega_{\mu_1}\omega_{\mu_2}(C_0v_y)_{i,l}^{\mu_1,\mu_2}-\Delta x\sum_{\mu=1}^N\omega_\mu\left(C_0v_{i,l+\frac{1}{2}}^{\mu,-}-C_0v_{i,l-\frac{1}{2}}^{\mu,+}\right) \\
=&C_0\int_Kv_ydxdy-C_0\int_{x_{i-\frac{1}{2}}}^{x_{i+\frac{1}{2}}}[v(x,y_{l+\frac{1}{2}}^{-})-v(x,y_{l-\frac{1}{2}}^{+})]dx \\
=&0.
\end{align*}
It follows that
\[
\left\langle F^{[2]}(\mathbf{U}_h),v_x\right\rangle_K+\left\langle G^{[2]}(\mathbf{U}_h),v_y\right\rangle_K+
\left\langle S^{[2]},v\right\rangle_K-\left\langle \widehat{F}_{i\pm\frac{1}{2}}^{[2]},v\right\rangle_y
-\left\langle \widehat{G}_{l\pm\frac{1}{2}}^{[2]},v\right\rangle_x=0.
\]
Similar derivation {for the equation of momentum $m_2$} gives
\[
\left\langle F^{[3]}(\mathbf{U}_h),v_x\right\rangle_K+\left\langle G^{[3]}(\mathbf{U}_h),v_y\right\rangle_K+
\left\langle S^{[3]},v\right\rangle_K-\left\langle \widehat{F}_{i\pm\frac{1}{2}}^{[3]},v\right\rangle_y
-\left\langle \widehat{G}_{l\pm\frac{1}{2}}^{[3]},v\right\rangle_x=0.
\]
Hence the right hand side of (\ref{2D WB semi DG}) becomes zero when $\mathbf{U}_h$ {reaches} the hydrostatic state. This implies
$\mathbf{U}_h^{\text{new}}=\mathbf{U}_h$ and completes the proof.
\end{proof}

\subsection{Positivity of first-order WB DG scheme}\label{pp analysis of 2d first-order schemes}
Let $\overline{\mathbf{U}}_K(t):=\frac{1}{|K|}\int_K\mathbf{U}_h(x,y,t) \, dx \, dy$ denote the cell average of $\mathbf{U}_h$ over cell $K$. By setting $v=1$ in the scheme (\ref{2D WB semi DG}), one obtains the evolution equations for the cell average:
\begin{equation}\label{2d_cell_average_evolution}
\overline{\mathbf{U}}_K^{\text{new}}=\overline{\mathbf{U}}_K-\frac{\Delta t}{|K|}\left[\left\langle\widehat{\mathbf{F}}_{i\pm\frac{1}{2}},1\right\rangle_y+\left\langle\widehat{\mathbf{G}}_{l\pm\frac{1}{2}},1\right\rangle_x\right]+\Delta t\overline{\mathbf{S}}_K^x+\Delta t\overline{\mathbf{S}}_K^y=:\overline{\mathbf{U}}_K+\Delta t\mathbf{L}_{K}(\mathbf{U}_h),
\end{equation}
where
\[
\overline{\mathbf{S}}_K^x:=\frac{1}{|K|}\left(0,\langle S^{[2]},1\rangle_K,0,\langle S^{[4]},1\rangle_K,\langle S^{[5,x]},1\rangle_K,0\right)^\top,
\]
\[
\overline{\mathbf{S}}_K^y:=\frac{1}{|K|}\left(0,0,\langle S^{[3]},1\rangle_K,0,\langle S^{[5,y]},1\rangle_K,\langle S^{[6]},1\rangle_K\right)^\top.
\]

If the DG polynomial degree $k=0$, then $\mathbf{U}_h(x,y,t)\equiv\overline{\mathbf{U}}_K(t)$ for all $(x,y)\in K$, and (\ref{2d_cell_average_evolution}) reduces to the evolution of the cell average in the first-order scheme which can be rewritten as
\begin{equation}
\overline{\mathbf{U}}_K^{\text{new}}=\overline{\mathbf{U}}_K+\Delta t\mathbf{L}_{K}(\mathbf{U}_h)
=\overline{\mathbf{U}}_K+\frac{\Delta t}{\Delta x}\mathbf{\Pi}_7+\frac{\Delta t}{\Delta y}\mathbf{\Pi}_8+\Delta t\alpha_{11,i}\overline{\mathbf{S}}_{1,K}+\Delta t\alpha_{22,l}\overline{\mathbf{S}}_{2,K}, \label{2d_first_order_CA_evolution}
\end{equation}
where
\begin{align*}
&\mathbf{\Pi}_7:=-\left[\mathbf{F}^{hllc}(\widehat{\overline{\mathbf{U}}}_{i,l},\widehat{\overline{\mathbf{U}}}_{i+1,l})-\mathbf{F}^{hllc}(\widehat{\overline{\mathbf{U}}}_{i-1,l},\widehat{\overline{\mathbf{U}}}_{i,l})\right], \\
&\mathbf{\Pi}_8:=-\left[\mathbf{G}^{hllc}(\widehat{\overline{\mathbf{U}}}_{i,l},\widehat{\overline{\mathbf{U}}}_{i,l+1})-\mathbf{G}^{hllc}(\widehat{\overline{\mathbf{U}}}_{i,l-1},\widehat{\overline{\mathbf{U}}}_{i,l})\right],
\\
&\alpha_{11,i}:=\frac{p_{11,i+\frac{1}{2}}^{e,\ast}-p_{11,i-\frac{1}{2}}^{e,\ast}}{\Delta x\overline{(\rho_h^e)}_K}, \quad \overline{\mathbf{S}}_{1,K}:=\left(0,\overline{\rho}_K,0,\overline{m}_{1,K},\frac{1}{2}\overline{m}_{2,K},0\right)^\top, \\
&\alpha_{22,l}:=\frac{p_{22,l+\frac{1}{2}}^{e,\ast}-p_{22,l-\frac{1}{2}}^{e,\ast}}{\Delta y\overline{(\rho_h^e)}_K}, \quad
\overline{\mathbf{S}}_{2,K}:=\left(0,0,\overline{\rho}_K,0,\frac{1}{2}\overline{m}_{1,K},\overline{m}_{2,K}\right)^\top.
\end{align*}

\begin{theorem}
If the DG polynomial degree $k=0$ and $\overline{\mathbf{U}}_K\in\mathcal{G}$ for all $K\in\mathcal{K}_h$, then
\begin{equation*}
\overline{\mathbf{U}}_K^{\text{new}}=\overline{\mathbf{U}}_K+\Delta t\mathbf{L}_K(\mathbf{U}_h)\in\mathcal{G} \quad \forall K\in\mathcal{K}_h
\end{equation*}
under the CFL-type condition
\begin{equation}\label{2D first order PP CFL}
\Delta t\left(\frac{\eta_{1,K}^\ast}{\Delta x}+\frac{\eta_{2,K}^\ast}{\Delta y}+\beta_{11,i}+\beta_{22,l}\right)\leq1,
\end{equation}
where
\begin{align*}
\eta_{1,K}^\ast&:=\frac{2}{\xi^\ast(\overline{\mathbf{U}}_K)}\max\limits_{\mathbf{U}\in\{\widehat{\overline{\mathbf{U}}}_{i,l},\widehat{\overline{\mathbf{U}}}_{i-1,l},\widehat{\overline{\mathbf{U}}}_{i+1,l}\}}\alpha_1(\mathbf{U}), \quad
\eta_{2,K}^\ast:=\frac{2}{\xi^\ast(\overline{\mathbf{U}}_K)}\max\limits_{\mathbf{U}\in\{\widehat{\overline{\mathbf{U}}}_{i,l},\widehat{\overline{\mathbf{U}}}_{i,l-1},\widehat{\overline{\mathbf{U}}}_{i,l+1}\}}\alpha_2(\mathbf{U}),
\\
\beta_{11,i}&:=\left|\alpha_{11,i}\delta_1(\overline{\mathbf{U}}_K)\right|,
\quad
\beta_{22,l}:=\left|\alpha_{22,l}\delta_2(\overline{\mathbf{U}}_K)\right|
\end{align*}
with $\alpha_2(\mathbf{U}):=|u_2|+\sqrt{\frac{3p_{22}}{\rho}}$.
\end{theorem}

\begin{proof}
Since the first component of the source term is zero, one has
\begin{align*}
\overline{\mathbf{U}}_K^{\text{new}}\cdot\mathbf{e}_1&=\overline{\mathbf{U}}_K\cdot\mathbf{e}_1+\frac{\Delta t}{\Delta x}\mathbf{\Pi}_7\cdot\mathbf{e}_1+\frac{\Delta t}{\Delta y}\mathbf{\Pi}_8\cdot\mathbf{e}_1 \\
&>\left(1-\frac{\Delta t}{\Delta x}\eta_{1,K}^\ast-\frac{\Delta t}{\Delta y}\eta_{2,K}^\ast\right)\overline{\mathbf{U}}_K\cdot\mathbf{e}_1 \\
&\geq0,
\end{align*}
where Corollary \ref{1d_first_order_corollary} has been used in the second step, and the CFL condition (\ref{2D first order PP CFL}) has been used in the last step.
Applying the linearity of $\varphi(\mathbf{U};\mathbf{z},\mathbf{u}_\ast)$ with respect to $\mathbf{U}$, one has from (\ref{2d_first_order_CA_evolution}) that
\begin{align*}
\varphi(\overline{\mathbf{U}}_K^{\text{new}};\mathbf{z},\mathbf{u}_\ast) &= \varphi(\overline{\mathbf{U}}_K;\mathbf{z},\mathbf{u}_\ast) + \frac{\Delta t}{\Delta x} \varphi(\mathbf{\Pi}_7;\mathbf{z},\mathbf{u}_\ast) + \frac{\Delta t}{\Delta y} \varphi(\mathbf{\Pi}_8;\mathbf{z},\mathbf{u}_\ast) \\
&\quad + \Delta t \alpha_{11,i} \varphi(\overline{\mathbf{S}}_{1,K};\mathbf{z},\mathbf{u}_\ast) + \Delta t \alpha_{22,l} \varphi(\overline{\mathbf{S}}_{2,K};\mathbf{z},\mathbf{u}_\ast) \\
&> \left(1 - \frac{\Delta t}{\Delta x} \eta_{1,K}^\ast - \frac{\Delta t}{\Delta y} \eta_{2,K}^\ast - \Delta t \beta_{11,i} - \Delta t \beta_{22,l}\right) \varphi(\overline{\mathbf{U}}_K;\mathbf{z},\mathbf{u}_\ast) \\
&\geq 0,
\end{align*}
where Corollaries \ref{GQL corollary} and \ref{1d_first_order_corollary} have been used in the second step, and the CFL condition (\ref{2D first order PP CFL}) has been used in the last step. Combining the above two inequalities with the GQL representation of the admissible state set (Lemma \ref{GQL}), one concludes that $\overline{\mathbf{U}}_K^{\text{new}}\in\mathcal{G}$ and completes the proof.
\end{proof}

\subsection{Positivity-preserving high-order WB DG schemes}\label{pp analysis of 2d high-order schemes}
If the degree of the DG polynomial $k \geq 1$, the high-order WB DG schemes (\ref{2D WB semi DG}) generally do not preserve the positivity. Analogous to the 1D case, we will first establish a weak positivity property of the cell averages for our schemes under appropriate conditions. Subsequently, a simple scaling limiter can be applied to ensure the physical admissibility of the DG solution polynomials at certain points of interest while maintaining both the conservation and high-order accuracy.

For the case of $k\geq1$, the evolution equations (\ref{2d_cell_average_evolution}) of the cell average can be rewritten as
\begin{align}
\overline{\mathbf{U}}_K^{\text{new}}&=\overline{\mathbf{U}}_K+\Delta t\mathbf{L}_K(\mathbf{U}_h) \notag \\
&=\overline{\mathbf{U}}_K+\frac{\Delta t}{\Delta x}\sum_{\mu=1}^{N}\omega_\mu\left(\mathbf{\Pi}_{9}^{(\mu)}+\mathbf{\Pi}_{10}^{(\mu)}\right)
+\frac{\Delta t}{\Delta y}\sum_{\mu=1}^{N}\omega_\mu\left(\mathbf{\Pi}_{11}^{(\mu)}+\mathbf{\Pi}_{12}^{(\mu)}\right)+\Delta t\overline{\mathbf{S}}_K^x+\Delta t\overline{\mathbf{S}}_K^y, \label{2d_high_order_CA_evolution}
\end{align}
where
\begin{align*}
&\mathbf{\Pi}_{9}^{(\mu)}:=-\left[\mathbf{F}^{hllc}(\widehat{\mathbf{U}}_{i+\frac{1}{2},l}^{-,\mu},\widehat{\mathbf{U}}_{i+\frac{1}{2},l}^{+,\mu})-
\mathbf{F}^{hllc}(\widehat{\mathbf{U}}_{i-\frac{1}{2},l}^{+,\mu},\widehat{\mathbf{U}}_{i+\frac{1}{2},l}^{-,\mu})\right], \\
&\mathbf{\Pi}_{10}^{(\mu)}:=-\left[\mathbf{F}^{hllc}(\widehat{\mathbf{U}}_{i-\frac{1}{2},l}^{+,\mu},\widehat{\mathbf{U}}_{i+\frac{1}{2},l}^{-,\mu})-
\mathbf{F}^{hllc}(\widehat{\mathbf{U}}_{i-\frac{1}{2},l}^{-,\mu},\widehat{\mathbf{U}}_{i-\frac{1}{2},l}^{+,\mu})\right], \\
&\mathbf{\Pi}_{11}^{(\mu)}:=-\left[\mathbf{G}^{hllc}(\widehat{\mathbf{U}}_{i,l+\frac{1}{2}}^{\mu,-},\widehat{\mathbf{U}}_{i,l+\frac{1}{2}}^{\mu,+})-
\mathbf{G}^{hllc}(\widehat{\mathbf{U}}_{i,l-\frac{1}{2}}^{\mu,+},\widehat{\mathbf{U}}_{i,l+\frac{1}{2}}^{\mu,-})\right], \\
&\mathbf{\Pi}_{12}^{(\mu)}:=-\left[\mathbf{G}^{hllc}(\widehat{\mathbf{U}}_{i,l-\frac{1}{2}}^{\mu,+},\widehat{\mathbf{U}}_{i,l+\frac{1}{2}}^{\mu,-})-
\mathbf{G}^{hllc}(\widehat{\mathbf{U}}_{i,l-\frac{1}{2}}^{\mu,-},\widehat{\mathbf{U}}_{i,l-\frac{1}{2}}^{\mu,+})\right],
\end{align*}
and
\begin{align}
\overline{\mathbf{S}}_K^x&=\sum_{1\leq\mu_1,\mu_2\leq N}\omega_{\mu_1}\omega_{\mu_2}\left(\frac{(p_{11,h}^e)_x}{\rho_h^e}\mathbf{S}_{1,h}\right)_{i,l}^{\mu_1,\mu_2}+\alpha_{11,K}\overline{\mathbf{S}}_{1,K},
\label{2d high order S1}\\
\overline{\mathbf{S}}_K^y&=\sum_{1\leq\mu_1,\mu_2\leq N}\omega_{\mu_1}\omega_{\mu_2}\left(\frac{(p_{22,h}^e)_y}{\rho_h^e}\mathbf{S}_{2,h}\right)_{i,l}^{\mu_1,\mu_2}+\alpha_{22,K}\overline{\mathbf{S}}_{2,K}
\label{2d high order S2}
\end{align}
with
\begin{align*}
&\mathbf{S}_{1,h}:=\left(0,\rho_h,0,m_{1,h},\frac{1}{2}m_{2,h},0\right)^\top, \quad
\alpha_{11,K}:=\frac{\sum_{\mu=1}^{N}\omega_\mu\left(p_{11,i+\frac{1}{2},l}^{e,\ast,\mu}-p_{11,i-\frac{1}{2},l}^{e,\ast,\mu}-p_{11,i+\frac{1}{2},l}^{e,-,\mu}+p_{11,i-\frac{1}{2},l}^{e,+,\mu}\right)}{\Delta x\overline{(\rho_h^e)}_K}, \\
&\mathbf{S}_{2,h}:=\left(0,0,\rho_h,0,\frac{1}{2}m_{1,h},m_{2,h}\right)^\top, \quad
\alpha_{22,K}:=\frac{\sum_{\mu=1}^{N}\omega_\mu\left(p_{22,i,l+\frac{1}{2}}^{e,\mu,\ast}-p_{22,i,l-\frac{1}{2}}^{e,\mu,\ast}-p_{22,i,l+\frac{1}{2}}^{e,\mu,-}+p_{22,i,l-\frac{1}{2}}^{e,\mu,+}\right)}{\Delta y\overline{(\rho_h^e)}_K}.
\end{align*}


To analyze the positivity-preserving property of high-order WB DG schemes, we first introduce the following feasible convex decomposition \cite{cui2023classic} of the 2D cell average values:
\begin{equation}\label{2D feasible decomposition}
\overline{\mathbf{U}}_K=\sum_{\mu=1}^{N}\omega_\mu\left(\omega_1^{-}\mathbf{U}_{i+\frac{1}{2},l}^{-,\mu}+\omega_1^{+}\mathbf{U}_{i-\frac{1}{2},l}^{+,\mu}
+\omega_2^{-}\mathbf{U}_{i,l+\frac{1}{2}}^{\mu,-}+\omega_2^{+}\mathbf{U}_{i,l-\frac{1}{2}}^{\mu,+}\right)+\sum_{s=1}^{S}\widetilde{\omega}_s\mathbf{U}_h(x_K^{(s)},y_K^{(s)}),
\end{equation}
which is assumed to hold for any polynomials in $\mathbb{P}^k(K)$.
Here the weights $\omega_1^{\pm},\omega_2^{\pm},\widetilde{\omega}_s>0$, $\omega_1^{-}+\omega_1^{+}+\omega_2^{-}+\omega_2^{+}+\sum_{s=1}^{S}\widetilde{\omega}_s=1$, and the points $(x_K^{(s)},y_K^{(s)})\subset K$, which will be specified later. Denote the set of all the points involved in \eqref{2D feasible decomposition} as
\begin{equation}\label{2D limit point set}
\mathbb{S}_K:=\left\{(x_{i+\frac{1}{2}}^{-},y_l^{(\mu)}),(x_{i-\frac{1}{2}}^{+},y_l^{(\mu)}),(x_i^{(\mu)},y_{l+\frac{1}{2}}^{-}),
(x_i^{(\mu)},y_{l-\frac{1}{2}}^{+})\right\}_{\mu=1}^N\cup\left\{(x_K^{(s)},y_K^{(s)})\right\}_{s=1}^S\cup\left\{(x_i^{(\mu_1)},y_l^{(\mu_2)})\right\}_{\mu_1,\mu_2=1 }^N.
\end{equation}

\begin{theorem}\label{2D positivity of scheme theorem}
Assume that the projected hydrostatic equilibrium solution satisfies
\begin{equation}\label{PP of 2D projection}
\rho_h^e(x,y)>0, \quad \mathbf{z}^\top\mathbf{p}_h^e(x,y)\mathbf{z}>0 \quad \forall \mathbf{z}\in\mathbb{R}^2\setminus\{\mathbf{0}\}, \quad \forall (x,y)\in\mathbb{S}_K, \quad \forall K\in\mathcal{K}_h,
\end{equation}
and $\mathbf{U}_h$ satisfies
\begin{equation}\label{2D nodes positivity}
\mathbf{U}_h(x,y)\in\mathcal{G} \quad \forall (x,y)\in\mathbb{S}_K, \quad \forall K\in\mathcal{K}_h.
\end{equation}
Then
\begin{equation}\label{2D positivity of high-order scheme}
\overline{\mathbf{U}}_K^{new}=\overline{\mathbf{U}}_K+\Delta t\mathbf{L}_K(\mathbf{U}_h)\in\mathcal{G} \quad \forall K\in\mathcal{K}_h
\end{equation}
under the CFL-type condition
\begin{equation}\label{2D high-order scheme PP CFL}
\Delta t\left(\max\limits_{1\leq\mu\leq N}\left\{\frac{\eta_{1,i+\frac{1}{2},l}^{\ast,-,\mu}}{\omega_1^-\Delta x},\frac{\eta_{1,i-\frac{1}{2},l}^{\ast,+,\mu}}{\omega_1^+\Delta x},\frac{\eta_{2,i,l+\frac{1}{2}}^{\ast,\mu,-}}{\omega_2^-\Delta y},
\frac{\eta_{2,i,l-\frac{1}{2}}^{\ast,\mu,+}}{\omega_2^+\Delta y}\right\}+\beta_{11,K}+\beta_{22,K}\right)\leq1,
\end{equation}
where
\begin{align*}
&\eta_{1,i+\frac{1}{2},l}^{\ast,-,\mu} := \frac{2}{\xi^\ast(\mathbf{U}_{i+\frac{1}{2},l}^{-,\mu})} \max_{\mathbf{U} \in \left\{\widehat{\mathbf{U}}_{i+\frac{1}{2},l}^{-,\mu}, \widehat{\mathbf{U}}_{i+\frac{1}{2},l}^{+,\mu}, \widehat{\mathbf{U}}_{i-\frac{1}{2},l}^{+,\mu}\right\}} \alpha_{1}(\mathbf{U}), \\
&\eta_{1,i-\frac{1}{2},l}^{\ast,+,\mu} := \frac{2}{\xi^\ast(\mathbf{U}_{i-\frac{1}{2},l}^{+,\mu})} \max_{\mathbf{U} \in \left\{\widehat{\mathbf{U}}_{i-\frac{1}{2},l}^{-,\mu}, \widehat{\mathbf{U}}_{i-\frac{1}{2},l}^{+,\mu}, \widehat{\mathbf{U}}_{i+\frac{1}{2},l}^{-,\mu}\right\}} \alpha_{1}(\mathbf{U}), \\
&\eta_{2,i,l+\frac{1}{2}}^{\ast,\mu,-} := \frac{2}{\xi^\ast(\mathbf{U}_{i,l+\frac{1}{2}}^{\mu,-})} \max_{\mathbf{U} \in \left\{\widehat{\mathbf{U}}_{i,l+\frac{1}{2}}^{\mu,-}, \widehat{\mathbf{U}}_{i,l+\frac{1}{2}}^{\mu,+}, \widehat{\mathbf{U}}_{i,l-\frac{1}{2}}^{\mu,+}\right\}} \alpha_{2}(\mathbf{U}), \\
&\eta_{2,i,l-\frac{1}{2}}^{\ast,\mu,+} := \frac{2}{\xi^\ast(\mathbf{U}_{i,l-\frac{1}{2}}^{\mu,+})} \max_{\mathbf{U} \in \left\{\widehat{\mathbf{U}}_{i,l-\frac{1}{2}}^{\mu,-}, \widehat{\mathbf{U}}_{i,l-\frac{1}{2}}^{\mu,+}, \widehat{\mathbf{U}}_{i,l+\frac{1}{2}}^{\mu,-}\right\}} \alpha_{2}(\mathbf{U}), \\
&\beta_{11,K} := \max_{1\leq\mu_1,\mu_2\leq N} \left\{ \left| \left(\frac{(p_{11,h}^e)_x}{\rho_h^e} \delta_{1,h} \right)_{i,l}^{\mu_1,\mu_2} \right| \right\} + \left| \alpha_{11,K} \delta_{1}(\overline{\mathbf{U}}_K) \right|, \\
&\beta_{22,K} := \max_{1\leq\mu_1,\mu_2\leq N} \left\{ \left| \left(\frac{(p_{22,h}^e)_y}{\rho_h^e} \delta_{2,h} \right)_{i,l}^{\mu_1,\mu_2} \right| \right\} + \left| \alpha_{22,K} \delta_{2}(\overline{\mathbf{U}}_K) \right|
\end{align*}
with $(\delta_{1,h})_{i,l}^{\mu_1,\mu_2}:=\delta_1((\mathbf{U}_h)_{i,l}^{\mu_1,\mu_2})$ and $(\delta_{2,h})_{i,l}^{\mu_1,\mu_2}:=\delta_2((\mathbf{U}_h)_{i,l}^{\mu_1,\mu_2})$.
\end{theorem}

\begin{proof}
Since $\overline{\mathbf{S}}_K^x\cdot\mathbf{e}_1=0$ and $\overline{\mathbf{S}}_K^y\cdot\mathbf{e}_1=0$, the scheme (\ref{2d_high_order_CA_evolution}) implies that
\begin{align*}
\overline{\mathbf{U}}_K^{\text{new}}\cdot\mathbf{e}_1&=\overline{\mathbf{U}}_K\cdot\mathbf{e}_1+\frac{\Delta t}{\Delta x}\sum_{\mu=1}^{N}\omega_\mu\left(\mathbf{\Pi}_{9}^{(\mu)}\cdot\mathbf{e}_1+\mathbf{\Pi}_{10}^{(\mu)}\cdot\mathbf{e}_1\right)
+\frac{\Delta t}{\Delta y}\sum_{\mu=1}^{N}\omega_\mu\left(\mathbf{\Pi}_{11}^{(\mu)}\cdot\mathbf{e}_1+\mathbf{\Pi}_{12}^{(\mu)}\cdot\mathbf{e}_1\right) \\
&>\overline{\mathbf{U}}_K\cdot\mathbf{e}_1-\frac{\Delta t}{\Delta x}\sum_{\mu=1}^{N}\omega_\mu\left(\eta_{1,i+\frac{1}{2},l}^{\ast,-,\mu}\mathbf{U}_{i+\frac{1}{2},l}^{-,\mu}\cdot\mathbf{e}_1+
\eta_{1,i-\frac{1}{2},l}^{\ast,+,\mu}\mathbf{U}_{i-\frac{1}{2},l}^{+,\mu}\cdot\mathbf{e}_1\right) \\
&\quad-\frac{\Delta t}{\Delta y}\sum_{\mu=1}^{N}\omega_\mu\left(\eta_{2,i,l+\frac{1}{2}}^{\ast,\mu,-}\mathbf{U}_{i,l+\frac{1}{2}}^{\mu,-}\cdot\mathbf{e}_1+
\eta_{2,i,l-\frac{1}{2}}^{\ast,\mu,+}\mathbf{U}_{i,l-\frac{1}{2}}^{\mu,+}\cdot\mathbf{e}_1\right) \\
&=\sum_{\mu=1}^{N}\omega_\mu\left[\left(\omega_1^{-}-\frac{\Delta t}{\Delta x}\eta_{1,i+\frac{1}{2},l}^{\ast,-,\mu}\right)\mathbf{U}_{i+\frac{1}{2},l}^{-,\mu}\cdot\mathbf{e}_1+
\left(\omega_1^{+}-\frac{\Delta t}{\Delta x}\eta_{1,i-\frac{1}{2},l}^{\ast,+,\mu}\right)\mathbf{U}_{i-\frac{1}{2},l}^{+,\mu}\cdot\mathbf{e}_1\right] \\
&\quad+\sum_{\mu=1}^{N}\omega_\mu\left[
\left(\omega_2^{-}-\frac{\Delta t}{\Delta y}\eta_{2,i,l+\frac{1}{2}}^{\ast,\mu,-}\right)\mathbf{U}_{i,l+\frac{1}{2}}^{\mu,-}\cdot\mathbf{e}_1+
\left(\omega_2^{+}-\frac{\Delta t}{\Delta y}\eta_{2,i,l-\frac{1}{2}}^{\ast,\mu,+}\right)\mathbf{U}_{i,l-\frac{1}{2}}^{\mu,+}\cdot\mathbf{e}_1\right] \\
&\quad+\sum_{s=1}^{S}\widetilde{\omega}_s\mathbf{U}_h(x_K^{(s)},y_K^{(s)})\cdot\mathbf{e}_1 \\
&>0,
\end{align*}
where Corollary \ref{1d_first_order_corollary} has been used in the second step, the 2D feasible cell average decomposition (\ref{2D feasible decomposition}) has been applied in the third step, and the assumption (\ref{2D nodes positivity}) and the CFL condition (\ref{2D high-order scheme PP CFL}) have been used in the last step.

Using the linearity of $\varphi(\mathbf{U};\mathbf{z},\mathbf{u}_\ast)$, one can derive from (\ref{2d_high_order_CA_evolution}) that
\begin{align*}
\varphi(\overline{\mathbf{U}}_K^{\text{new}};\mathbf{z},\mathbf{u}_\ast)&=\varphi(\overline{\mathbf{U}}_K;\mathbf{z},\mathbf{u}_\ast)+\frac{\Delta t}{\Delta x}\sum_{\mu=1}^{N}\omega_\mu\left(\varphi(\mathbf{\Pi}_{9}^{(\mu)};\mathbf{z},\mathbf{u}_\ast)+\varphi(\mathbf{\Pi}_{10}^{(\mu)};\mathbf{z},\mathbf{u}_\ast)\right) \\
&\quad+\frac{\Delta t}{\Delta y}\sum_{\mu=1}^{N}\omega_\mu\left(\varphi(\mathbf{\Pi}_{11}^{(\mu)};\mathbf{z},\mathbf{u}_\ast)+\varphi(\mathbf{\Pi}_{12}^{(\mu)};\mathbf{z},\mathbf{u}_\ast)\right)
+\Delta t\varphi(\overline{\mathbf{S}}_K^x;\mathbf{z},\mathbf{u}_\ast)+\Delta t\varphi(\overline{\mathbf{S}}_K^y;\mathbf{z},\mathbf{u}_\ast) \\
&>\varphi(\overline{\mathbf{U}}_K;\mathbf{z},\mathbf{u}_\ast)-\frac{\Delta t}{\Delta x}\sum_{\mu=1}^{N}\omega_\mu\left(\eta_{1,i+\frac{1}{2},l}^{\ast,-,\mu}\varphi(\mathbf{U}_{i+\frac{1}{2},l}^{-,\mu};\mathbf{z},\mathbf{u}_\ast)
+\eta_{1,i-\frac{1}{2},l}^{\ast,+,\mu}\varphi(\mathbf{U}_{i-\frac{1}{2},l}^{+,\mu};\mathbf{z},\mathbf{u}_\ast)\right) \\
&\quad-\frac{\Delta t}{\Delta y}\sum_{\mu=1}^{N}\omega_\mu\left(\eta_{2,i,l+\frac{1}{2}}^{\ast,\mu,-}\varphi(\mathbf{U}_{i,l+\frac{1}{2}}^{\mu,-};\mathbf{z},\mathbf{u}_\ast)
+\eta_{2,i,l-\frac{1}{2}}^{\ast,\mu,+}\varphi(\mathbf{U}_{i,l-\frac{1}{2}}^{\mu,+};\mathbf{z},\mathbf{u}_\ast)\right) \\
&\quad-\Delta t\max_{1\leq\mu_1,\mu_2\leq N}\left\{\left|\left(\frac{(p_{11,h}^e)_x}{\rho_h^e}\delta_{1,h}\right)_{i,l}^{\mu_1,\mu_2}\right|\right\}\sum_{1\leq\mu_1,\mu_2\leq N}\omega_{\mu_1}\omega_{\mu_2}\varphi((\mathbf{U}_h)_{i,l}^{\mu_1,\mu_2};\mathbf{z},\mathbf{u}_\ast) \\
&\quad-\Delta t\max_{1\leq\mu_1,\mu_2\leq N}\left\{\left|\left(\frac{(p_{22,h}^e)_y}{\rho_h^e}\delta_{2,h}\right)_{i,l}^{\mu_1,\mu_2}\right|\right\}\sum_{1\leq\mu_1,\mu_2\leq N}\omega_{\mu_1}\omega_{\mu_2}\varphi((\mathbf{U}_h)_{i,l}^{\mu_1,\mu_2};\mathbf{z},\mathbf{u}_\ast) \\
&\quad-\Delta t\left|\alpha_{11,K}\delta_{1}(\overline{\mathbf{U}}_K)\right|\varphi(\overline{\mathbf{U}}_{K};\mathbf{z},\mathbf{u}_\ast)-\Delta t\left|\alpha_{22,K}\delta_{2}(\overline{\mathbf{U}}_K)\right|\varphi(\overline{\mathbf{U}}_{K};\mathbf{z},\mathbf{u}_\ast) \\
&=(1-\Delta t(\beta_{11,K}+\beta_{22,K}))\varphi(\overline{\mathbf{U}}_{K};\mathbf{z},\mathbf{u}_\ast)-\frac{\Delta t}{\Delta x}\sum_{\mu=1}^{N}\omega_\mu\left(\eta_{1,i+\frac{1}{2},l}^{\ast,-,\mu}\varphi(\mathbf{U}_{i+\frac{1}{2},l}^{-,\mu};\mathbf{z},\mathbf{u}_\ast) \right. \\
&\quad\left.+\eta_{1,i-\frac{1}{2},l}^{\ast,+,\mu}\varphi(\mathbf{U}_{i-\frac{1}{2},l}^{+,\mu};\mathbf{z},\mathbf{u}_\ast)\right)
-\frac{\Delta t}{\Delta y}\sum_{\mu=1}^{N}\omega_\mu\left(\eta_{2,i,l+\frac{1}{2}}^{\ast,\mu,-}\varphi(\mathbf{U}_{i,l+\frac{1}{2}}^{\mu,-};\mathbf{z},\mathbf{u}_\ast)
+\eta_{2,i,l-\frac{1}{2}}^{\ast,\mu,+}\varphi(\mathbf{U}_{i,l-\frac{1}{2}}^{\mu,+};\mathbf{z},\mathbf{u}_\ast)\right) \\
&=\sum_{\mu=1}^{N}\omega_\mu\left[\omega_1^{-}(1-\Delta t(\beta_{11,K}+\beta_{22,K}))-\frac{\Delta t}{\Delta x}\eta_{1,i+\frac{1}{2},l}^{\ast,-,\mu}\right]\varphi(\mathbf{U}_{i+\frac{1}{2},l}^{-,\mu};\mathbf{z},\mathbf{u}_\ast) \\
&\quad+\sum_{\mu=1}^{N}\omega_\mu\left[\omega_1^{+}(1-\Delta t(\beta_{11,K}+\beta_{22,K}))-\frac{\Delta t}{\Delta x}\eta_{1,i-\frac{1}{2},l}^{\ast,+,\mu}\right]\varphi(\mathbf{U}_{i-\frac{1}{2},l}^{+,\mu};\mathbf{z},\mathbf{u}_\ast) \\
&\quad+\sum_{\mu=1}^{N}\omega_\mu\left[\omega_2^{-}(1-\Delta t(\beta_{11,K}+\beta_{22,K}))-\frac{\Delta t}{\Delta y}\eta_{2,i,l+\frac{1}{2}}^{\ast,\mu,-}\right]\varphi(\mathbf{U}_{i,l+\frac{1}{2}}^{\mu,-};\mathbf{z},\mathbf{u}_\ast) \\
&\quad+\sum_{\mu=1}^{N}\omega_\mu\left[\omega_2^{+}(1-\Delta t(\beta_{11,K}+\beta_{22,K}))-\frac{\Delta t}{\Delta y}\eta_{2,i,l-\frac{1}{2}}^{\ast,\mu,+}\right]\varphi(\mathbf{U}_{i,l-\frac{1}{2}}^{\mu,+};\mathbf{z},\mathbf{u}_\ast) \\
&\quad+\sum_{s=1}^{S}\widetilde{\omega}_s\varphi(\mathbf{U}_h(x_K^{(s)},y_K^{(s)});\mathbf{z},\mathbf{u}_\ast) \\
&>0,
\end{align*}
where Corollary \ref{GQL}, Corollary \ref{1d_first_order_corollary}, (\ref{2d high order S1}), and (\ref{2d high order S2}) have been used in the second step, the 2D feasible cell average decomposition (\ref{2D feasible decomposition}) has been applied in the fourth step, and the assumption (\ref{2D nodes positivity}) and the CFL condition (\ref{2D high-order scheme PP CFL}) have been used in the last step.
 According to the GQL representation of $\mathcal{G}$ (Lemma \ref{GQL}), one obtains  $\overline{\mathbf{U}}_K^{\text{new}}\in\mathcal{G}$ and completes the proof.

\end{proof}

Theorem \ref{2D positivity of scheme theorem} shows that (\ref{2D nodes positivity}) is a sufficient condition for the proposed high-order WB DG schemes (\ref{2D WB semi DG}) to be positivity-preserving. It can be again be enforced by a simple positivity-preserving limiter similar to 1D case; see (\ref{pp for rho})-(\ref{pp for gdet}) with the 1D point set $\mathbb{S}_j$ (\ref{Sj}) replaced by the 2D point set (\ref{2D limit point set}) accordingly. With the limiter applied at each stage of SSP-RK time steps, the resulting schemes are also positivity-preserving.

All above positivity-preserving analyses for 2D high-order WB DG schemes are based on the 2D feasible convex decomposition (\ref{2D feasible decomposition}). In this paper, for the third-order ($P^2$-based) and fourth-order ($P^3$-based) DG methods, we employ the optimal cell average decomposition (OCAD) proposed in \cite{cui2023classic}, which allows us to achieve the mildest positivity-preserving CFL condition in theory. Specifically, in (\ref{2D feasible decomposition}), we take
\[
\omega_1^{-}=\omega_1^{+}=\frac{\varpi_1}{2}, \quad \omega_2^{-}=\omega_2^{+}=\frac{\varpi_2}{2},
\]
and
\[
\{(x_K^{(s)},y_K^{(s)})\}=\begin{cases}
                            \left(x_i,y_l\pm\frac{\Delta y}{2\sqrt{3}}\sqrt{\frac{\sigma_\ast-\sigma_2}{\sigma_\ast}}\right) & \text{if } \sigma_1\geq\sigma_2 \\
                            \left(x_i\pm\frac{\Delta x}{2\sqrt{3}}\sqrt{\frac{\sigma_\ast-\sigma_1}{\sigma_\ast}},y_l\right) & \text{otherwise}
                          \end{cases},
\quad
\widetilde{\omega}_s=\frac{\sigma_\ast}{\chi},
\]
where
\[
\sigma_1=\frac{\eta_{1,\max}^{\ast,\mu}}{\Delta x}, ~\sigma_2=\frac{\eta_{2,\max}^{\ast,\mu}}{\Delta y}, ~\sigma_\ast=\max\{\sigma_1,\sigma_2\}, ~
\chi=\sigma_1+\sigma_2+2\sigma_\ast, ~\varpi_1=\frac{\sigma_1}{\chi}, ~\varpi_2=\frac{\sigma_2}{\chi}
\]
with $\eta_{1,\max}^{\ast,\mu}:=\max\{\eta_{1,i+\frac{1}{2},l}^{\ast,-,\mu},\eta_{1,i-\frac{1}{2},l}^{\ast,+,\mu}\}$ and $\eta_{2,\max}^{\ast,\mu}:=\max\{\eta_{2,i,l+\frac{1}{2}}^{\ast,\mu,-},\eta_{2,i,l-\frac{1}{2}}^{\ast,\mu,+}\}$.
Then the theoretical positivity-preserving CFL condition (\ref{2D high-order scheme PP CFL}) becomes
\[
\Delta t\left(\max\limits_{1\leq\mu\leq N}\left\{2\frac{\eta_{1,\max}^{\ast,\mu}}{\Delta x}+2\frac{\eta_{2,\max}^{\ast,\mu}}{\Delta y}+4\max\left\{\frac{\eta_{1,\max}^{\ast,\mu}}{\Delta x},\frac{\eta_{2,\max}^{\ast,\mu}}{\Delta y}\right\}\right\}+\beta_{11,K}+\beta_{22,K}\right)\leq1.
\]
We also refer the readers to \cite{cui2022optimal} for the analyses of OCAD in general polynomial spaces on Cartesian meshes.


\begin{remark}
It should be emphasized that the positivity-preserving CFL conditions {\rm \eqref{eq:CFLWKL1}, (\ref{1D high order PP CFL}), \eqref{2D first order PP CFL}}, and {\rm \eqref{2D high-order scheme PP CFL}} are sufficient (generally not necessary) conditions. In practical computations, the following efficient strategy is often adopted (refer to {\rm \cite{zhang2017positivity}} for details): initiate the simulation with an appropriate time step size, and if subsequent calculations yield nonphysical values for density or pressure in the cell averages, then the simulation is restarted from the preceding time step with the time step size reduced by half. Notably, in our numerical experiments (Section \ref{sec5}), this restarting procedure was never required.
\end{remark}

\section{Numerical experiments}\label{sec5}

This section presents a series of numerical examples in one and two dimensions to illustrate the high-order accuracy, the WB and positivity-preserving properties of our DG methods on uniform Cartesian grids. For comparison, we also include results from non-well-balanced (non-WB) DG schemes that employ straightforward source terms discretization and the original HLLC flux. Unless otherwise stated, the third-order SSP-RK time discretization, as defined in (\ref{SSPRK3}), is applied. The time step size for the 1D examples is determined by
\[
\Delta t = C_{\text{cfl}} \frac{\Delta x}{\alpha_{1,\max}}  \quad \mbox{with} \quad  \alpha_{1,\max}:=\max_j\alpha_1(\overline{\mathbf{U}}_j),
\]
except in the case of the fourth-order DG scheme during accuracy testing. Here, the time step size is adjusted to $(\Delta t)^{4/3}$ for the third-order SSP-RK time discretization and to $\frac{1}{3}(\Delta t)^{4/3}$ for the third-order SSP-MS time discretization to align with the fourth-order spatial discretization accuracy.
For the 2D examples, we calculate the time step size by
\[
\Delta t = \frac{C_{\text{cfl}}}{\frac{\alpha_{1,\max}}{\Delta x} + \frac{\alpha_{2,\max}}{\Delta y}} \quad \mbox{with} \quad \alpha_{1,\max}:=\max_K\alpha_1(\overline{\mathbf{U}}_K),~   \alpha_{2,\max}:=\max_K\alpha_2(\overline{\mathbf{U}}_K).
\]
The CFL numbers are taken as $C_{\text{cfl}} = 0.2$ for the third-order DG methods and $C_{\text{cfl}} = 0.125$ for the fourth-order DG methods, respectively.

\subsection{Example 1: Accuracy test}
In the initial example, we examine a smooth problem within the interval $[-0.25, 0.25]$, considering the potential $W(x) = x$ and the exact solution
\[
\rho = \epsilon + \sin^2(2\pi(x - t)), ~ u_1 = 1, ~ u_2 = 0, ~ p_{11} = 1 + (t - x)\left(\frac{\epsilon}{2} + \frac{1}{4}\right) + \frac{\sin(4\pi(x - t))}{16\pi}, ~ p_{12} = 0, ~ p_{22} = 1.
\]
The parameter $\epsilon$ is taken as either $10^{-2}$ or $10^{-5}$. The computations are performed up to time $t=0.1$ with the exact boundary conditions being applied. The third-order and fourth-order WB DG schemes are applied to the grid comprising $N$ uniform cells. In the milder test case with $\epsilon = 10^{-2}$, the third-order SSP-RK method, as outlined in (\ref{SSPRK3}), is employed for temporal discretization, and it is noted that the positivity-preserving limiter remains inactive. Conversely, in the low-density scenario where $\epsilon = 10^{-5}$, we utilize the third-order SSP-MS time discretization method (\ref{SSPMS3}). This challenging test necessitates the activation of the positivity-preserving limiter.
The numerical results of these tests are displayed in Tables \ref{sin_nr1} and \ref{sin_nr2}. Inspection of these results confirms that both WB DG schemes retain the anticipated order of convergence, signifying that neither the WB modification nor the positivity-preserving limiter impairs the schemes' inherent high-order accuracy.

\begin{table}[!htbp]
\centering
\caption{Example 1: The convergent results of $\rho$ and $p_{11} $at $t=0.1$ for the 1D accuracy test with $\epsilon=10^{-2}$.}
\label{sin_nr1}
\begin{center}
\small
\begin{tabular}{c|c|c|cc|cc|cc}
\hline
& & $N$ & $l^1$ error & order & $l^2$ error & order & $l^\infty$ error & order \\
\hline
\multirow{12}{*}{$P^2$} & \multirow{6}{*}{$\rho$} & 10 & 1.8240e-04 & -- & 3.7472e-04 & -- & 2.0389e-03 & -- \\
\multirow{12}{*}{} & \multirow{6}{*}{} & 20 & 2.2319e-05 & 3.03 & 4.7129e-05 & 2.99 & 2.5606e-04 & 2.99 \\
\multirow{12}{*}{} & \multirow{6}{*}{} & 40 & 2.7786e-06 & 3.01 & 5.8955e-06 & 3.00 & 3.2224e-05 & 2.99 \\
\multirow{12}{*}{} & \multirow{6}{*}{} & 80 & 3.4700e-07 & 3.00 & 7.3707e-07 & 3.00 & 4.0348e-06 & 3.00 \\
\multirow{12}{*}{} & \multirow{6}{*}{} & 160 & 4.3365e-08 & 3.00 & 9.2136e-08 & 3.00 & 5.0457e-07 & 3.00 \\
\multirow{12}{*}{} & \multirow{6}{*}{} & 320 & 5.4204e-09 & 3.00 & 1.1517e-08 & 3.00 & 6.3081e-08 & 3.00 \\
\cline{2-9}
\multirow{12}{*}{} & \multirow{6}{*}{$p_{11}$} & 10 & 7.0546e-06 & -- & 1.1958e-05 & -- & 4.5535e-05 & -- \\
\multirow{12}{*}{} & \multirow{6}{*}{} & 20 & 8.8542e-07 & 2.99 & 1.4856e-06 & 3.01 & 5.3797e-06 & 3.08 \\
\multirow{12}{*}{} & \multirow{6}{*}{} & 40 & 1.1038e-07 & 3.00 & 1.8554e-07 & 3.00 & 6.5632e-07 & 3.04 \\
\multirow{12}{*}{} & \multirow{6}{*}{} & 80 & 1.3791e-08 & 3.00 & 2.3188e-08 & 3.00 & 8.1406e-08 & 3.01 \\
\multirow{12}{*}{} & \multirow{6}{*}{} & 160 & 1.7236e-09 & 3.00 & 2.8983e-09 & 3.00 & 1.0088e-08 & 3.01 \\
\multirow{12}{*}{} & \multirow{6}{*}{} & 320 & 2.1545e-10 & 3.00 & 3.6228e-10 & 3.00 & 1.2573e-09 & 3.00  \\
\hline
\multirow{12}{*}{$P^3$} & \multirow{6}{*}{$\rho$} & 10 & 6.0870e-06 & -- & 1.3963e-05 & -- & 8.9033e-05 & -- \\
\multirow{12}{*}{} & \multirow{6}{*}{} & 20 & 3.9685e-07 & 3.94 & 8.8936e-07 & 3.97 & 5.8075e-06 & 3.94 \\
\multirow{12}{*}{} & \multirow{6}{*}{} & 30 & 7.8426e-08 & 4.00 & 1.7745e-07 & 3.98 & 1.1549e-06 & 3.98 \\
\multirow{12}{*}{} & \multirow{6}{*}{} & 40 & 2.4786e-08 & 4.00 & 5.6356e-08 & 3.99 & 3.6575e-07 & 4.00 \\
\multirow{12}{*}{} & \multirow{6}{*}{} & 50 & 1.0096e-08 & 4.03 & 2.3099e-08 & 4.00 & 1.4977e-07 & 4.00 \\
\multirow{12}{*}{} & \multirow{6}{*}{} & 60 & 4.8670e-09 & 4.00 & 1.1137e-08 & 4.00 & 7.2187e-08 & 4.00 \\
\cline{2-9}
\multirow{12}{*}{} & \multirow{6}{*}{$p_{11}$} & 10 & 2.0892e-07 & -- & 4.1250e-07 & -- & 1.9700e-06 & -- \\
\multirow{12}{*}{} & \multirow{6}{*}{} & 20 & 1.3012e-08 & 4.01 & 2.5834e-08 & 4.00 & 1.2003e-07 & 4.04 \\
\multirow{12}{*}{} & \multirow{6}{*}{} & 30 & 2.5571e-09 & 4.01 & 5.1124e-09 & 4.00 & 2.3637e-08 & 4.01 \\
\multirow{12}{*}{} & \multirow{6}{*}{} & 40 & 8.0502e-10 & 4.02 & 1.6165e-09 & 4.00 & 7.3999e-09 & 4.04 \\
\multirow{12}{*}{} & \multirow{6}{*}{} & 50 & 3.2875e-10 & 4.01 & 6.6106e-10 & 4.01 & 3.0020e-09 & 4.04 \\
\multirow{12}{*}{} & \multirow{6}{*}{} & 60 & 1.5838e-10 & 4.01 & 3.1838e-10 & 4.01 & 1.4455e-09 & 4.01  \\
\hline
\end{tabular}
\end{center}
\end{table}

\begin{table}[!htbp]
\centering
\caption{Example 1: The convergent results of $\rho$ and $p_{11} $at $t=0.1$ for the 1D accuracy test with $\epsilon=10^{-5}$.}
\label{sin_nr2}
\begin{center}
\small
\begin{tabular}{c|c|c|cc|cc|cc}
\hline
& & $N$ & $l^1$ error & order & $l^2$ error & order & $l^\infty$ error & order \\
\hline
\multirow{12}{*}{$P^2$} & \multirow{6}{*}{$\rho$} & 10 & 2.3811e-03 & -- & 6.1702e-03 & -- & 3.9297e-02 & -- \\
\multirow{12}{*}{} & \multirow{6}{*}{} & 20 & 2.7714e-04 & 3.10 & 9.5093e-04 & 2.70 & 7.6869e-03 & 2.35 \\
\multirow{12}{*}{} & \multirow{6}{*}{} & 40 & 2.7785e-06 & 6.64 & 5.8956e-06 & 7.33 & 3.2223e-05 & 7.90 \\
\multirow{12}{*}{} & \multirow{6}{*}{} & 80 & 3.4699e-07 & 3.00 & 7.3707e-07 & 3.00 & 4.0346e-06 & 3.00 \\
\multirow{12}{*}{} & \multirow{6}{*}{} & 160 & 4.3364e-08 & 3.00 & 9.2137e-08 & 3.00 & 5.0457e-07 & 3.00 \\
\multirow{12}{*}{} & \multirow{6}{*}{} & 320 & 5.4204e-09 & 3.00 & 1.1517e-08 & 3.00 & 6.3086e-08 & 3.00 \\
\cline{2-9}
\multirow{12}{*}{} & \multirow{6}{*}{$p_{11}$} & 10 & 2.4860e-05 & -- & 6.3193e-05 & -- & 2.6126e-04 & -- \\
\multirow{12}{*}{} & \multirow{6}{*}{} & 20 & 1.8917e-06 & 3.72 & 5.3252e-06 & 3.57 & 3.0949e-05 & 3.08 \\
\multirow{12}{*}{} & \multirow{6}{*}{} & 40 & 1.1046e-07 & 4.10 & 1.8895e-07 & 4.82 & 1.1419e-06 & 4.76 \\
\multirow{12}{*}{} & \multirow{6}{*}{} & 80 & 1.3793e-08 & 3.00 & 2.3302e-08 & 3.02 & 1.2573e-07 & 3.18 \\
\multirow{12}{*}{} & \multirow{6}{*}{} & 160 & 1.7236e-09 & 3.00 & 2.9008e-09 & 3.01 & 1.3266e-08 & 3.24 \\
\multirow{12}{*}{} & \multirow{6}{*}{} & 320 & 2.1548e-10 & 3.00 & 3.6240e-10 & 3.00 & 1.4183e-09 & 3.23  \\
\hline
\multirow{12}{*}{$P^3$} & \multirow{6}{*}{$\rho$} & 10 & 1.7559e-04 & -- & 6.0147e-04 & -- & 5.1763e-03 & -- \\
\multirow{12}{*}{} & \multirow{6}{*}{} & 20 & 3.9693e-07 & 8.79 & 8.8943e-07 & 9.40 & 5.8077e-06 & 9.80 \\
\multirow{12}{*}{} & \multirow{6}{*}{} & 30 & 7.8434e-08 & 4.00 & 1.7746e-07 & 3.98 & 1.1551e-06 & 3.98 \\
\multirow{12}{*}{} & \multirow{6}{*}{} & 40 & 2.4790e-08 & 4.00 & 5.6359e-08 & 3.99 & 3.6591e-07 & 4.00 \\
\multirow{12}{*}{} & \multirow{6}{*}{} & 50 & 1.0100e-08 & 4.02 & 2.3102e-08 & 4.00 & 1.4999e-07 & 4.00 \\
\multirow{12}{*}{} & \multirow{6}{*}{} & 60 & 4.8671e-09 & 4.00 & 1.1137e-08 & 4.00 & 7.2202e-08 & 4.01 \\
\cline{2-9}
\multirow{12}{*}{} & \multirow{6}{*}{$p_{11}$} & 10 & 5.6708e-05 & -- & 2.1415e-04 & -- & 1.9702e-03 & -- \\
\multirow{12}{*}{} & \multirow{6}{*}{} & 20 & 1.2853e-08 & 12.11 & 2.5810e-08 & 13.02 & 1.1928e-07 & 14.01 \\
\multirow{12}{*}{} & \multirow{6}{*}{} & 30 & 2.5361e-09 & 4.00 & 5.0896e-09 & 4.00 & 2.3344e-08 & 4.02 \\
\multirow{12}{*}{} & \multirow{6}{*}{} & 40 & 8.0148e-10 & 4.00 & 1.6098e-09 & 4.00 & 7.5354e-09 & 3.93 \\
\multirow{12}{*}{} & \multirow{6}{*}{} & 50 & 3.3040e-10 & 3.97 & 6.6666e-10 & 3.95 & 3.1988e-09 & 3.84 \\
\multirow{12}{*}{} & \multirow{6}{*}{} & 60 & 1.5842e-10 & 4.03 & 3.1805e-10 & 4.06 & 1.4696e-09 & 4.27  \\
\hline
\end{tabular}
\end{center}
\end{table}

\subsection{Example 2: 1D WB test}
This subsection evaluates the WB property of our proposed schemes. We consider the following three test cases \cite{meena2018well}:
\begin{itemize}
\item Polytropic case: The hydrostatic equilibrium solution is given by
\[
\rho(x)=\rho_0\left(1+\frac{\nu-1}{\alpha\nu\rho_0^{\nu-1}}\left(\frac{W_0-W(x)}{2}\right)\right)^{\frac{1}{\nu-1}}, ~ p_{11}=\alpha\rho^\nu, ~ p_{12}=0.5, ~ p_{22}=1
\]
with parameters $\rho_0=1$, $\alpha=1$, $\nu=1.2$, and $W_0=0$.

\item Isentropic case: The hydrostatic equilibrium solution is given by
\begin{equation}\label{isentropic state}
\rho(x)=\rho_0\left(1+\frac{1}{3\alpha\rho_0^{2}}\left(W_0-W(x)\right)\right)^{\frac{1}{2}}, ~ p_{11}=\alpha\rho^3, ~ p_{12}=0, ~ p_{22}=1
\end{equation}
with parameters $\rho_0=1$, $\alpha=1$, and $W_0=0$.

\item Isothermal case: The hydrostatic equilibrium solution is given by
\[
\rho(x)=\rho_0\exp\left(\frac{-W(x)}{2RT_0}\right), ~ p_{11}=p_{11,0}\exp\left(\frac{-W(x)}{2RT_0}\right), ~ p_{12}=0.5, ~ p_{22}=1,
\]
where $p_{11,0}=\rho_0RT_0$, $\rho_0=1$, $R=1$, and $T_0=1$.
\end{itemize}

The computational domain is defined over $[0,2]$ with a potential $W(x)=\frac{x^2}{2}$. We apply the third-order and fourth-order WB DG schemes to simulate these cases until a final time $t=2$. The simulations are conducted on meshes consisting of $N$ cells, with $N$ set to 50 and 100. For comparative analysis, we also include the results of the non-WB DG schemes. In these tests, the positivity-preserving limiter was not triggered, as the conservative variables remained significantly distant from the boundary of $\mathcal{G}$.
Tables \ref{wb_nr1}--\ref{wb_nr3} present the errors in density $\rho$ and pressure component $p_{11}$ for the three test cases, respectively. It is evident that the WB schemes, even on a coarse mesh, achieve errors reaching the level of machine precision, thereby validating their WB property. In contrast, the errors from the non-WB scheme are noticeably larger.

As discussed in Remark \ref{t1t2_high_order}, the transformation matrices $\mathbf{T}_{j+\frac{1}{2}}^{\pm}$ approximate the identity matrix to at least  $(k+1)$-order accuracy. To substantiate this claim, we evaluate the $l^1$ errors of the matrices' first-row elements, denoted as $t_1$ and $t_2$. The results, presented in Table \ref{wb_nr4} for the polytropic test case, showcase the expected third-order and fifth-order convergence rates for the third-order and fourth-order WB DG methods, respectively, corroborating the assertions made in Remark \ref{t1t2_high_order}.

\begin{table}[!htbp]
\centering
\caption{Example 2: The errors in $\rho$ and $p_{11}$ at $t=2$ for the polytropic test case.}
\label{wb_nr1}
\begin{center}
\small
\begin{tabular}{c|c|c|ccc|ccc}
\hline
\multicolumn{3}{c|}{} & \multicolumn{3}{c|}{$\rho$}  & \multicolumn{3}{c}{$p_{11}$} \\
\hline
& Scheme & $N$ & $l^1$ error & $l^2$ error & $l^\infty$ error & $l^1$ error & $l^2$ error & $l^\infty$ error \\
\hline
\multirow{4}{*}{$P^2$} & \multirow{2}{*}{WB} & 50 & 2.5487e-15 & 2.3583e-15 & 6.6613e-15 & 8.8098e-15 & 6.2570e-15 & 5.6621e-15 \\
\multirow{4}{*}{} & \multirow{2}{*}{} & 100 & 6.8565e-15 & 5.8222e-15 & 3.1752e-14 & 1.7910e-14 & 1.2735e-14 & 1.0769e-14 \\
\cline{2-9}
\multirow{4}{*}{} & \multirow{2}{*}{non-WB} & 50 & 1.6227e-07 & 1.6539e-07 & 4.4324e-07 & 2.0180e-07 & 1.6331e-07 & 2.5919e-07  \\
\multirow{4}{*}{} & \multirow{2}{*}{} & 100 & 2.0242e-08 & 2.0692e-08 & 5.5597e-08 & 2.5214e-08 & 2.0414e-08 & 3.2406e-08  \\
\hline
\multirow{4}{*}{$P^3$} & \multirow{2}{*}{WB} & 50 & 1.3543e-15 & 1.9353e-15 & 2.2093e-14 & 1.0158e-14 & 8.4518e-15 & 8.7708e-15 \\
\multirow{4}{*}{} & \multirow{2}{*}{} & 100 & 4.1026e-15 & 4.4694e-15 & 3.0365e-14 & 2.1123e-14 & 1.7562e-14 & 1.7208e-14 \\
\cline{2-9}
\multirow{4}{*}{} & \multirow{2}{*}{non-WB} & 50 & 3.2469e-10 & 3.5759e-10 & 1.6018e-09 & 3.6862e-10 & 3.7784e-10 & 9.5206e-10  \\
\multirow{4}{*}{} & \multirow{2}{*}{} & 100 & 2.0275e-11 & 2.2375e-11 & 1.0210e-10 & 2.3040e-11 & 2.3621e-11 & 5.9537e-11  \\
\hline
\end{tabular}
\end{center}
\end{table}

\begin{table}[!htbp]
\centering
\caption{Example 2: The errors in $\rho$ and $p_{11}$ at $t=2$ for the isentropic test case.}
\label{wb_nr2}
\begin{center}
\small
\begin{tabular}{c|c|c|ccc|ccc}
\hline
\multicolumn{3}{c|}{} & \multicolumn{3}{c|}{$\rho$}  & \multicolumn{3}{c}{$p_{11}$} \\
\hline
&  Scheme & $N$ & $l^1$ error & $l^2$ error & $l^\infty$ error & $l^1$ error & $l^2$ error & $l^\infty$ error \\
\hline
\multirow{4}{*}{$P^2$} & \multirow{2}{*}{WB} & 50 & 4.0434e-15 & 3.7646e-15 & 1.9318e-14 & 9.7475e-15 & 6.9757e-15 & 6.7724e-15 \\
\multirow{4}{*}{} & \multirow{2}{*}{} & 100 & 1.0363e-14 & 9.0394e-15 & 1.7764e-14 & 1.8458e-14 & 1.3259e-14 & 1.2101e-14 \\
\cline{2-9}
\multirow{4}{*}{} & \multirow{2}{*}{non-WB} & 50 & 6.6993e-07 & 1.2876e-06 & 9.4571e-06 & 2.2824e-07 & 2.2671e-07 & 7.0178e-07  \\
\multirow{4}{*}{} & \multirow{2}{*}{} & 100 & 8.4264e-08 & 1.6239e-07 & 1.2514e-06 & 2.8446e-08 & 2.8342e-08 & 9.1753e-08 \\
\hline
\multirow{4}{*}{$P^3$} & \multirow{2}{*}{WB} & 50 & 1.6867e-15 & 2.0815e-15 & 1.3767e-14 & 1.0268e-14 & 8.5591e-15 & 8.9928e-15 \\
\multirow{4}{*}{} & \multirow{2}{*}{} & 100 & 4.6378e-15 & 5.1236e-15 & 3.6637e-14 & 2.1659e-14 & 1.8013e-14 & 1.7208e-14 \\
\cline{2-9}
\multirow{4}{*}{} & \multirow{2}{*}{non-WB} & 50 & 2.6992e-09 & 8.3779e-09 & 9.9931e-08 & 7.7321e-10 & 9.8609e-10 & 5.1808e-09  \\
\multirow{4}{*}{} & \multirow{2}{*}{} & 100 & 1.7072e-10 & 5.2824e-10 & 6.7417e-09 & 4.8240e-11 & 6.1808e-11 & 3.4614e-10 \\
\hline
\end{tabular}
\end{center}
\end{table}

\begin{table}[!htbp]
\centering
\caption{Example 2: The errors in $\rho$ and $p_{11}$ at $t=2$ for the isothermal test case.}
\label{wb_nr3}
\begin{center}
\small
\begin{tabular}{c|c|c|ccc|ccc}
\hline
\multicolumn{3}{c|}{} & \multicolumn{3}{c|}{$\rho$}  & \multicolumn{3}{c}{$p_{11}$} \\
\hline
& Scheme & $N$ & $l^1$ error & $l^2$ error & $l^\infty$ error & $l^1$ error & $l^2$ error & $l^\infty$ error \\
\hline
\multirow{4}{*}{$P^2$} & \multirow{2}{*}{WB} & 50 & 2.3666e-15 & 2.3217e-15 & 1.5987e-14  & 8.7103e-15 & 6.2063e-15 & 5.8842e-15 \\
\multirow{4}{*}{} & \multirow{2}{*}{} & 100 & 6.9348e-15 & 5.9350e-15 & 3.0420e-14 & 1.8576e-14 & 1.3212e-14 & 1.1657e-14  \\
\cline{2-9}
\multirow{4}{*}{} & \multirow{2}{*}{non-WB} & 50 & 1.9706e-07 & 1.6853e-07 & 3.7095e-07 & 1.9683e-07 & 1.5931e-07 & 2.6034e-07  \\
\multirow{4}{*}{} & \multirow{2}{*}{} & 100 & 2.4612e-08 & 2.1067e-08 & 4.6444e-08 & 2.4597e-08 & 1.9914e-08 & 3.2538e-08 \\
\hline
\multirow{4}{*}{$P^3$} & \multirow{2}{*}{WB} & 50 & 1.5484e-15 & 2.6601e-15 & 4.2577e-14  & 1.0268e-14 & 8.5353e-15 & 8.2157e-15 \\
\multirow{4}{*}{} & \multirow{2}{*}{} & 100 & 4.4496e-15 & 4.8082e-15 & 3.2474e-14 & 2.1371e-14 & 1.7765e-14 & 1.7542e-14  \\
\cline{2-9}
\multirow{4}{*}{} & \multirow{2}{*}{non-WB} & 50 & 4.5348e-10 & 4.6727e-10 & 1.2341e-09 & 4.2239e-10 & 4.3134e-10 & 1.1423e-09  \\
\multirow{4}{*}{} & \multirow{2}{*}{} & 100 & 2.8323e-11 & 2.9211e-11 & 7.7326e-11 & 2.6398e-11 & 2.6964e-11 & 7.1436e-11 \\
\hline
\end{tabular}
\end{center}
\end{table}


\begin{table}[!htbp]
\centering
\caption{Example 2: The convergent results of $t_1$ and $t_2$ for the polytropic test case.}
\label{wb_nr4}
\begin{center}
\small
\begin{tabular}{c|cc|cc|cc|cc}\hline
& \multicolumn{4}{c|}{$P^2$} & \multicolumn{4}{c}{$P^3$} \\
\hline
 & \multicolumn{2}{c|}{$t_1$} & \multicolumn{2}{c|}{$t_2$} & \multicolumn{2}{c|}{$t_1$} & \multicolumn{2}{c}{$t_2$} \\
\hline
 $N$ & $l^1$ error & order & $l^1$ error & order & $l^1$ error & order & $l^1$ error & order \\
\hline
 10 &   3.2605e-05 & -- &      1.6303e-05 & -- & 1.7757e-07 & -- & 8.8785e-08 & -- \\
 20 &   4.1712e-06 & 2.97 &    2.0856e-06 & 2.97 & 6.0479e-09 & 4.88 & 3.0239e-09 & 4.88\\
 40 &   5.3088e-07 & 2.97 &    2.6544e-07 & 2.97 & 1.9719e-10 & 4.94 & 9.8593e-11 & 4.94 \\
 80 &   6.7076e-08 & 2.98 &    3.3538e-08 & 2.98 & 6.2933e-12 & 4.97 & 3.1466e-12 & 4.97 \\
160 &   8.4333e-09 & 2.99 &    4.2167e-09 & 2.99 & 1.9876e-13 & 4.98 & 9.9378e-14 & 4.98 \\
320 &   1.0574e-09 & 3.00 &    5.2868e-10 & 3.00 & 6.1968e-15 & 5.00 & 3.0984e-15 & 5.00 \\
\hline
\end{tabular}
\end{center}
\end{table}

\subsection{Example 3: Small perturbation test for isentropic case}
The WB schemes are expected to outperform non-WB schemes in accurately capturing solutions near the steady state, particularly on coarser meshes. To substantiate this, we consider the isentropic hydrostatic state \eqref{isentropic state}, and add a small periodic velocity perturbation at the left boundary:
\[
u_1(0,t)=A\sin(4\pi t), \quad u_2(0,t)=A\sin(4\pi t)
\]
with $A=10^{-7}$. This experimental setup extends Example 1 in \cite{wu2021uniformly}. Figure \ref{test9_nr} displays the perturbation of variables at time $t=1$, computed using both the third-order WB and the non-WB DG methods on a mesh consisting of 50 uniform cells. For reference, we also present the solutions obtained by the third-order WB DG method on a much finer mesh of 10,000 cells. The positivity-preserving limiter is not necessary for this mild example.
 As one can see, the WB scheme agrees well with the reference solutions on the coarse mesh.
 In contrast, the non-WB scheme exhibits a notable deviation from these reference solutions.
This underlines the superiority of WB schemes in effectively and accurately capturing solutions near the steady state on coarser meshes.

%

\begin{figure}[!htbp]
\subfigure[$\rho$ perturbation]{
\begin{minipage}[c]{0.3\linewidth}
\centering
\includegraphics[width=5cm]{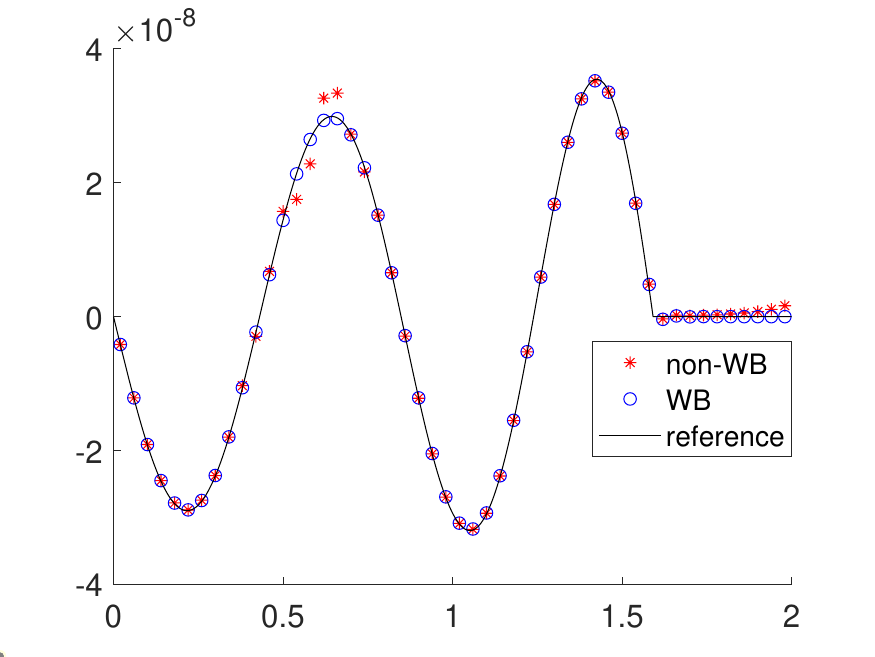}
\end{minipage}
}
\subfigure[$u_1$ perturbation]{
\begin{minipage}[c]{0.3\linewidth}
\centering
\includegraphics[width=5cm]{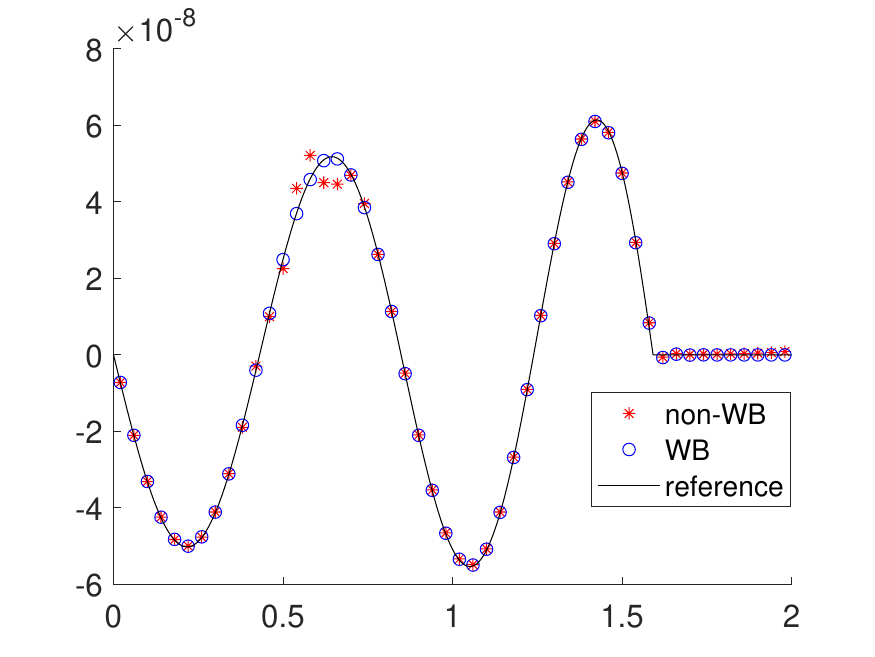}
\end{minipage}
}
\subfigure[$u_2$ perturbation]{
\begin{minipage}[c]{0.3\linewidth}
\centering
\includegraphics[width=5cm]{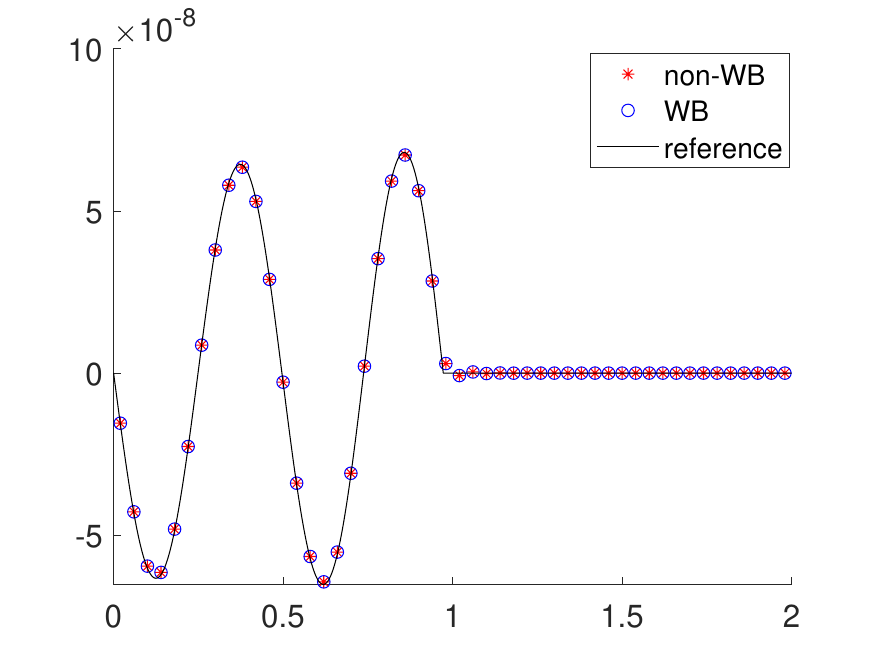}
\end{minipage}
}
\subfigure[$p_{11}$ perturbation]{
\begin{minipage}[c]{0.3\linewidth}
\centering
\includegraphics[width=5cm]{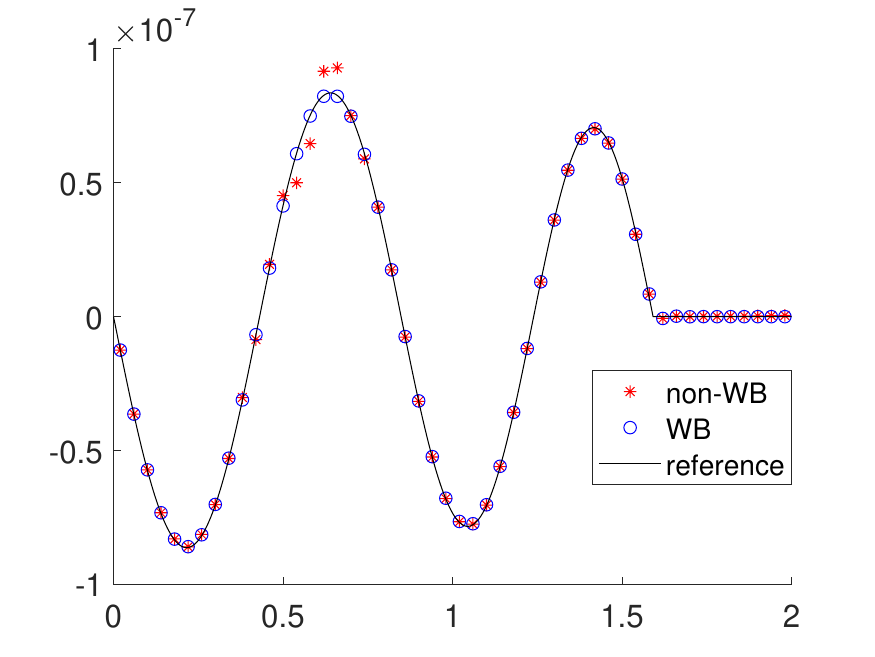}
\end{minipage}
}
\subfigure[$p_{12}$ perturbation]{
\begin{minipage}[c]{0.3\linewidth}
\centering
\includegraphics[width=5cm]{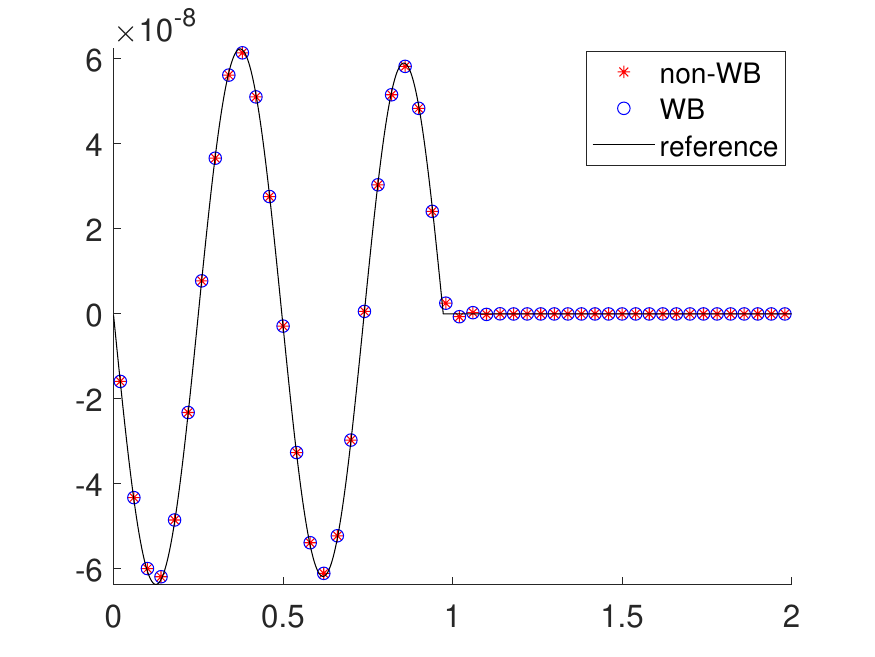}
\end{minipage}
}
\subfigure[$p_{22}$ perturbation]{
\begin{minipage}[c]{0.3\linewidth}
\centering
\includegraphics[width=5cm]{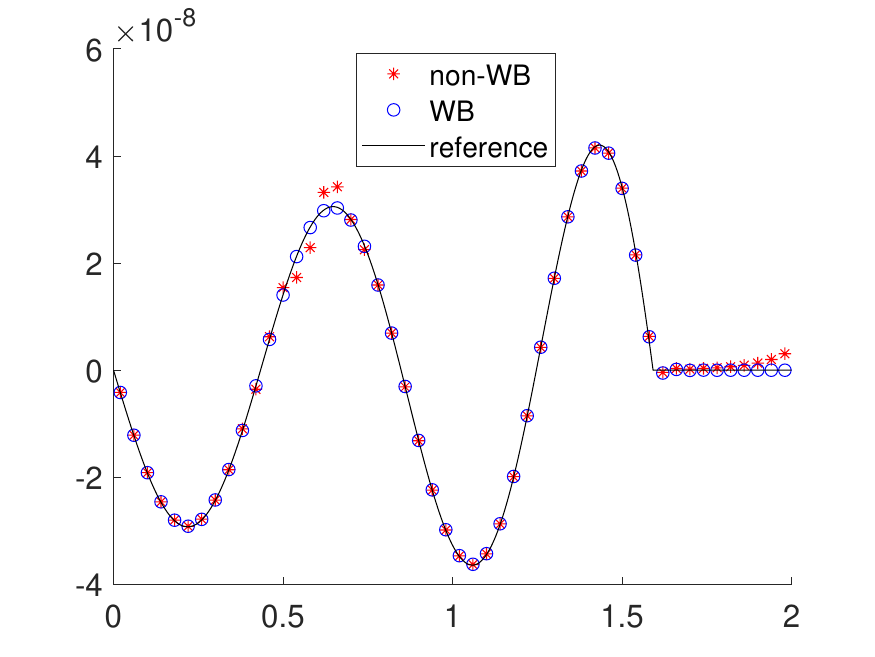}
\end{minipage}
}
\centering
\caption{Example 3: Small perturbation wave traveling up the isentropic hydrostatic state. The results of the third-order WB and non-WB schemes are obtained with 50 uniform cells. The reference solutions are obtained by the third-order WB DG scheme on 10,000 cells.}
\label{test9_nr}
\end{figure}

\subsection{Example 4: Small perturbation test for isothermal case}
In this test \cite{meena2018well}, we consider a small perturbation of isothermal hydrostatic solution in the pressure component $p_{11}$, i.e., the initial profile is set as
\begin{align*}
&\rho(x)=\rho_0\exp\left(\frac{-x}{2RT_0}\right), ~~ u_1=u_2=0, \\
&p_{11}=\rho_0RT_0\exp\left(\frac{-x}{2RT_0}\right)+\epsilon\exp\left(\frac{-100(x-0.5)^2}{2RT_0}\right), ~~ p_{12}=0, ~~ p_{22}=1,
\end{align*}
where $\rho_0=1$, $R=1$, and $T_0=1$.
The perturbation parameter $\epsilon$ is set as $10^{-6}$, $10^{-8}$, and $10^{-10}$, respectively. The third-order WB and non-WB DG schemes are applied to solve these problems up to $t=0.25$ on a mesh consisting of 50 cells. The positivity-preserving limiter is not activated in this test. The comparative analyses are illustrated in Figures \ref{test7_nr1} and \ref{test7_nr3}, where the reference solutions are obtained by the WB scheme with 10,000 cells. From Figure \ref{test7_nr1}, one can see that for the perturbation parameter $\epsilon=10^{-6}$, both WB and non-WB schemes provide satisfactory results. However, for smaller $\epsilon=10^{-8}$ and $\epsilon=10^{-10}$, the WB scheme is notably more accurate than the non-WB scheme. Furthermore, Figure \ref{test7_nr3} shows that, for the smallest perturbation of $\epsilon=10^{-10}$, the non-WB scheme fails to achieve good resolution unless the mesh is refined to 400 cells.


\begin{figure}[!htbp]
\subfigure[$\epsilon=10^{-6}$]{
\begin{minipage}[c]{0.3\linewidth}
\centering
\includegraphics[width=5cm]{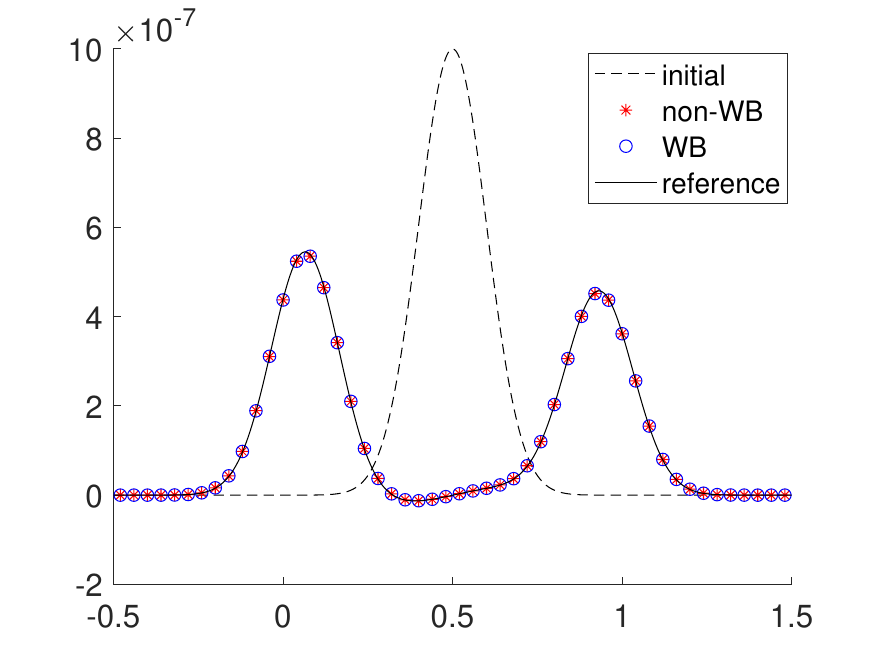}
\end{minipage}
}
\subfigure[$\epsilon=10^{-8}$]{
\begin{minipage}[c]{0.3\linewidth}
\centering
\includegraphics[width=5cm]{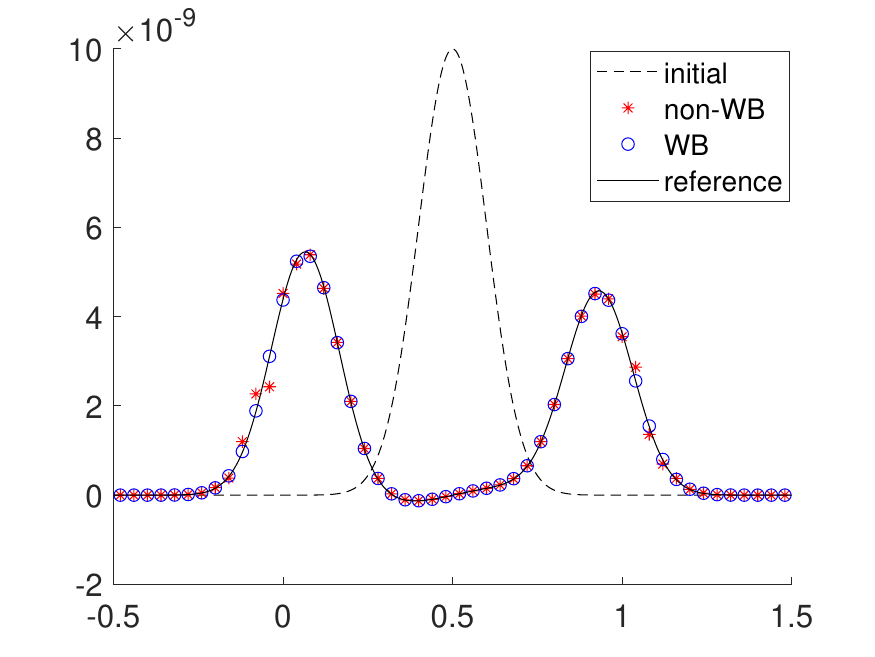}
\end{minipage}
}
\subfigure[$\epsilon=10^{-10}$]{
\begin{minipage}[c]{0.3\linewidth}
\centering
\includegraphics[width=5cm]{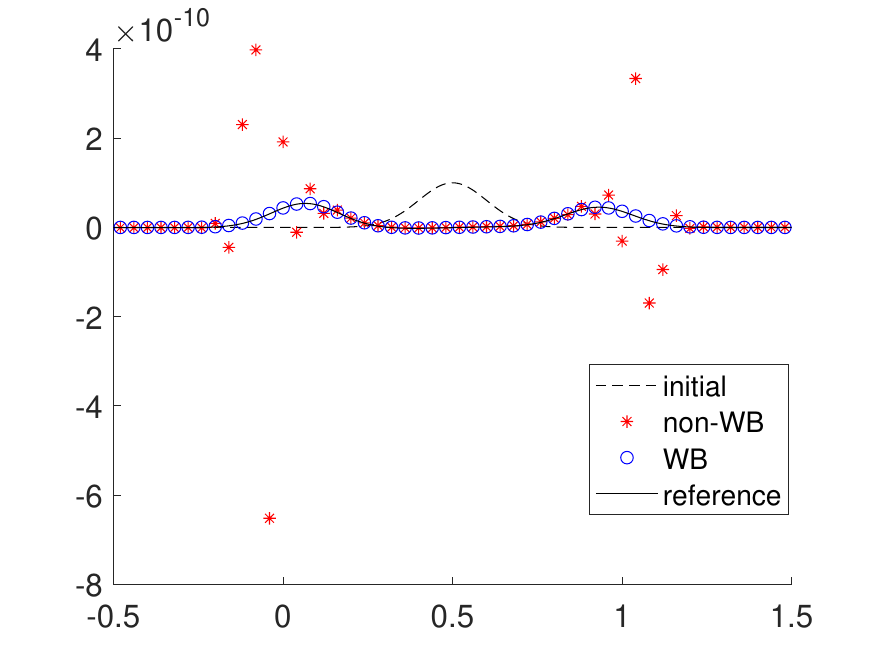}
\end{minipage}
}
\centering
\caption{Example 4: The comparison plots of $p_{11}$ perturbation on a mesh of 50 cells. The reference solutions are obtained by the WB scheme with 10,000 cells.}
\label{test7_nr1}
\end{figure}


\begin{figure}[!htbp]
\subfigure[100 cells]{
\begin{minipage}[c]{0.3\linewidth}
\centering
\includegraphics[width=5cm]{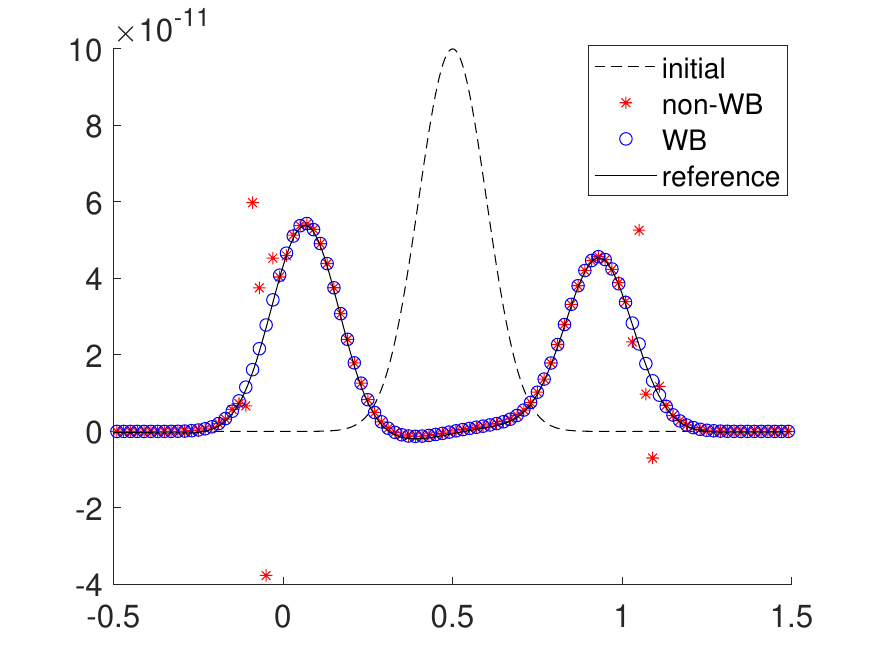}
\end{minipage}
}
\subfigure[200 cells]{
\begin{minipage}[c]{0.3\linewidth}
\centering
\includegraphics[width=5cm]{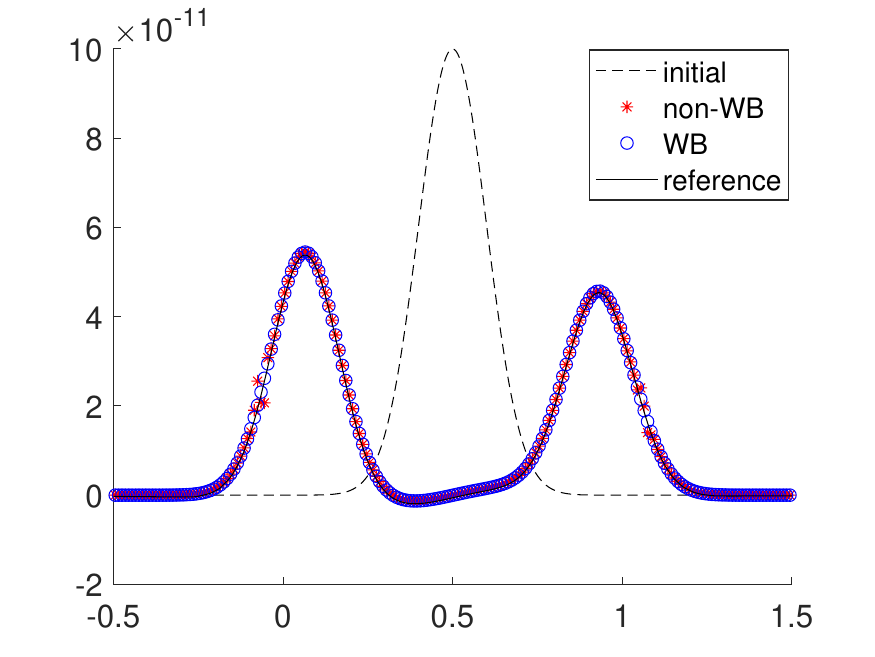}
\end{minipage}
}
\subfigure[400 cells]{
\begin{minipage}[c]{0.3\linewidth}
\centering
\includegraphics[width=5cm]{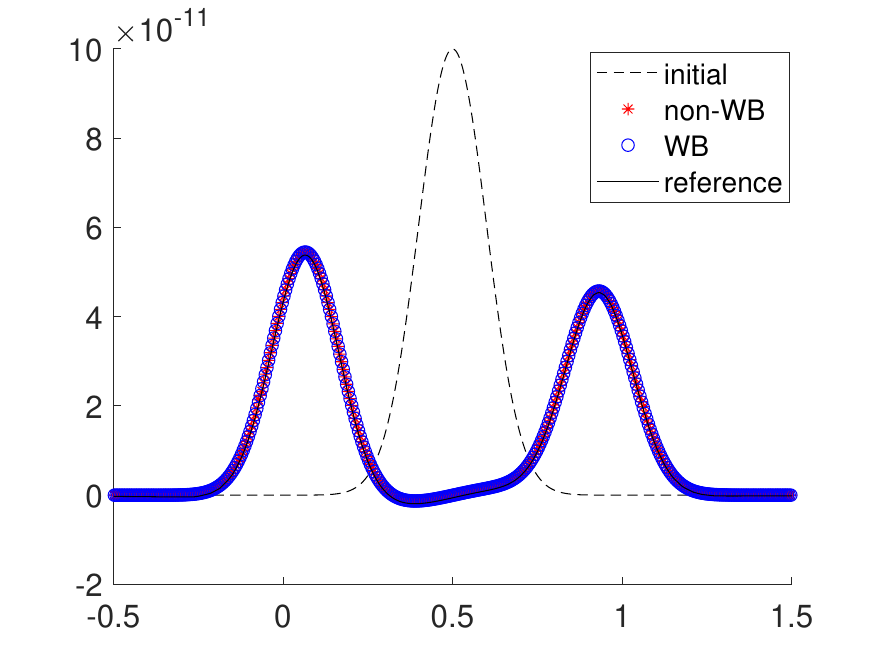}
\end{minipage}
}
\centering
\caption{Example 4: The comparison plots of $p_{11}$ perturbation of $\epsilon=10^{-10}$. The reference solutions are obtained by the WB scheme with 10,000 cells.}
\label{test7_nr3}
\end{figure}


\subsection{Example 5: Sod shock-tube problem}
To test the shock-capturing ability of the proposed WB schemes, we consider the Sod shock-tube problem with a non-zero source term \cite{meena2018well}. The initial solution with a discontinuity at $x=0.5$ over the interval $[0,1]$ is given by
\[
(\rho,u_1,u_2,p_{11},p_{12},p_{22})=\begin{cases}
                                      (1,0,0,2,0.05,0.6), & \text{if } x\leq0.5, \\
                                      (0.125,0,0,0.2,0.1,0.2), & \text{if } x>0.5.
                                    \end{cases}
\]
 The potential $W(x)=x$ introduces a non-trivial source, and the reflection boundary conditions are imposed. The third-order WB and non-WB DG schemes are employed to simulate this problem up to time $t=0.125$ on a mesh consisting of 400 cells. The positivity-preserving limiter is turned off for this problem.
 The results are presented in Figure \ref{test11_nr}, where the reference solutions are computed by the third-order WB DG scheme with 10,000 cells. At $t=0.125$, the solution to this problem encompasses a left-going rarefaction wave extending across the region $(0.188, 0.406)$, a left-moving shear wave near $x=0.489$, a right-going shear wave near $x=0.827$, and a right-moving shock wave at approximately  $x=0.898$, separated by a contact discontinuity roughly at $x=0.603$. One can see that the WB scheme effectively captures all the waves, comparable to the performance of the non-WB scheme. This demonstrates that our WB modification does not compromise the capability of resolving complex wave structures including discontinuities.

\begin{figure}[!htbp]
\subfigure[$\rho$]{
\begin{minipage}[c]{0.3\linewidth}
\centering
\includegraphics[width=5cm]{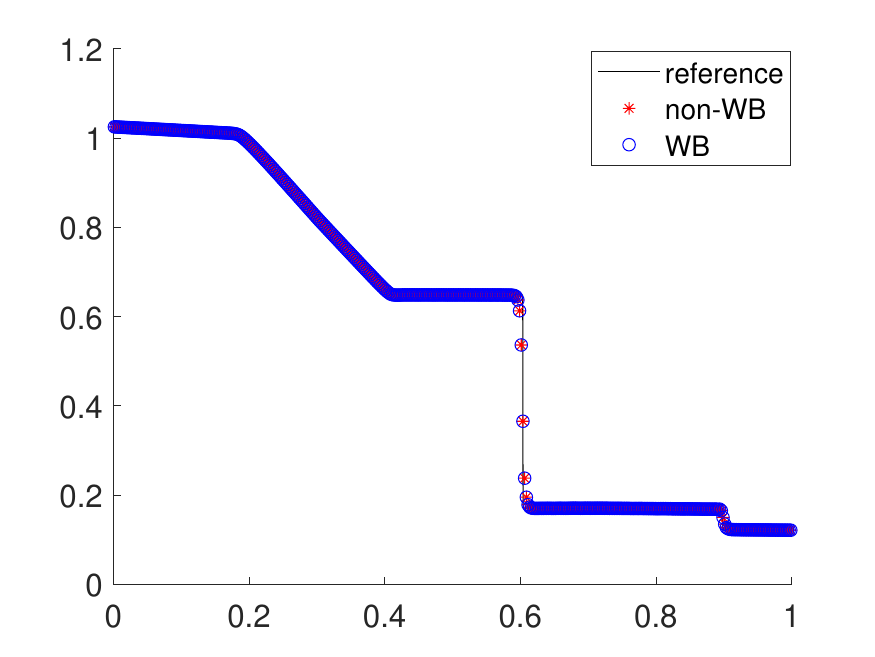}
\end{minipage}
}
\subfigure[$u_1$]{
\begin{minipage}[c]{0.3\linewidth}
\centering
\includegraphics[width=5cm]{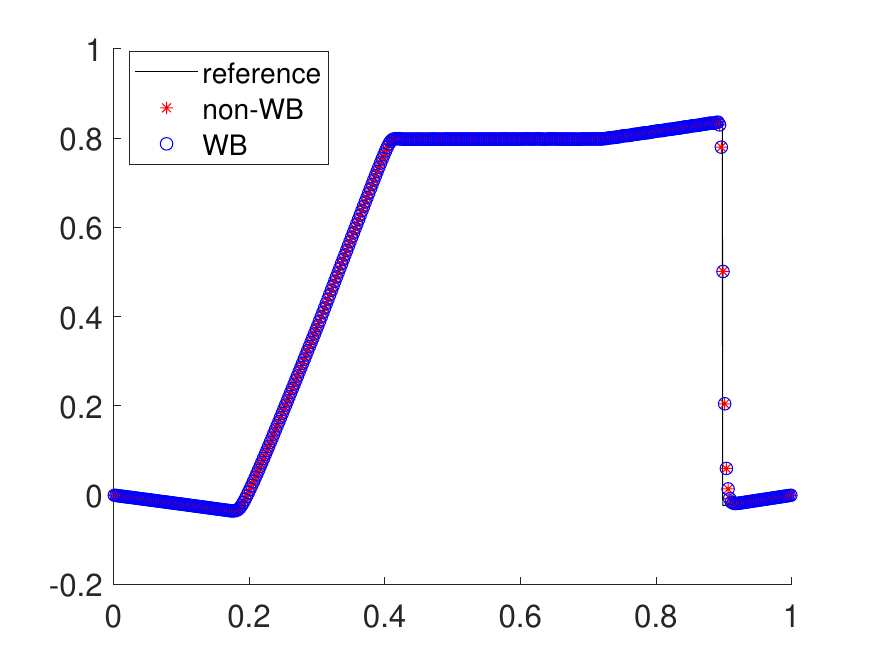}
\end{minipage}
}
\subfigure[$u_2$]{
\begin{minipage}[c]{0.3\linewidth}
\centering
\includegraphics[width=5cm]{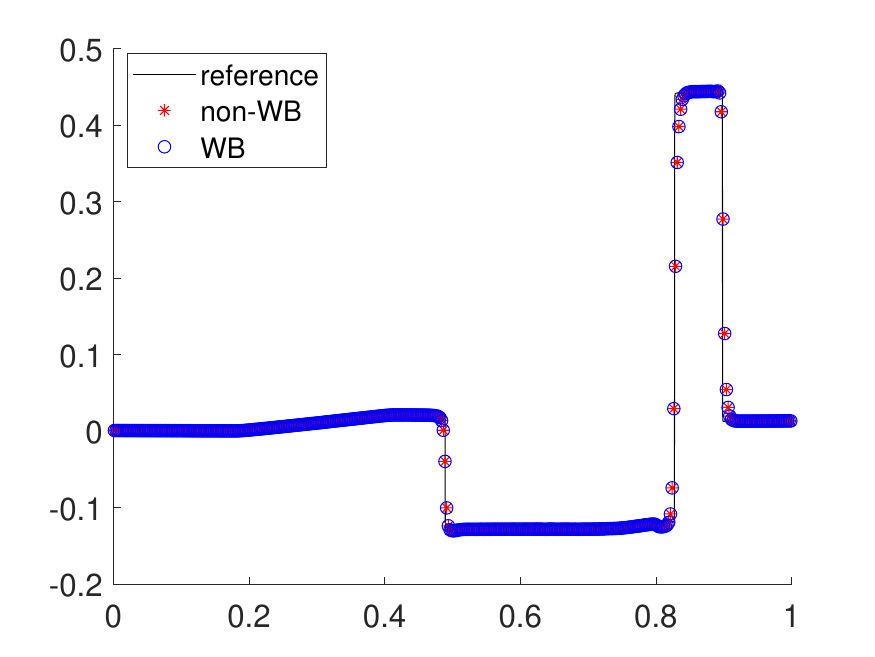}
\end{minipage}
}
\subfigure[$p_{11}$]{
\begin{minipage}[c]{0.3\linewidth}
\centering
\includegraphics[width=5cm]{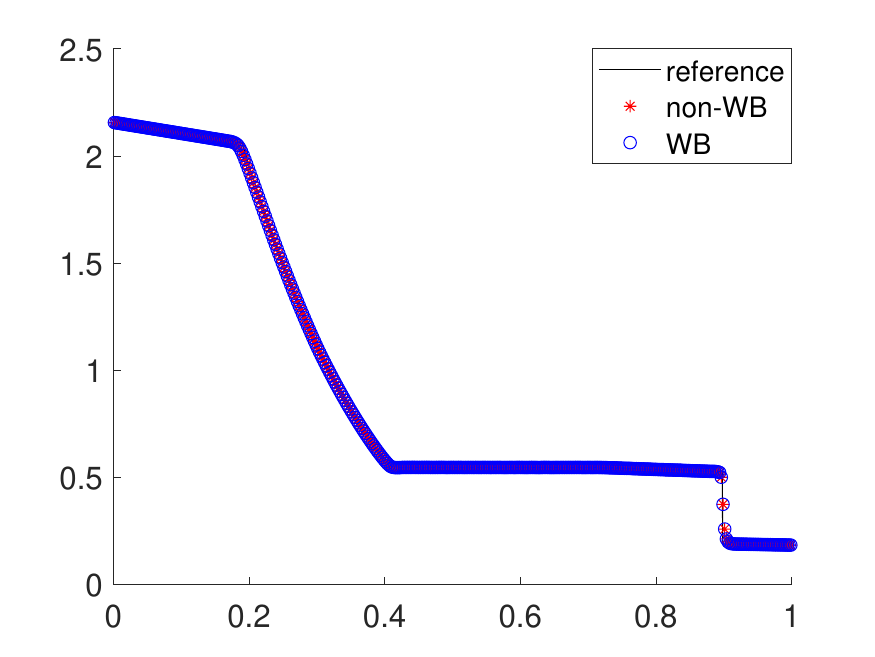}
\end{minipage}
}
\subfigure[$p_{12}$]{
\begin{minipage}[c]{0.3\linewidth}
\centering
\includegraphics[width=5cm]{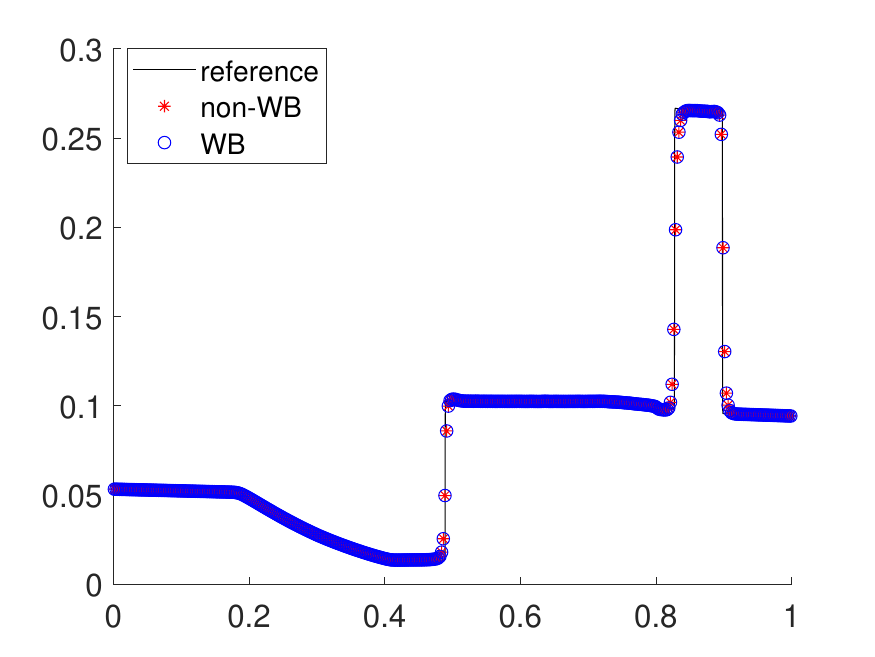}
\end{minipage}
}
\subfigure[$p_{22}$]{
\begin{minipage}[c]{0.3\linewidth}
\centering
\includegraphics[width=5cm]{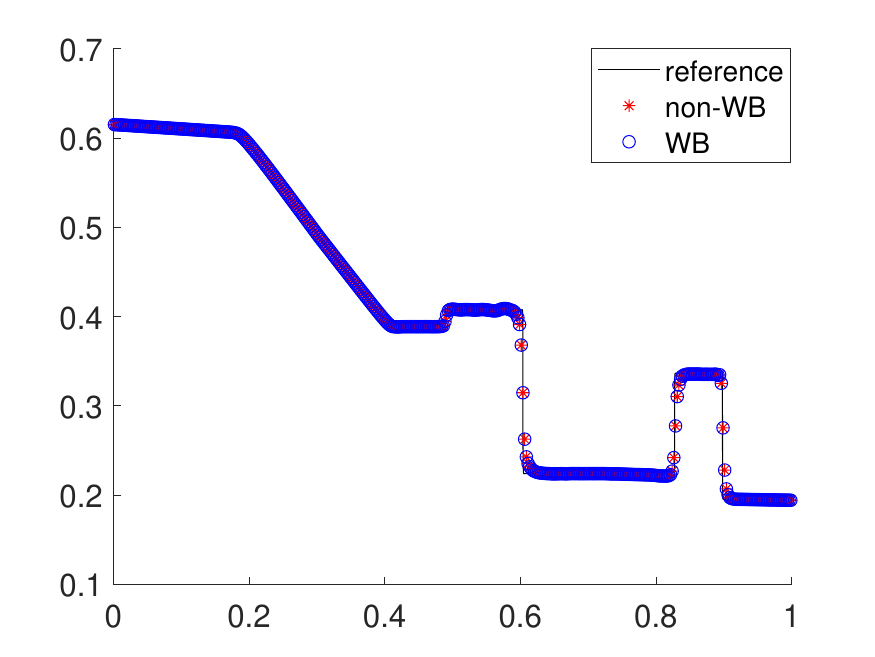}
\end{minipage}
}
\centering
\caption{Example 5: The numerical solutions of the third-order WB and non-WB schemes are obtained on 400 uniform cells. The reference solutions are obtained by the third-order WB DG scheme with 10,000 cells.}
\label{test11_nr}
\end{figure}

\subsection{Example 6: 1D near-vacuum test}
This is a demanding Riemann problem, which extends the 1D near-vacuum test in \cite{meena2017positivity,meena2020positivity},  used to check the positivity-preserving property of the proposed DG schemes.
The initial conditions are given by
\[
(\rho,u_1,u_2,p_{11},p_{12},p_{22})=\begin{cases}
                                      (\epsilon, -8\epsilon, 0, 2\epsilon, 0, 2\epsilon), & \text{if } x\leq0 \\
                                      (\epsilon, 8\epsilon, 0, 2\epsilon, 0, 2\epsilon), & \text{if } x>0
                                    \end{cases}
\]
with $\epsilon=10^{-5}$ over the domain $[-1,1]$. The potential function is taken as $W(x)=\frac{x^2}{2}$. The outflow boundary conditions are imposed. The exact solution describes the propagation of two rarefaction waves away from the center, which leads to very low density and pressure near the origin.

We apply the third-order WB DG scheme with the positivity-preserving limiter to simulate this problem up to $t=0.05$ on the uniform mesh of 400 cells. The numerical results are displayed in Figure \ref{1d_near_vacuum_nr1}, where the reference solution is obtained by the same scheme with 10,000 cells. Throughout the simulation, the observed minimum values of density $\rho$, pressure component $p_{11}$, and the determinant of pressure tensor $\det(\mathbf{p})$ are $9.3937\times10^{-9}$, $7.5100\times10^{-10}$, and $1.4635\times10^{-17}$, respectively. The proposed positivity-preserving WB DG scheme behaves robustly, even in the low density and low pressure region. It is noticed that the code would quickly break down, if the positivity-preserving limiter is not employed.

\begin{figure}[!htbp]
\subfigure[$\rho$]{
\begin{minipage}[c]{0.3\linewidth}
\centering
\includegraphics[width=5cm]{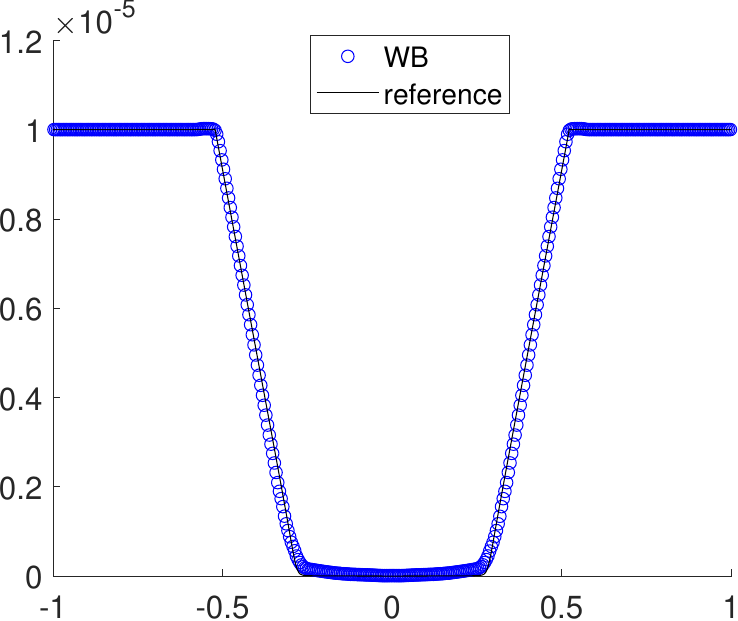}
\end{minipage}
}
\subfigure[$p_{11}$]{
\begin{minipage}[c]{0.3\linewidth}
\centering
\includegraphics[width=5cm]{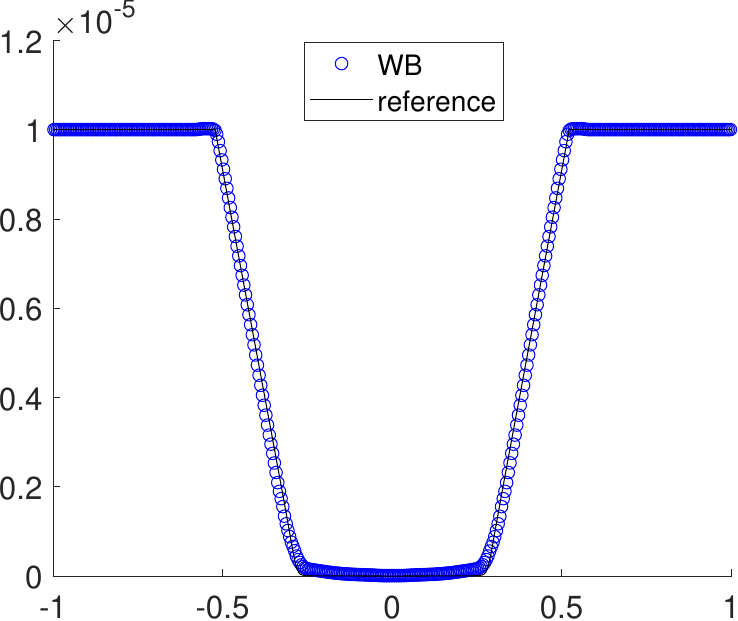}
\end{minipage}
}
\subfigure[$\det(\mathbf{p})$]{
\begin{minipage}[c]{0.3\linewidth}
\centering
\includegraphics[width=5cm]{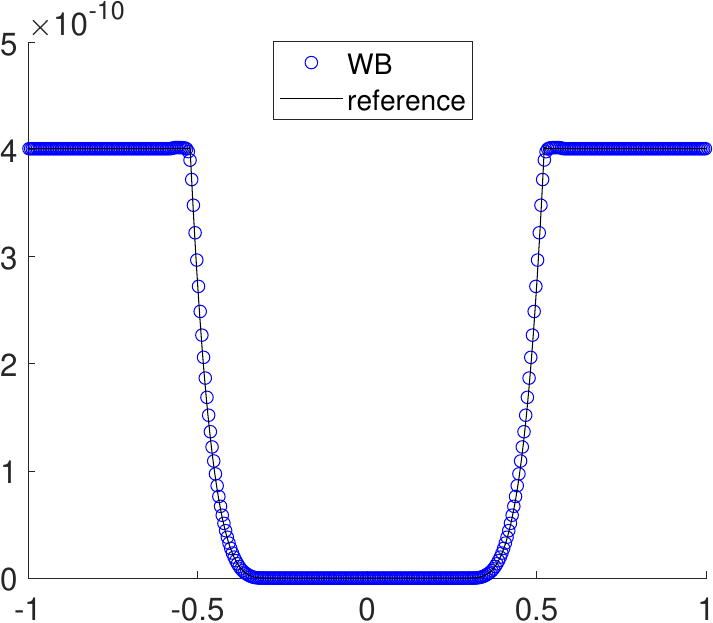}
\end{minipage}
}
\centering
\caption{Example 6: The numerical results obtained by the third-order positivity-preserving WB DG scheme for the 1D near vacuum test. The reference solutions are obtained by the same scheme with 10000 cells.}
\label{1d_near_vacuum_nr1}
\end{figure}

\subsection{Example 7: 2D WB test}
This example is used to examine the WB property of the proposed 2D schemes. The following three test cases \cite{meena2018well} are considered.
\begin{itemize}
\item Polytropic case: The hydrostatic equilibrium solution is given by
\begin{align*}
&\rho(x,y)=\rho_0\left(1+\frac{\nu-1}{\alpha\nu\rho_0^{\nu-1}}\left(\frac{W_0-W(x,y)}{2}\right)\right)^{\frac{1}{\nu-1}}, ~ p_{11}(x,y)=p_{11,0}\left(1+\frac{\nu-1}{\alpha\nu\rho_0^{\nu-1}}\left(\frac{W_0-W(x,y)}{2}\right)\right)^{\frac{\nu}{\nu-1}},
\\
&p_{12}=0, ~ p_{22}(x,y)=p_{22,0}\left(1+\frac{\nu-1}{\alpha\nu\rho_0^{\nu-1}}\left(\frac{W_0-W(x,y)}{2}\right)\right)^{\frac{\nu}{\nu-1}},
\end{align*}
where $\rho_0=1$, $p_{11,0}=1$, $p_{22,0}=1$, $\alpha=1$, $\nu=1.2$, and $W_0=0$.

\item Isentropic case: The hydrostatic equilibrium solution is given by
\begin{align*}
&\rho(x,y)=\rho_0\left(1+\frac{1}{3\alpha\rho_0^{2}}\left(W_0-W(x,y)\right)\right)^{\frac{1}{2}}, ~ p_{11}(x,y)=p_{11,0}\left(1+\frac{1}{3\alpha\rho_0^{2}}\left(W_0-W(x,y)\right)\right)^{\frac{3}{2}},
\\
&p_{12}=0, ~ p_{22}(x,y)=p_{22,0}\left(1+\frac{1}{3\alpha\rho_0^{2}}\left(W_0-W(x,y)\right)\right)^{\frac{3}{2}},
\end{align*}
where $\rho_0=1$, $p_{11,0}=1$, $p_{22,0}=1$, $\alpha=1$, and $W_0=0$.

\item Isothermal case: The hydrostatic equilibrium solution is given by
\begin{equation}\label{2D isothermal case}
\rho(x,y)=\rho_0\exp\left(\frac{-\rho_0W(x,y)}{2}\right), ~ p_{11}(x,y)=p_{22}(x,y)=\exp\left(\frac{-\rho_0W(x,y)}{2}\right), ~ p_{12}=0,
\end{equation}
where $\rho_0=1.21$.
\end{itemize}

Consider the computational domain $[0,1]^2$ and the potential $W(x,y)=x+y$. The third-order WB and non-WB DG schemes are employed to solve the above three test cases on the meshes consisting of $50\times50$ and $100\times100$ cells, up to the final time $t=1$. The positivity-preserving limiter is not turned on for these problems. The errors of density $\rho$ and the pressure components $p_{11}$ and $p_{22}$ for the three cases are presented in Tables \ref{2D_wb_nr1}--\ref{2D_wb_nr3}, respectively. One can see that the $l^1$ errors are at the level of machine precision for our proposed 2D WB scheme, thereby verifying its WB property. In contrast, the errors for the non-WB scheme are notably larger, particularly on coarser meshes.


\begin{table}[!htbp]
\centering
\caption{Example 7: The $l^1$ errors in $\rho$, $p_{11}$ and $p_{22}$ at $t=1$ for the polytropic test case.}
\label{2D_wb_nr1}
\begin{center}
\small
\begin{tabular}{c|c|ccc}
\hline
 Scheme & $N$ & $\rho$ & $p_{11}$ & $p_{22}$  \\
\hline
 \multirow{2}{*}{WB} & 50 & 9.1018e-16 & 8.1588e-16 & 7.9613e-16  \\
 \multirow{2}{*}{} & 100 & 1.2444e-15 & 1.6866e-15 & 1.6623e-15  \\
\hline
\multirow{2}{*}{non-WB} & 50 & 3.5622e-09 & 6.5502e-09 & 6.5502e-09  \\
\multirow{2}{*}{} & 100 & 4.4071e-10 & 8.1373e-10 & 8.1373e-10  \\
\hline
\end{tabular}
\end{center}
\end{table}

\begin{table}[!htbp]
\centering
\caption{Example 7: The $l^1$ errors in $\rho$, $p_{11}$ and $p_{22}$ at $t=1$ for the isentropic test case.}
\label{2D_wb_nr2}
\begin{center}
\small
\begin{tabular}{c|c|ccc}
\hline
 Scheme & $N$ & $\rho$ & $p_{11}$ & $p_{22}$  \\
\hline
 \multirow{2}{*}{WB} & 50 & 8.4933e-16 & 6.1084e-16 & 6.2504e-16  \\
 \multirow{2}{*}{} & 100 & 1.2542e-15 & 1.2212e-15 & 1.2639e-15  \\
\hline
\multirow{2}{*}{non-WB} & 50 & 8.2118e-09 & 3.6419e-09 & 3.6419e-09  \\
\multirow{2}{*}{} & 100 & 1.0394e-09 & 4.5836e-10 & 4.5836e-10  \\
\hline
\end{tabular}
\end{center}
\end{table}

\begin{table}[!htbp]
\centering
\caption{Example 7: The $l^1$ errors in $\rho$, $p_{11}$ and $p_{22}$ at $t=1$ for the isothermal test case.}
\label{2D_wb_nr3}
\begin{center}
\small
\begin{tabular}{c|c|ccc}
\hline
 Scheme & $N$ & $\rho$ & $p_{11}$ & $p_{22}$  \\
\hline
 \multirow{2}{*}{WB} & 50 & 9.9460e-16 & 7.7001e-16 & 7.9929e-16  \\
 \multirow{2}{*}{} & 100 & 1.8564e-15 & 1.7519e-15 & 1.7938e-15  \\
\hline
\multirow{2}{*}{non-WB} & 50 & 1.7205e-08 & 1.7033e-08 & 1.7033e-08  \\
\multirow{2}{*}{} & 100 & 2.1728e-09 & 2.1453e-09 & 2.1453e-09  \\
\hline
\end{tabular}
\end{center}
\end{table}

\subsection{Example 8: 2D small perturbation test}
To verify the ability of the proposed 2D WB scheme in capturing solutions close to the hydrostatic solution, we consider the isothermal steady state (\ref{2D isothermal case}) and impose a small Gaussian hump perturbation \cite{meena2018well} centered at $(0.3,0.3)$ to the pressure components $p_{11}$ and $p_{22}$, i.e., the initial profile is set as
\begin{align*}
&\rho(x,y)=\rho_0\exp\left(\frac{-\rho_0W(x,y)}{2}\right), \\
&p_{11}(x,y)=\exp\left(\frac{-\rho_0W(x,y)}{2}\right)+\epsilon\exp\left(\frac{-100((x-0.3)^2+(y-0.3)^2)}{2}\right), \\
&p_{22}(x,y)=\exp\left(\frac{-\rho_0W(x,y)}{2}\right)+\epsilon\exp\left(\frac{-100((x-0.3)^2+(y-0.3)^2)}{2}\right),
\end{align*}
and $u_1=0$, $u_2=0$, $p_{12}=0$. The parameters are set as $\rho_0=1.21$ and $\epsilon=10^{-7}$.

This problem is simulated until $t=0.15$ by using the third-order WB and non-WB DG schemes with $50\times50$ uniform cells. The transmissive boundary conditions are imposed. The positivity-preserving limiter is deactivated for this simulation. The contour plots of the density perturbation, the $\mathrm{trace}(\mathbf{p}):=p_{11}+p_{22}$ perturbation and the $\det(\mathbf{p})$ perturbation are presented in Figure \ref{2d_small_per_nr}.
It is observed that the WB DG scheme accurately resolves such small perturbations even on a relatively coarse mesh, while the non-WB one cannot capture it well.

%

\begin{figure}[!htbp]
\subfigure[WB scheme: $\rho$ perturbation]{
\begin{minipage}[c]{0.3\linewidth}
\centering
\includegraphics[width=5cm]{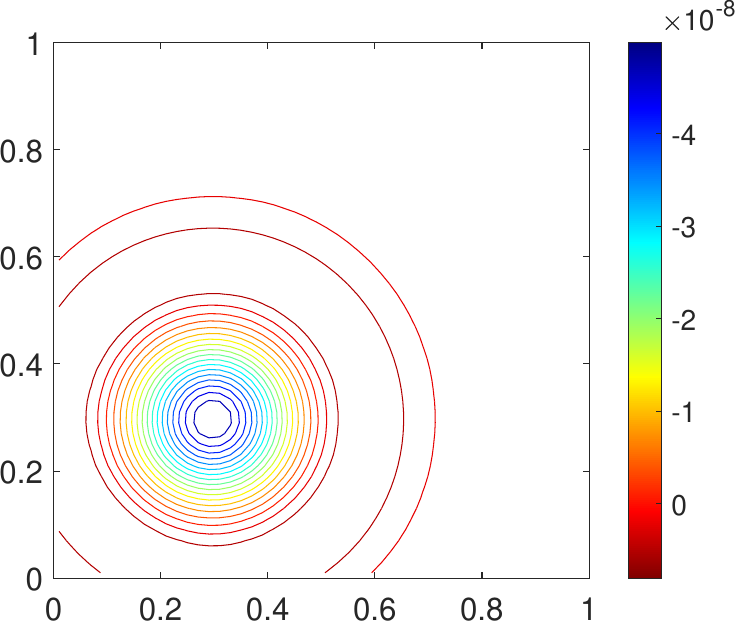}
\end{minipage}
}
\subfigure[WB scheme: ${\rm trace}(\mathbf{p})$ perturbation]{
\begin{minipage}[c]{0.3\linewidth}
\centering
\includegraphics[width=5cm]{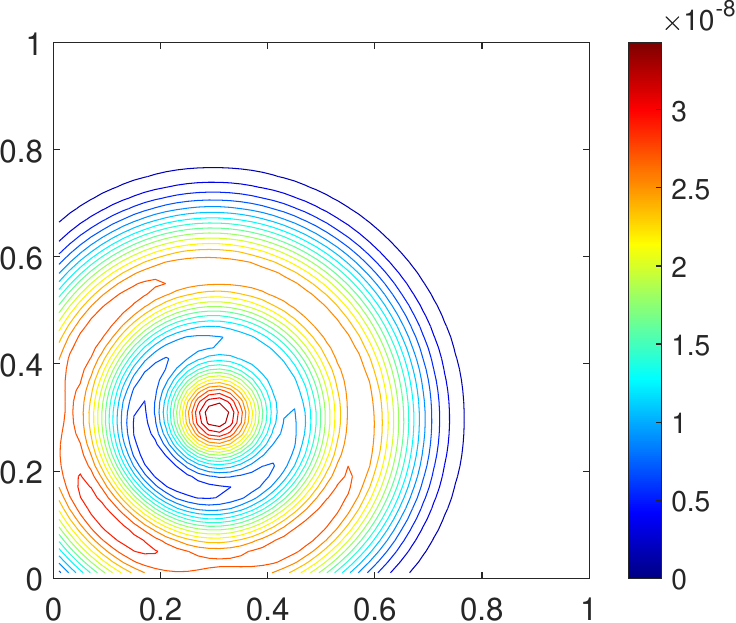}
\end{minipage}
}
\subfigure[WB scheme: $\det(\mathbf{p})$ perturbation]{
\begin{minipage}[c]{0.3\linewidth}
\centering
\includegraphics[width=5cm]{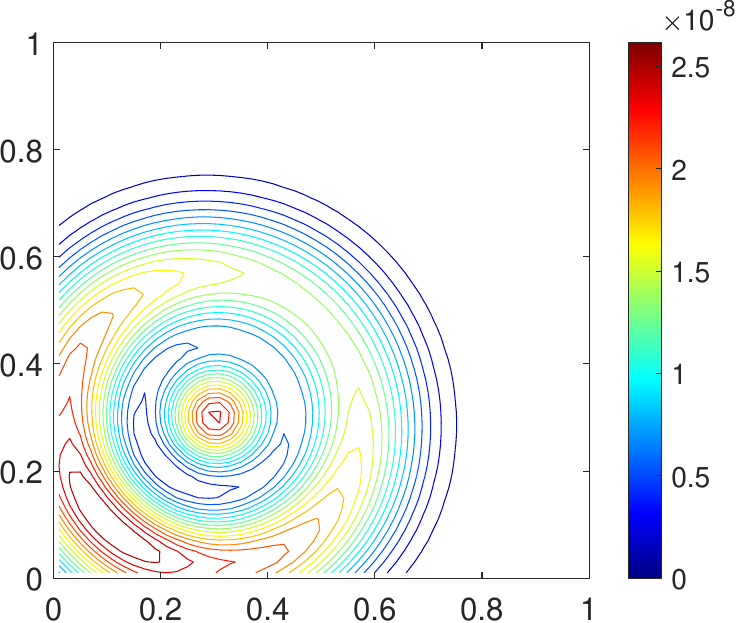}
\end{minipage}
}
\subfigure[non-WB scheme: $\rho$ perturbation]{
\begin{minipage}[c]{0.3\linewidth}
\centering
\includegraphics[width=5cm]{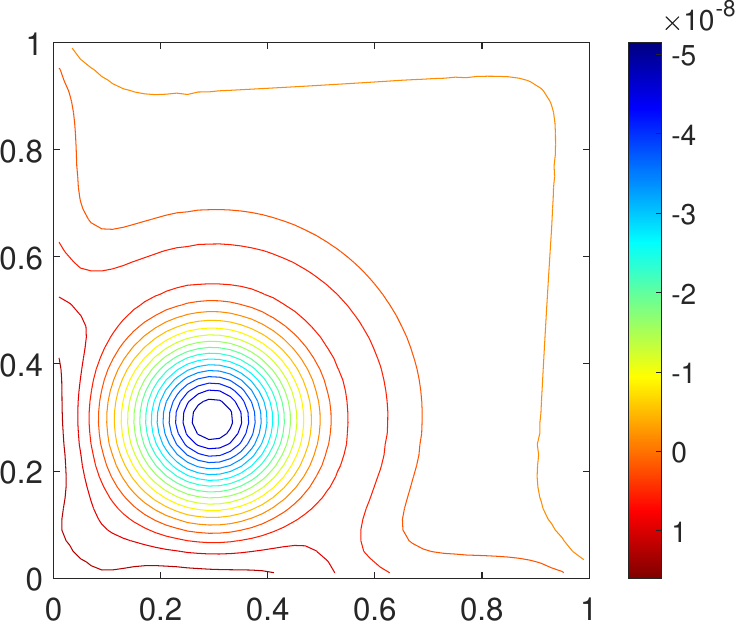}
\end{minipage}
}
\subfigure[non-WB scheme: ${\rm trace}(\mathbf{p})$ perturbation]{
\begin{minipage}[c]{0.3\linewidth}
\centering
\includegraphics[width=5cm]{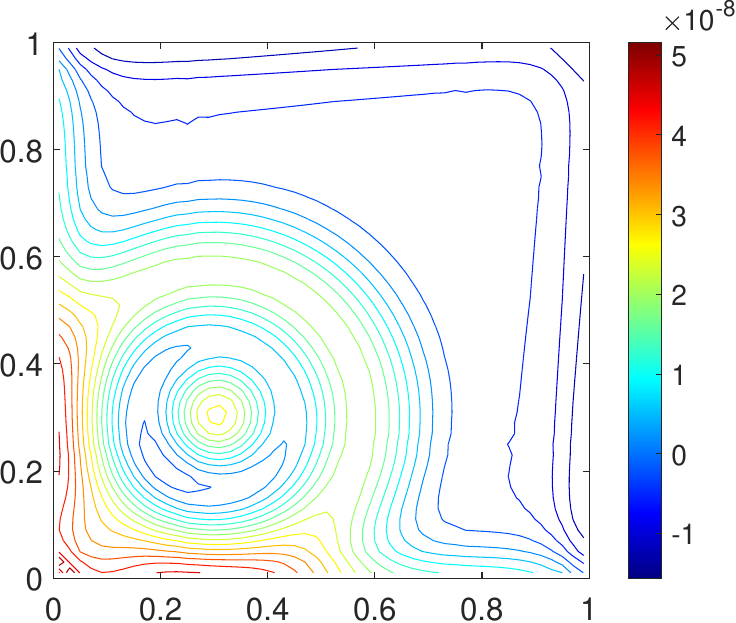}
\end{minipage}
}
\subfigure[non-WB scheme: $\det(\mathbf{p})$ perturbation]{
\begin{minipage}[c]{0.3\linewidth}
\centering
\includegraphics[width=5cm]{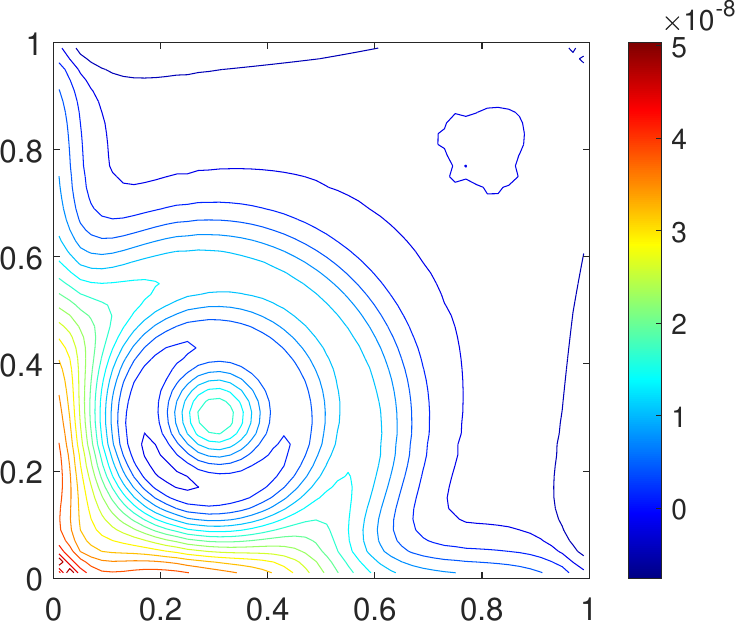}
\end{minipage}
}
\centering
\caption{Example 8: The contour plots of the perturbations of the isothermal hydrostatic solution at time $t=0.15$ obtained by the third-order WB and non-WB DG schemes with $50\times50$ cells. 20 equally spaced contour lines are displayed.}
\label{2d_small_per_nr}
\end{figure}

\subsection{Example 9: 2D perturbation test with low density and low pressure}
To validate the WB and positivity-preserving properties of the proposed 2D DG schemes, we consider an isothermal steady state with low density and low pressure, and impose a Gaussian hump perturbation centered at $(0.3,0.3)$ to the pressure components $p_{11}$ and $p_{22}$ over the domain $[0,1]^2$. The initial profile is as follows:
\begin{align*}
&\rho(x,y)=10^{-7}\rho_0\exp\left(\frac{-\rho_0W(x,y)}{2}\right), \\
&p_{11}(x,y)=10^{-7}\exp\left(\frac{-\rho_0W(x,y)}{2}\right)+\epsilon\exp\left(\frac{-100((x-0.3)^2+(y-0.3)^2)}{2}\right), \\
&p_{22}(x,y)=10^{-7}\exp\left(\frac{-\rho_0W(x,y)}{2}\right)+\epsilon\exp\left(\frac{-100((x-0.3)^2+(y-0.3)^2)}{2}\right),
\end{align*}
and $u_1=0$, $u_2=0$, $p_{12}=0$. The parameters are set as $\rho_0=1.21$ and $\epsilon=10^{-6}$. The potential function $W(x,y)=x+y$.

This problem is simulated up to $t=5$ by using the third-order WB and non-WB DG schemes with $50\times50$ uniform cells. The transmissive boundary conditions are imposed. The contour plots of the $\mathrm{trace}(\mathbf{p})$ perturbation at $t=1,2,5$ are presented in Figure \ref{2d_small_per_2_nr}.
It is observed that, compared to the non-WB DG scheme, the WB one can better capture the perturbation even for a long time. Since the density and pressure of this isothermal steady state are very low, a DG scheme without the positivity-preserving property may easily produce negative density or pressure. We observe that, at the time $t=5$, the maximum and minimum of ${\rm trace}(\mathbf{p})$ are about  $2.3416\times10^{-7}$ and $4.4152\times10^{-8}$, respectively.
The proposed positivity-preserving DG schemes are robust for this demanding problem. However, without the positivity-preserving limiter, the third-order DG code fails at the time $t=0.092186$.

\begin{figure}[!htbp]
\subfigure[WB scheme: $t=1$]{
\begin{minipage}[c]{0.3\linewidth}
\centering
\includegraphics[width=5cm]{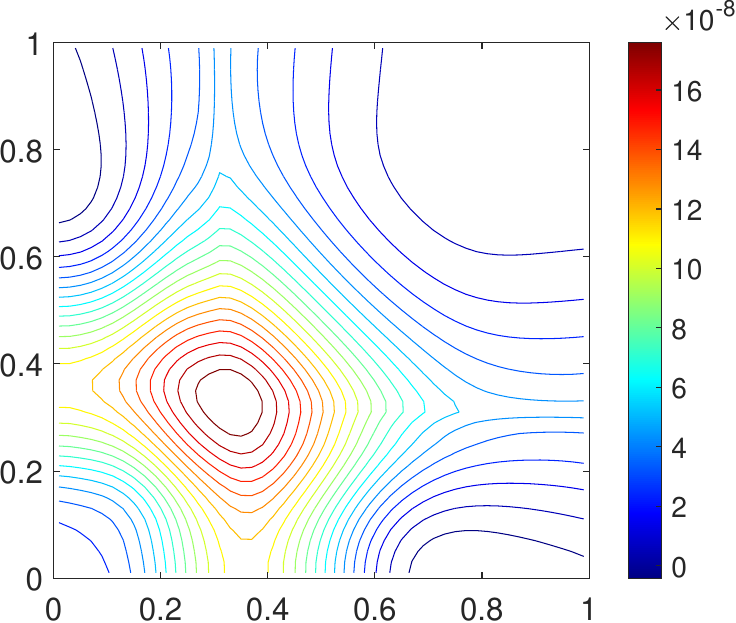}
\end{minipage}
}
\subfigure[WB scheme: $t=2$]{
\begin{minipage}[c]{0.3\linewidth}
\centering
\includegraphics[width=5cm]{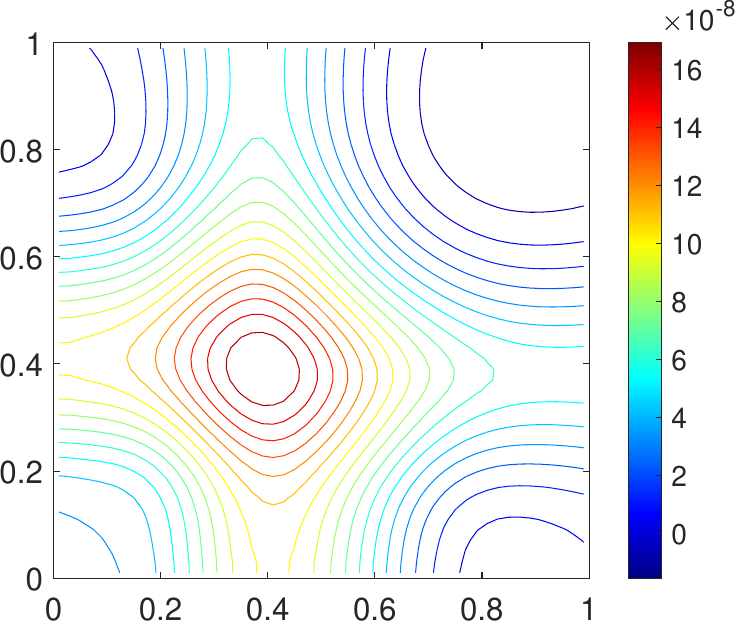}
\end{minipage}
}
\subfigure[WB scheme: $t=5$]{
\begin{minipage}[c]{0.3\linewidth}
\centering
\includegraphics[width=5cm]{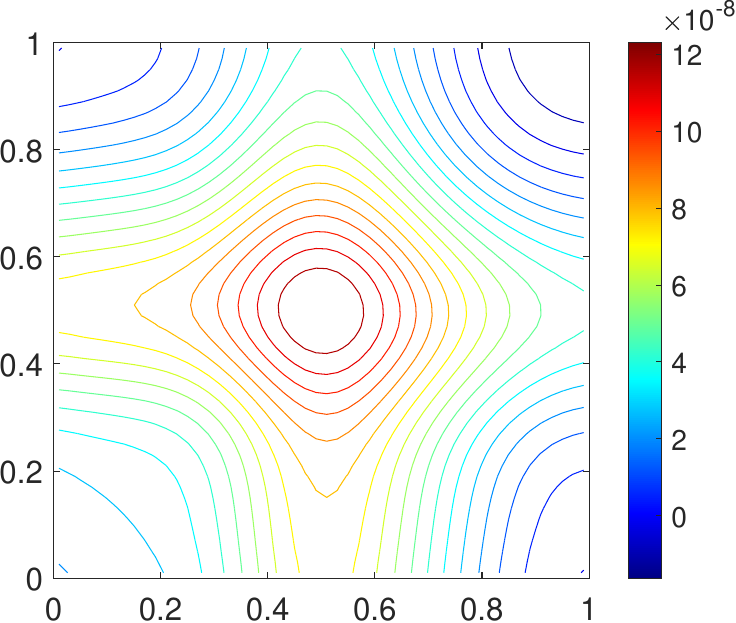}
\end{minipage}
}
\subfigure[non-WB scheme: $t=1$]{
\begin{minipage}[c]{0.3\linewidth}
\centering
\includegraphics[width=5cm]{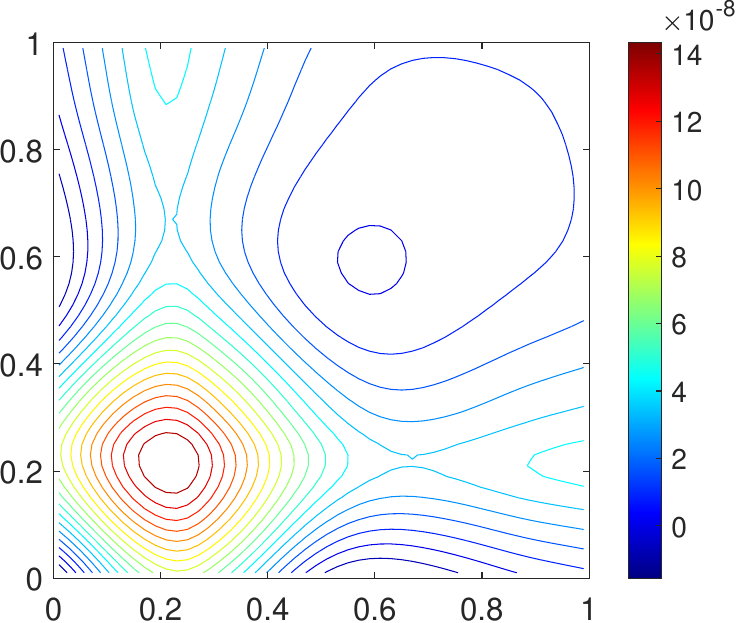}
\end{minipage}
}
\subfigure[non-WB scheme: $t=2$]{
\begin{minipage}[c]{0.3\linewidth}
\centering
\includegraphics[width=5cm]{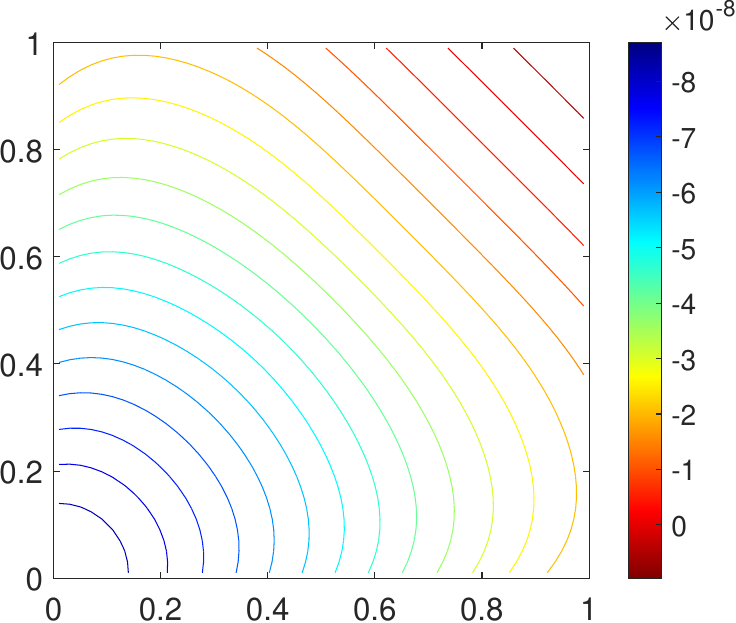}
\end{minipage}
}
\subfigure[non-WB scheme: $t=5$]{
\begin{minipage}[c]{0.3\linewidth}
\centering
\includegraphics[width=5cm]{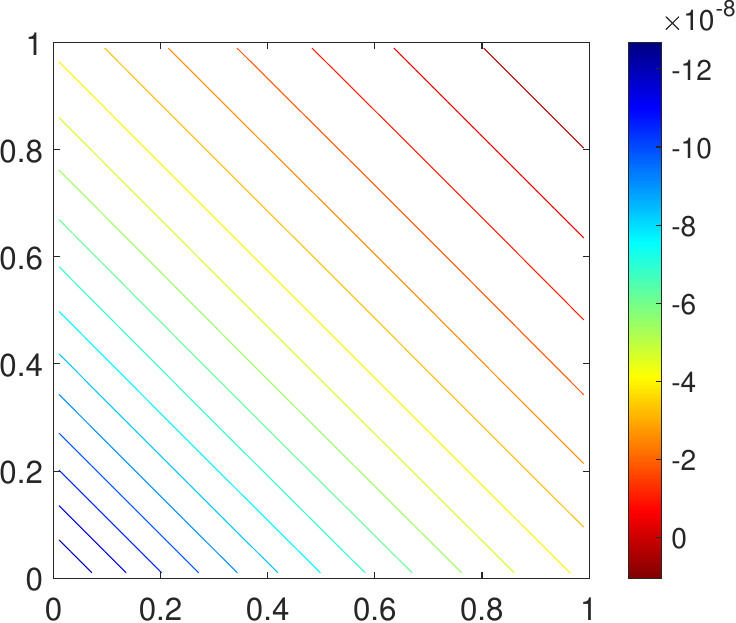}
\end{minipage}
}
\centering
\caption{Example 9: The contour plots of the ${\rm trace}(\mathbf{p})$ perturbations at time $t=1,2,5$ obtained by the third-order WB and non-WB DG schemes with $50\times50$ cells. 20 equally spaced contour lines are displayed.}
\label{2d_small_per_2_nr}
\end{figure}

\subsection{Example 10: 2D near-vacuum test}
%

To further demonstrate the positivity-preserving property of the proposed WB DG scheme, the 2D near-vacuum test \cite{meena2017positivity,meena2020positivity} is employed. The initial conditions are set to
\begin{equation*}
\rho = 1, \quad p_{11} = 2, \quad p_{12} = 0, \quad p_{22} = 2,
\end{equation*}
and a radially outward velocity field
\begin{equation*}
u_1 = 8\frac{x}{r}f(r,s), \quad u_2 = 8\frac{y}{r}f(r,s),
\end{equation*}
where $r = \sqrt{x^2 + y^2}$ and $s = \Delta x$. The smoothing function $f$ is defined by
\begin{equation*}
f(r,s) =
\begin{cases}
    -2\left(\frac{r}{s}\right)^3 + 3\left(\frac{r}{s}\right)^2, & \text{if } r < s, \\
    1, & \text{otherwise},
\end{cases}
\end{equation*}
which moderates the velocity profile in the vicinity of the origin. The potential function is given by $W(x,y) = \frac{1}{2}(x^2 + y^2)$. The strong outward velocity induces a continuous decline in both density and pressure near the center as time progresses.

The simulation is conducted using the third-order WB DG scheme until the final time $t = 0.05$ across the domain $[-2,2]^2$ with outflow boundary conditions. A grid consisting of $151 \times 151$ cells is utilized. Contour plots of the numerical solution are depicted in Figure \ref{2d_near_vacuum_nr4}. At the final time, the minimum values recorded for density, the pressure component $p_{11}$, and $\det(\mathbf{p})$ are approximately $1.7492 \times 10^{-3}$, $3.0958 \times 10^{-4}$, and $1.0378 \times 10^{-7}$, respectively. It is noteworthy that disabling the positivity-preserving limiter results in the failure of the DG method at $t = 0.000761$.

\begin{figure}[!htbp]
\subfigure[$\rho$]{
\begin{minipage}[c]{0.3\linewidth}
\centering
\includegraphics[width=5cm]{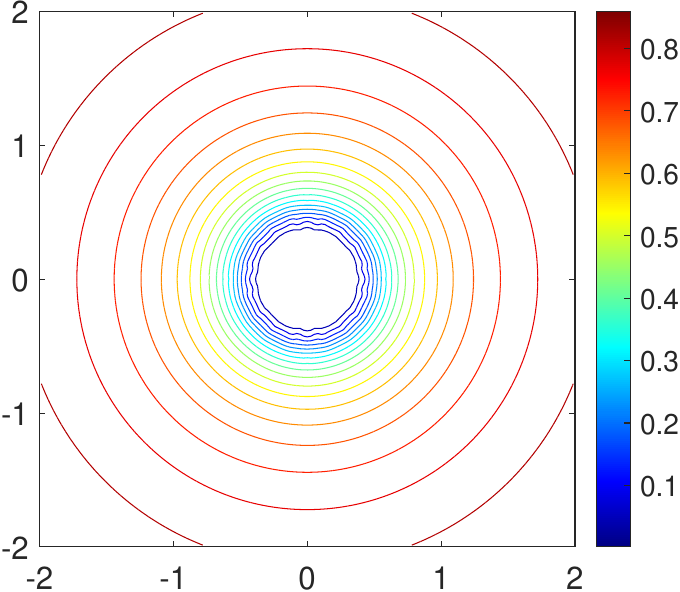}
\end{minipage}
}
\subfigure[$p_{11}$]{
\begin{minipage}[c]{0.3\linewidth}
\centering
\includegraphics[width=5cm]{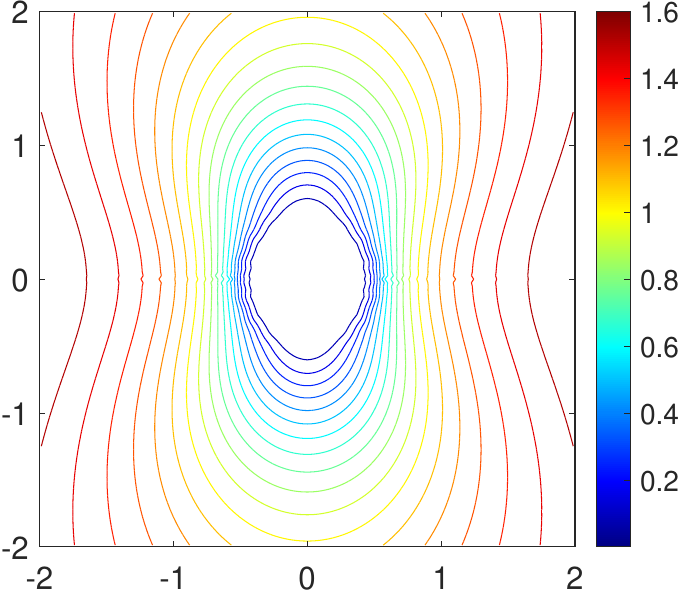}
\end{minipage}
}
\subfigure[$p_{12}$]{
\begin{minipage}[c]{0.3\linewidth}
\centering
\includegraphics[width=5cm]{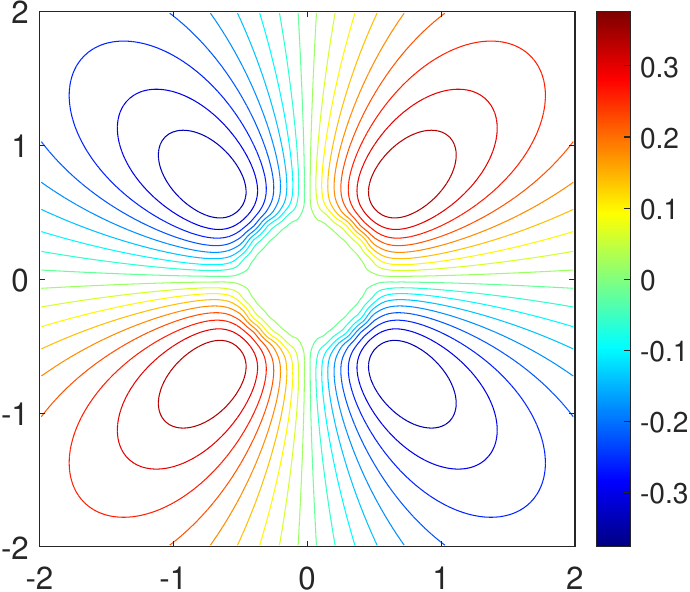}
\end{minipage}
}
\subfigure[$p_{22}$]{
\begin{minipage}[c]{0.3\linewidth}
\centering
\includegraphics[width=5cm]{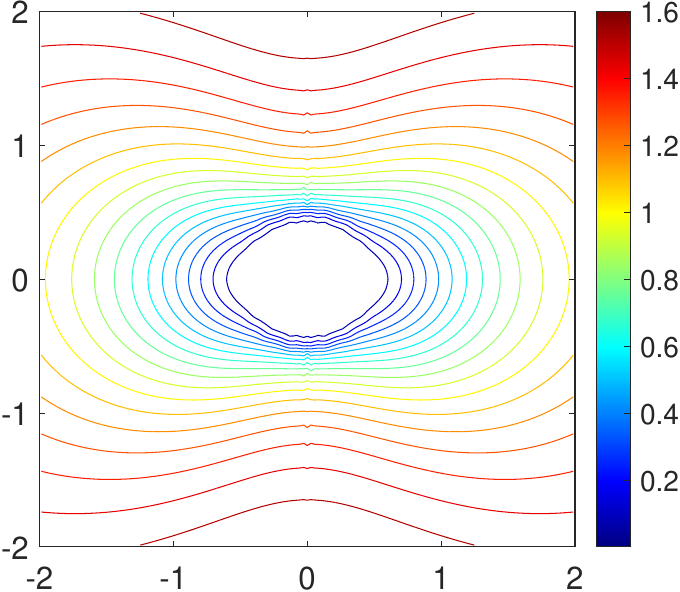}
\end{minipage}
}
\subfigure[${\rm trace}(\mathbf{p})$]{
\begin{minipage}[c]{0.3\linewidth}
\centering
\includegraphics[width=5cm]{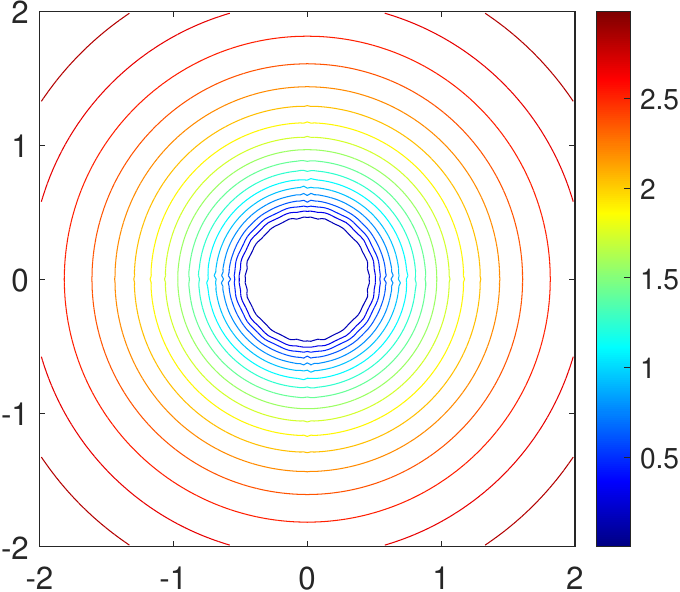}
\end{minipage}
}
\subfigure[$\det(\mathbf{p})$]{
\begin{minipage}[c]{0.3\linewidth}
\centering
\includegraphics[width=5cm]{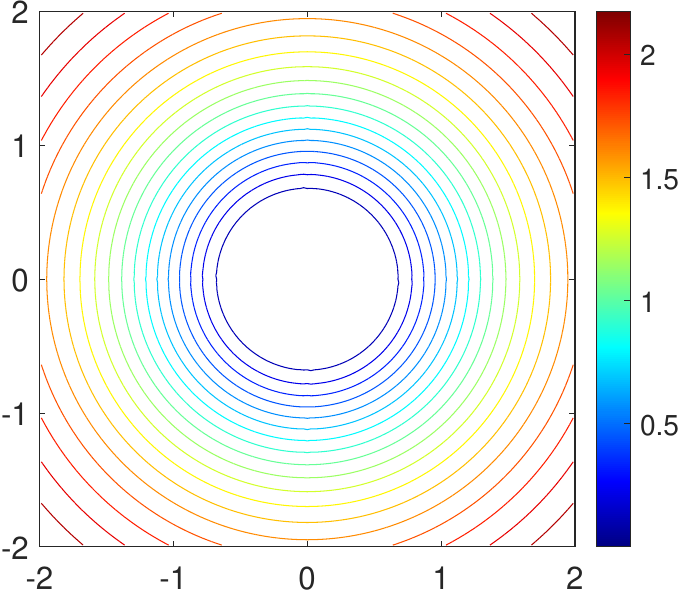}
\end{minipage}
}
\centering
\caption{Example 10: The contour plots of the 2D near vacuum test at time $t=0.05$ obtained by the third-order positivity-preserving and WB DG scheme with $151\times151$ cells. 20 equally spaced contour lines are displayed.}
\label{2d_near_vacuum_nr4}
\end{figure}

\subsection{Example 11: Uniform plasma state with 2D Gaussian source}

This test evaluates the influence of a Gaussian source term on a 2D plasma model \cite{meena2017positivity,sen2018entropy,meena2020positivity}. The plasma is initially in a uniform state given by
\begin{equation*}
(\rho, u_1, u_2, p_{11}, p_{12}, p_{22}) = (0.1, 0, 0, 9, 7, 9),
\end{equation*}
with the potential specified as
\begin{equation*}
W(x,y) = 25\exp\left(-200\left((x - 2)^2 + (y - 2)^2\right)\right)
\end{equation*}
over the spatial domain $[0, 4]^2$. Figure \ref{uniform_plasma_nr} presents the numerical results at $t = 0.1$, obtained by applying the third-order and fourth-order WB DG methods on a grid consisting of $100 \times 100$ cells. The simulations are conducted without the use of the positivity-preserving limiter.
Figures \ref{p2_rho} and \ref{p3_rho} show the anisotropic changes in density due to the Gaussian source's influence. Furthermore, a comparison of Figures \ref{p2_tracep}-\ref{p2_detp} with Figures \ref{p3_tracep}-\ref{p3_detp} reveals that the fourth-order scheme achieves greater accuracy than the third-order one.

\begin{figure}[!htbp]
\subfigure[$\rho$]{
\begin{minipage}[c]{0.3\linewidth}
\centering
\includegraphics[width=5.1cm]{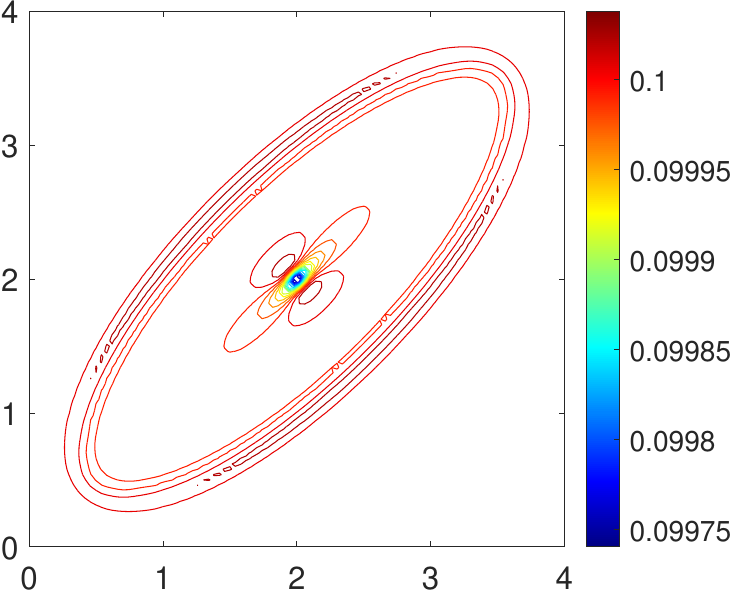}
\end{minipage}
\label{p2_rho}
}
\subfigure[${\rm trace}(\mathbf{p})$]{
\begin{minipage}[c]{0.3\linewidth}
\centering
\includegraphics[width=5cm]{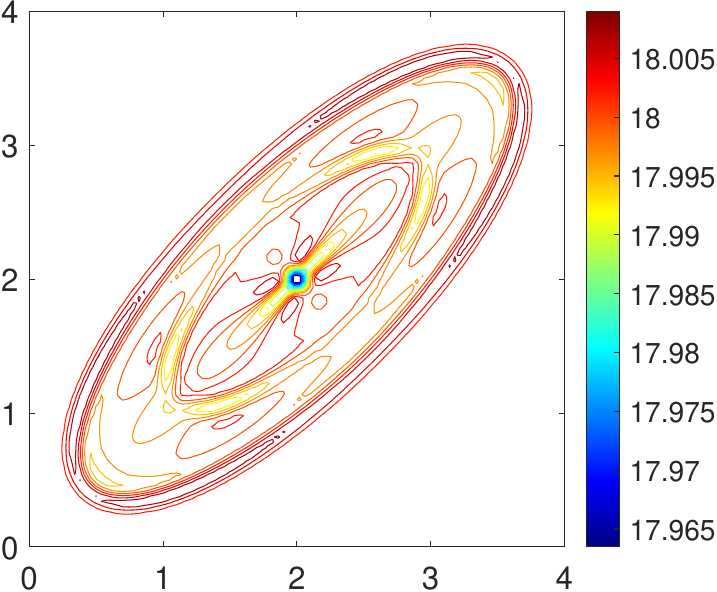}
\end{minipage}
\label{p2_tracep}
}
\subfigure[$\det(\mathbf{p})$]{
\begin{minipage}[c]{0.3\linewidth}
\centering
\includegraphics[width=4.8cm]{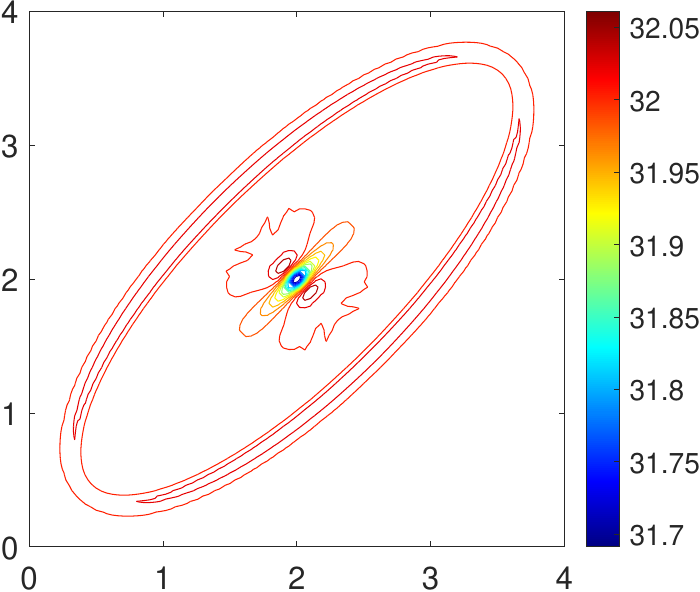}
\end{minipage}
\label{p2_detp}
}
\subfigure[$\rho$]{
\begin{minipage}[c]{0.3\linewidth}
\centering
\includegraphics[width=4.99cm]{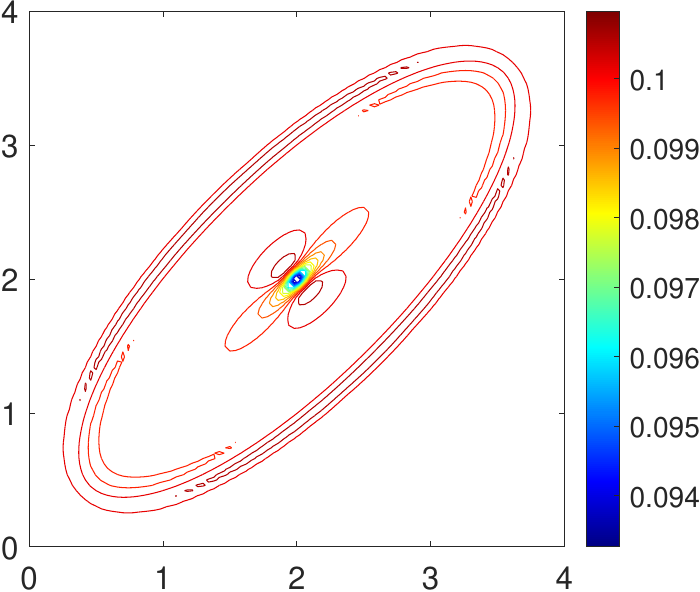}
\label{p3_rho}
\end{minipage}
}
\subfigure[${\rm trace}(\mathbf{p})$]{
\begin{minipage}[c]{0.3\linewidth}
\centering
\includegraphics[width=5cm]{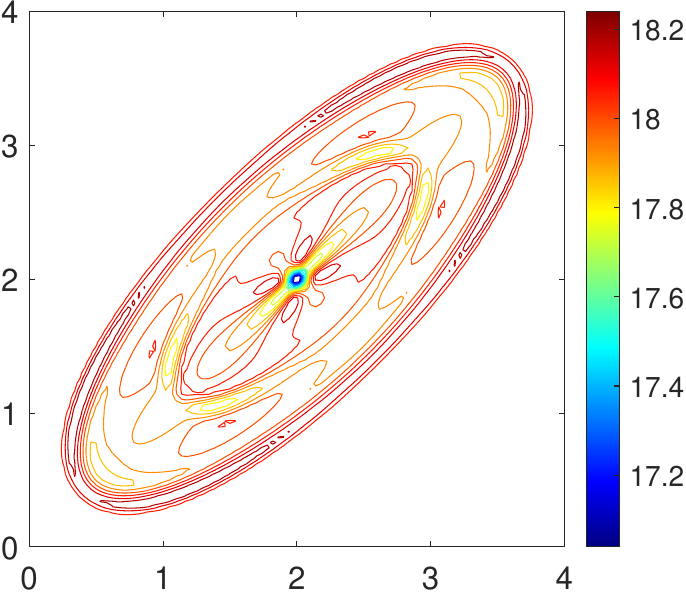}
\end{minipage}
\label{p3_tracep}
}
\subfigure[$\det(\mathbf{p})$]{
\begin{minipage}[c]{0.3\linewidth}
\centering
\includegraphics[width=4.8cm]{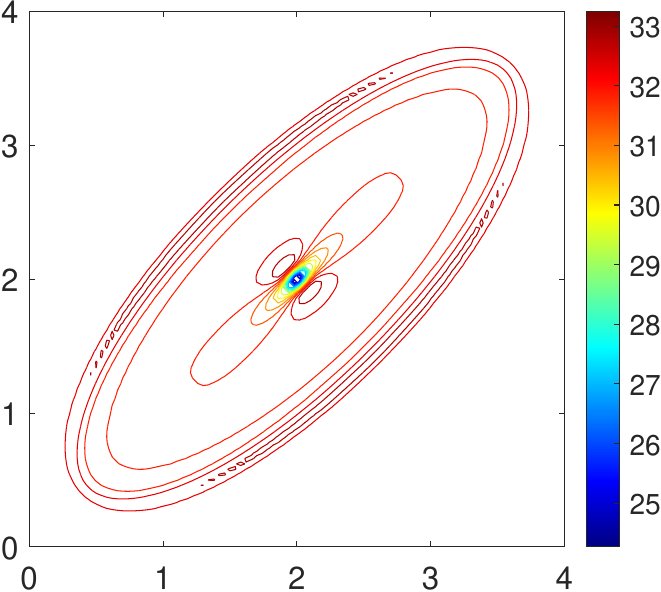}
\end{minipage}
\label{p3_detp}
}
\centering
\caption{Example 11: The contour plots of $\rho$ of the uniform plasma state with Gaussian source at time $t=0.1$ obtained by the third-order (the first row) and fourth-order (the second row) WB DG schemes with $100\times100$ cells. 20 equally spaced contour lines are displayed.}
\label{uniform_plasma_nr}
\end{figure}

\subsection{Example 12: Realistic simulation in two diemnsions}

In our final example \cite{berthon2015entropy,meena2017positivity,meena2020positivity}, we examine a plasma state within the domain $[0,100]^2$, initially defined by
\begin{equation*}
(\rho, u_1, u_2, p_{11}, p_{12}, p_{22}) = (0.109885, 0, 0, 1, 0, 1).
\end{equation*}
The plasma is subject to a source term with potential
\begin{equation*}
W(x,y) = \exp\left(\frac{-(x-50)^2 + (y-50)^2}{100}\right),
\end{equation*}
which exerts an influence solely in the $x$-direction, with the source term in the $y$-direction, $\mathbf{S}^y(\mathbf{U})$, being zero. Outflow boundary conditions are implemented on all edges of the domain.

This problem setup was originally designed to study the effects of inverse Bremsstrahlung absorption (IBA) \cite{berthon2015entropy}; more details on IBA can be found in \cite{firouzi2020inverse,turnbull2023inverse}. To simulate the IBA in an anisotropic plasma, we augment the energy equation for component $E_{11}$ with an additional source term, $v_T\rho W$, where $v_T$ denotes the absorption coefficient. We consider three scenarios with $v_T$ values of 0, 0.5, and 1.

Employing the third-order WB DG scheme, we simulate the problem up to $t = 0.5$ using a grid consisting of $200 \times 200$ cells. The positivity-preserving limiter is not activated for this simulation.
Figure \ref{realistic_simulation_nr1} showcases contour plots of $\rho$, ${\rm trace}(\mathbf{p})$, and $\det(\mathbf{p})$ for $v_T$ values of 0 and 1. Figure \ref{realistic_simulation_nr2} illustrates the 1D profiles of $\rho$ and $p_{11}$ along the line $y = 50$. An increase in the absorption coefficient, $v_T$, is observed to raise the pressure component $p_{11}$ around the center. This, in turn, drives a more pronounced expulsion of particles from the region, leading to a reduction in density near the center. These observations are consistent with the results documented in prior research \cite{meena2017positivity,sen2018entropy,meena2020positivity}.

\begin{figure}[!htbp]
\subfigure[$\rho$]{
\begin{minipage}[c]{0.3\linewidth}
\centering
\includegraphics[width=5.3cm]{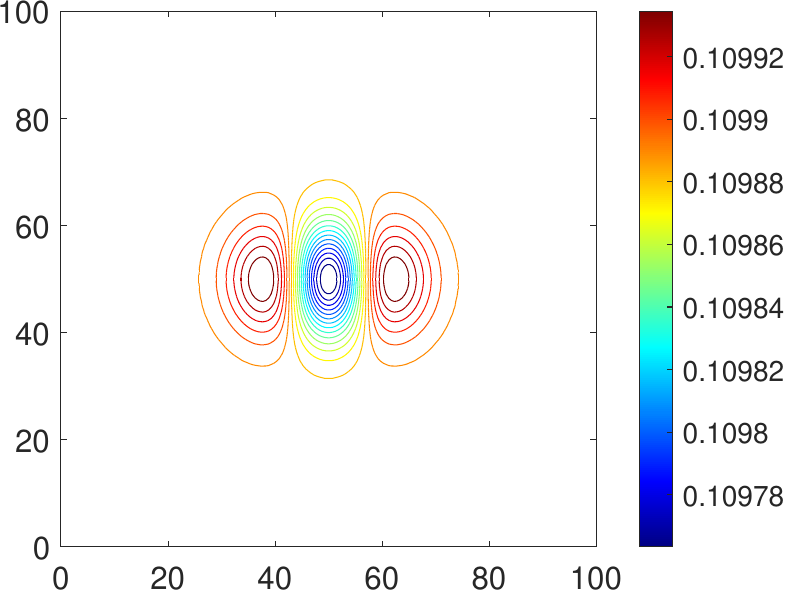}
\end{minipage}
}
\subfigure[${\rm trace}(\mathbf{p})$]{
\begin{minipage}[c]{0.3\linewidth}
\centering
\includegraphics[width=5cm]{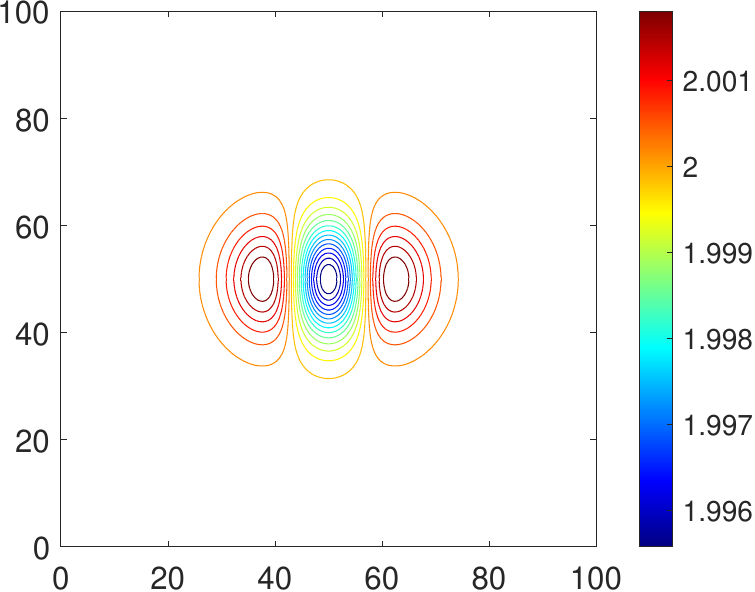}
\end{minipage}
}
\subfigure[$\det(\mathbf{p})$]{
\begin{minipage}[c]{0.3\linewidth}
\centering
\includegraphics[width=5cm]{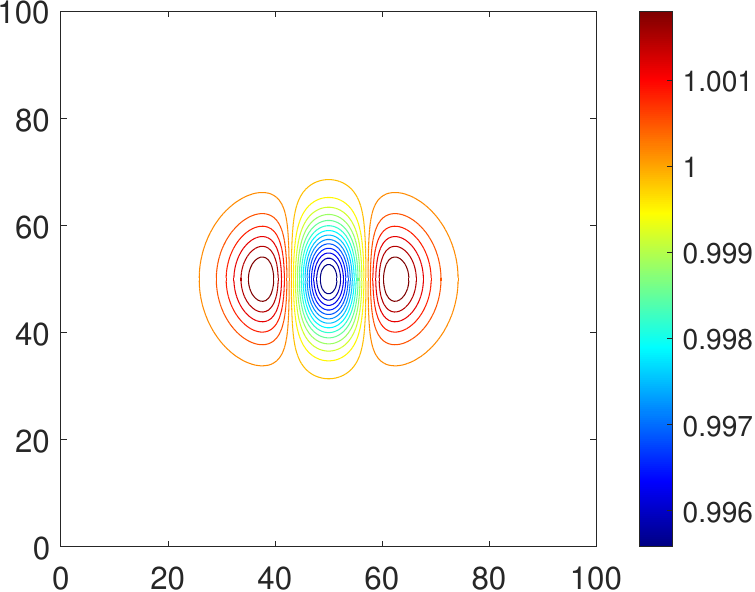}
\end{minipage}
}
\subfigure[$\rho$]{
\begin{minipage}[c]{0.3\linewidth}
\centering
\includegraphics[width=5.3cm]{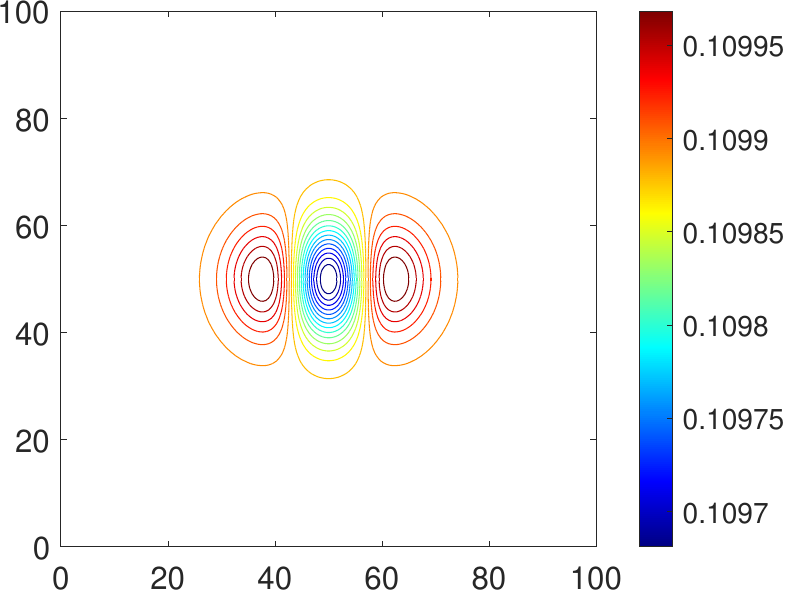}
\end{minipage}
}
\subfigure[${\rm trace}(\mathbf{p})$]{
\begin{minipage}[c]{0.3\linewidth}
\centering
\includegraphics[width=5cm]{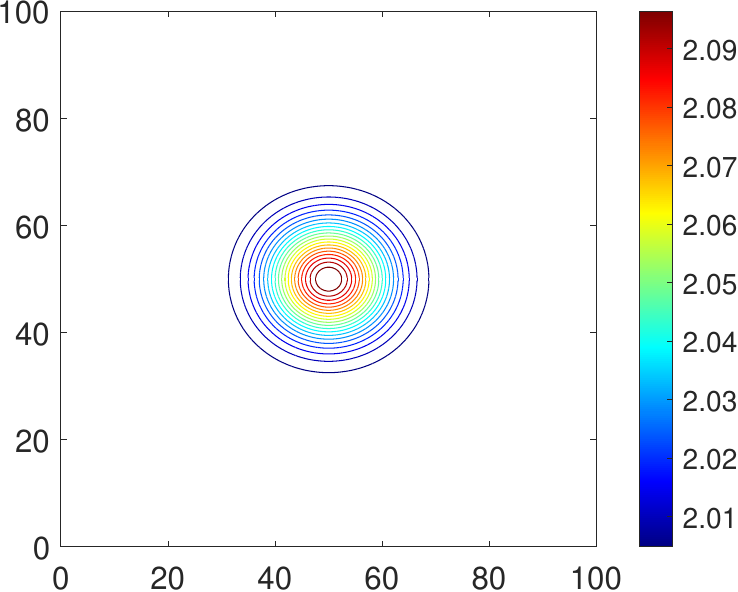}
\end{minipage}
}
\subfigure[$\det(\mathbf{p})$]{
\begin{minipage}[c]{0.3\linewidth}
\centering
\includegraphics[width=5cm]{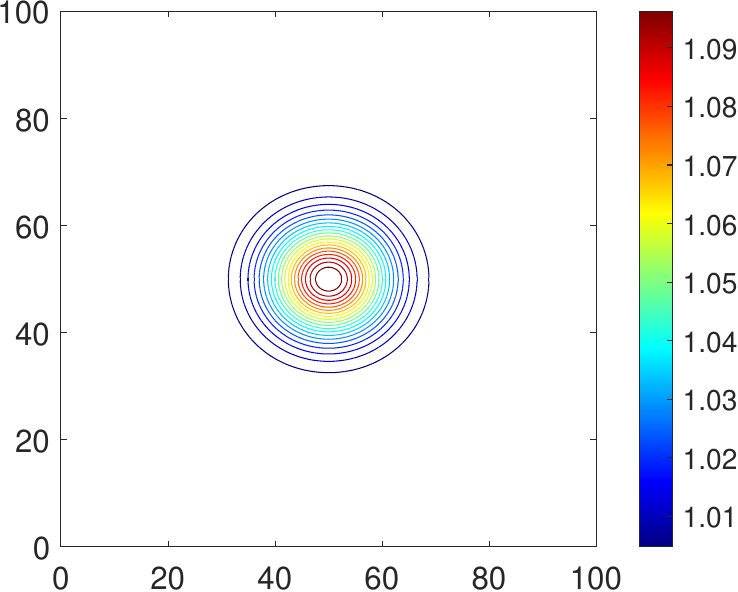}
\end{minipage}
}
\centering
\caption{Example 12: The contour plots of the realistic simulation at time $t=0.5$ obtained by the third-order WB DG scheme with $200\times200$ cells. The first and second row correspond to $v_T=0$ and $v_T=1$, respectively. 20 equally spaced contour lines are displayed.}
\label{realistic_simulation_nr1}
\end{figure}

\begin{figure}[!htbp]
\subfigure[$\rho$]{
\begin{minipage}[c]{0.3\linewidth}
\centering
\includegraphics[width=5.3cm]{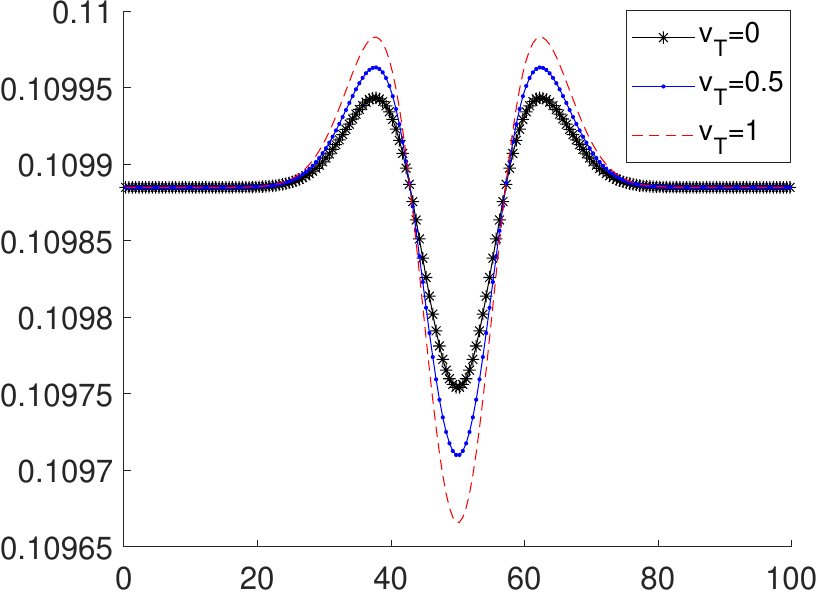}
\end{minipage}
}
\subfigure[$p_{11}$]{
\begin{minipage}[c]{0.3\linewidth}
\centering
\includegraphics[width=5cm]{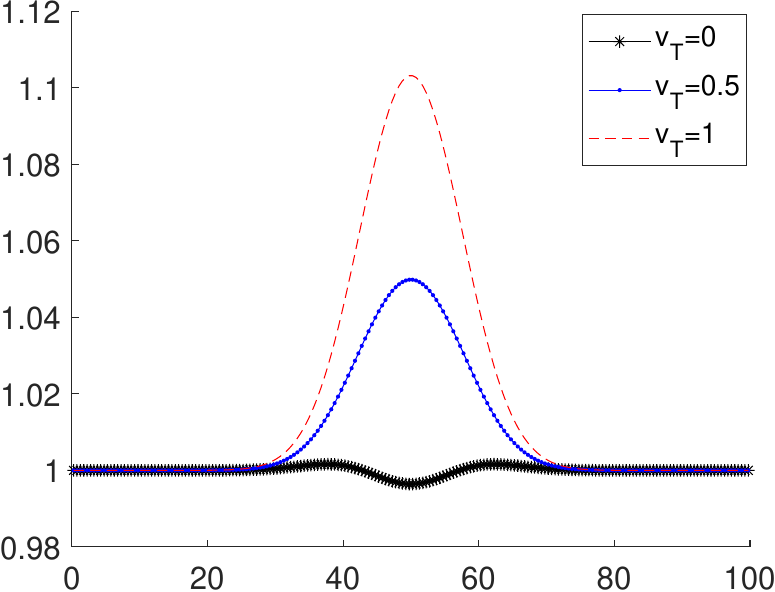}
\end{minipage}
}
\centering
\caption{Example 12: Comparison of $\rho$ and $p_{11}$ for different absorption coefficient $v_T=0$, $v_T=0.5$ and $v_T=1$ along the line $y=50$.}
\label{realistic_simulation_nr2}
\end{figure}


\section{Conclusion}\label{sec6}
This paper developed high-order accurate, well-balanced (WB), and positivity-preserving discontinuous Galerkin (DG) schemes for the one- and two-dimensional ten-moment Gaussian closure equations with source terms defined by a given potential. Our schemes were proven to maintain balance in known hydrostatic equilibrium states while ensuring the positivity of density and the positive-definiteness of the anisotropic pressure tensor.
The anisotropic effects posed new difficulties in this study, rendering the existing WB modification techniques designed for isotropic cases inapplicable for the ten-moment system. To address this, we introduced a novel modification to the solution states in the Harten--Lax--van Leer--contact (HLLC) flux, which, along with suitable discretization of the source terms, gave a new WB DG discretization.
We carried out the positivity-preserving analyses of our WB DG schemes, based on several key properties of the admissible state set, the HLLC flux and the HLLC solver, as well as the geometric quasilinearization (GQL) technique. The analyses proved a weak positivity for the cell averages of the DG solutions, so that a simple limiter effectively enforced the physical admissibility of the DG solution polynomials at certain points of interest.
 Extensive 1D and 2D numerical tests were conducted to demonstrate the accuracy, well-balancedness, positivity-preserving property, and high resolution of our proposed schemes.

 \section*{Acknowledgements}
 The works of J.~Wang and H.~Tang were partially supported by the National Key R\&D Program of China (Project Number
2020YFA0712000), the National Natural Science Foundation of China (Nos.~12171227 \& 12288101).
The work of K.~Wu was partially supported by Shenzhen Science and Technology Program
(No.~RCJC20221008092757098) and National Natural Science Foundation of China (Nos.~12171227 \& 92370108).

\bibliographystyle{siamplain}

\bibliography{Ten_Moment_PP_WB}
\end{document}